\documentclass[12pt,a4paper,notitlepage]{article}
\usepackage{amsthm}
\usepackage{amsfonts}
\usepackage{amsmath}
\usepackage{mathtools}

\usepackage{graphicx}

\usepackage{pdfpages}

\usepackage{cite}
\usepackage{indentfirst}
\usepackage{color}
\usepackage{hyperref}
\hypersetup{
   colorlinks,
   citecolor=black,
   filecolor=black,
   linkcolor=black,
   urlcolor=black
}

\usepackage{blindtext}
\usepackage[utf8]{inputenc}

\newtheorem{thm}{Theorem}[section]
\newtheorem*{thm*}{Theorem}
\newtheorem{lem}[thm]{Lemma}
\newtheorem{prop}[thm]{Proposition}
\newtheorem{cor}[thm]{Corollary}

\theoremstyle{definition}
\newtheorem{defn}[thm]{Definition}

\newtheorem{exmp}[thm]{Example}

\newtheorem{rmk}[thm]{Remark}

\usepackage{amssymb}

\newcommand{\sgn}{\text{sgn}}

\newcommand{\RN}[1]{%
 \textup{\uppercase\expandafter{\romannumeral#1}}%
}

\newcommand{\Addresses}{{
 \bigskip
 \footnotesize
 \textsc{Department of Mathematics, Eberhard Karls University of T\"{u}bingen,
   Auf der Morgenstelle 10, 72076 T\"{u}bingen, Germany}\par\nopagebreak
 \textit{E-mail address}: \texttt{jonathan.wilson@mnf.uni-tuebingen.de} }}
 
 \title{Laurent phenomenon algebras arising from surfaces II: Laminated surfaces}
\author{Jon Wilson}
\date{}

\usepackage{fancyhdr}

\fancypagestyle{plain}{
 \fancyhf{}
 \fancyfoot[C]{\thepage}%
}
\fancypagestyle{firstpage}{
 \fancyhf{}
 \fancyfoot[L]{\footnotesize{The author's PhD studies were funded by an EPSRC studentship.}}%
}

\usepackage[titles]{tocloft}

\usepackage{cite}

\usepackage{xcolor}
\newcommand*\circled[1]{\kern-2.5em%
  \put(0,4){\color{white}\circle*{18}}\put(-0.5,4){\circle{12}}%
  \put(-3,0){\color{black}\small#1}~~}
\usepackage{enumitem}

\usepackage{float}

\usepackage{stackengine}

\begin{document}

\maketitle

\thispagestyle{firstpage}

\begin{abstract}

It was shown by Fock and Goncharov \cite{fock2007dual}, and Fomin, Shapiro, Thurston \cite{fomin2008cluster} that some cluster algebras arise from triangulated orientable suraces. Subsequently Dupont and Palesi \cite{dupont2015quasi} generalised this construction to include unpunctured non-orientable surfaces, giving birth to quasi-cluster algebras. In \cite{wilson2017laurent} we linked this framework to Lam and Pylyavskyy's Laurent phenomenon algebras \cite{lam2012laurent}, showing that unpunctured surfaces admit an LP structure. In this paper we extend quasi-cluster algebras to include punctured surfaces. Moreover, by adding laminations to the surface we demonstrate that all punctured and unpunctured surfaces admit LP structures.

\end{abstract}

\pagenumbering{arabic}

\tableofcontents

\section{Introduction}

Fomin and Zelevinsky introduced cluster algebras in 2002 \cite{fomin2002cluster}; broadly speaking they are algebras whose generators, \textit{cluster variables}, have been grouped into overlapping finite sets of the same size called \textit{clusters}. Moreover, one can move between these clusters in a process known as \textit{mutation}, which consists of replacing one cluster variable with another. This framework of clusters and mutations is known as the \textit{cluster structure}.

In practice, the whole cluster structure is not actually given from the outset, but rather an initial cluster is provided together with an initial piece of combinatorial data which describes how to obtain more clusters. For cluster algebras (of geometric type) this data is a skew-symmetrizable matrix; the columns of the matrix are in bijection with the cluster variables. Moreover, to each column $i$ of the matrix one can canonically associate a binomial $F_i$, and mutation \textit{in direction $i$} consists of replacing the cluster variable $x_i$ with $x_i'$ using the relation:
$$x_i x_i' = F_i .$$
Additionally the matrix also changes under rules governed by itself, and repeated employment of this process is then used to obtain the whole cluster structure.

This $21^{\text{st}}$ century construct has already proven to be ubiquitous in mathematics -- examples cover vast areas including Poisson geometry, integrable systems, mathematical physics, quiver representations, Teichm\"{u}ller theory and polytopes. The theory is not just enjoyed using examples though, there are many deep and beautiful results running through this structure. Perhaps one of the most famous is the celebrated Laurent phenomenon which states that every cluster variable can be expressed as a Laurent polynomial in the initial cluster variables \cite{fomin2002laurent}. The Caterpillar Lemma \cite{fomin2002laurent} was the key ingredient used to realise this phenomenon but it is worth noting that the lemma provides much more generality than the cluster algebra setting requires. Aimed at extracting the full power of the lemma Lam and Pylyavskyy manufactured their own broader cluster structure purposely built to boast the Laurent phenomenon \cite{lam2012laurent}. They appropriately named their construction the \textit{Laurent phenomenon algebra}, or LP algebra for short. Their setup is essentially the same as in cluster algebras, however, now the cluster variables exchange under the relation:
$$x_i x_i' = \frac{F_i}{M}$$
where $F_i$ is an irreducible polynomial, and $M$ is a monomial determined by a \textit{normalisation process}. This freedom to have polynomials vastly generalises the binomial setup of cluster algebras, and even though there is a seemingly harsh restriction of the exchange polynomials being irreducible, Lam and Pylyavskyy showed that LP algebras still encompass cluster algebras. Namely, they showed every cluster algebra with principal coefficients is an LP algebra \cite{lam2012laurent}. 

One of the most visually comprehensible appearances of cluster algebras comes from the study of orientable surfaces \cite{fock2007dual},\cite{fomin2008cluster},\cite{fomin2012cluster}. Given an orientable marked surface we may \textit{triangulate} it. For each triangulation $T$ of the surface we may assign a \textit{seed} in which the cluster variables correspond to arcs in $T$, and the skew symmetric matrix is obtained via the process of inscribing cycles in each triangle, with respect to the surface's orientation. These seeds form a cluster algebra structure where mutations correspond to flipping arcs in triangulations. The underlying reason for this behaviour is explained by recognising that each cluster variable actually represents the \textit{(lambda) length} of their corresponding arc, and the matrices encode how these lengths are related.

Subsequently, on a purely geometric level, Dupont and Palesi generalised this process to unpunctured non-orientable surfaces \cite{dupont2015quasi}. They decided upon a notion of \textit{quasi-triangulation} that guaranteed the flippability of every constituent \textit{quasi-arc}, and then discovered the various relationships between their \textit{(lambda) lengths}. The cluster structure is initiated by fixing a quasi-triangulation $T$ and a set of algebraically independent \textit{cluster variables} corresponding to the quasi-arcs in $T$. Mutation then consists of performing flips of quasi-arcs and exchanging cluster variables under the relationship governed by how the lengths of their corresponding quasi-arcs transform. With this done, their \textit{quasi-cluster algebras} were born.

After making a small tweak to their setup, we provided a purely combinatorial description of this geometric mutation process, showing that both orientable and non-orientable unpunctured surfaces exhibit an LP structure \cite{wilson2017laurent}. Moreover, it was shown that for punctured surfaces the analogous constructions could not possess an LP structure -- the inhibiting factor being that punctured surfaces admit triangulations containing arcs whose exchange polynomials coincide.

In this paper we first extend the construction of a quasi-cluster algebra to punctured surfaces. Then, with the desire to modify the exchange polynomials, we imitate the work of Fomin and Thurston \cite{fomin2012cluster} by adding laminations to the surface and introducing \textit{laminated lambda lengths}; a notion of length that takes into account these laminations, as well as the underlying geometry. By embodying the concept of principal coefficients for cluster algebras, we introduce \textit{principal laminations}. Crucially this class of laminations guarantees the uniqueness of exchange polynomials in every quasi-triangulation, allowing us to obtain a geometric realisation of LP algebras for all bordered surfaces: \newline

\noindent \textbf{Main Theorem} (Theorem \ref{maintheorem}). \textit{Let $(S,M)$ be an orientable or non-orientable marked surface and $\mathbf{L}$ a principal lamination. Then the LP cluster complex $\Delta_{LP}(S,M,\mathbf{L})$ is isomorphic to the laminated quasi-arc complex $\Delta^{\otimes}(S,M,\mathbf{L})$, and the exchange graph of $\mathcal{A}_{LP}(S,M,\mathbf{L})$ is isomorphic to $E^{\otimes}(S,M,\mathbf{L})$.} 

\textit{More explicitly, if $(S,M)$ is not a once-punctured closed surface, the isomorphisms may be rephrased as follows. Let $T$ be a quasi-triangulation of $(S,M)$ and $\Sigma_{T}$ its associated LP seed. Then in the LP algebra $\mathcal{A}_{LP}(\Sigma_{T})$ generated by this seed the following correspondence holds:}
\begin{align*}
&\hspace{8mm} \mathbf{\mathcal{A}_{LP}(\Sigma_T)} & & &\mathbf{(S,M,\mathbf{L})} \hspace{26mm}&  \\ 
&\textit{Cluster variables} &\longleftrightarrow& &\textit{Laminated lambda lengths of quasi-arcs} & \\
&\hspace{8mm}\textit{Clusters}  &\longleftrightarrow& &\textit{Quasi-triangulations} \hspace{16mm}& \\
&\hspace{4mm} \textit{LP mutation}   &\longleftrightarrow&  &\textit{Flips} \hspace{30.5mm}& \\
\end{align*}

An analogous correspondence holds for once-punctured surfaces too, it is just in this case the exchange graph splits into two isomorphic components. Moreover, specialising the variables corresponding to the laminations we get the following result regarding the underlying quasi-cluster algebra: \newline

\noindent \textbf{Corollary} (Corollary \ref{maincor}). \textit{Let $(S,M)$ be a bordered surface. Then the quasi-cluster algebra $\mathcal{A}(S,M)$ is a specialised LP algebra.} \newline

The paper is organised as follows. In Section 2 we recall the construction of Lam and Pylyavskyy's LP algebras and introduce the notion of a specialised LP algebra. Section 3 extends Dupont and Palesi's quasi-cluster algebras to include punctured surfaces -- this extension is in keeping with the construction already established on orientable surfaces \cite{fomin2008cluster},\cite{fomin2012cluster}. Moreover, as in \cite{wilson2017laurent}, to enable a connection of these quasi-cluster algebras to LP algebras, we make a small alteration to Dupont and Palesi's compatibility relations. In Section 4 we consider the double cover of triangulated surfaces and remark on a certain \textit{anti-symmetric} property of the associated quivers. The section concludes with an explanation on the relationship between mutation of these anti-symmetric quivers and LP mutation. Sections 5 and 6 make up the bulk of the paper and are devoted to finding a cluster structure on surfaces that fits into the LP algebra framework. Namely, in Section 5, imitating \cite{fomin2012cluster}, we introduce laminations on our surface with the intention of defining a notion of length of quasi-arc that depends on intersection numbers with these laminations as well as the underlying geometry. We open up the punctures of $(S,M)$ so that intersection numbers between quasi-arcs and laminations on this opened surface $(S^*,M^*)$ are finite. For quasi-arcs $\gamma$ and laminations $L$ on $(S,M)$ we fix associated lifts $\overline{\gamma}$ and $L^*$ on $(S^*,M^*)$, and from here we can define the \textit{laminated lambda length} of each $\gamma$; this is a rescaling of the lambda length of the lifted arc $\overline{\gamma}$ by the \textit{tropical lambda length} $c_{\mathbf{L}^*}(\overline{\gamma})$ -- a length that measures the number of intersections between $\overline{\gamma}$ and $\mathbf{L}^*$. Moreover, we put boundary conditions on the opened punctures to ensure the laminated lambda length does not depend on the choice of lifts we take for the quasi-arcs. We conclude the section by defining the associated \textit{laminated quasi-cluster algebra}. In Section 6 we first investigate the exchange relations between laminated lambda lengths of quasi-arcs. In particular, we show how we can obtain these relations from the associated anti-symmetric quivers, and we discover how the quivers change under flips. From here, assuming distinctness of exchange polynomials in each quasi-triangulation, we show the laminated quasi-cluster algebra has an LP structure. With this in mind we introduce \textit{principal laminations}; a class of laminations that ensures distinctness of exchange polynomials in each triangulation. The proof of distinctness is essentially obtained by showing the rank of the \textit{shortened exchange matrix} of an anti-symmetric quiver is invariant under mutation. We conclude the paper with the statement and proof of the Main Theorem along with a corollary regarding specialised LP algebras.

\section*{\large \centering Acknowledgements}
I would like to thank Anna Felikson for her continued support both during and after my PhD studies -- her careful reading of earlier versions of this paper also vastly improved its presentation. I also wish to thank Pavel Tumarkin, Pavlo Pylyavskyy and Vladimir Fock for stimulating discussions and their many insightful comments.

\section{Laurent phenomenon algebras}

This chapter follows the work of Lam and Pylyavskyy \cite{lam2012laurent}. We will first introduce the notion of a Laurent phenomenon algebra and then conclude the section with the idea of a specialised Laurent phenomenon algebra. \newline

Let the \textit{coefficient ring} $R$ be a unique factorisation domain over $\mathbb{Z}$ and let $\mathcal{F}$ denote the field of rational functions in $n \geq 1$ independent variables over the field of fractions $\text{Frac}(R)$. \newline

A Laurent phenomenon (LP) \textit{\textbf{seed}} in $\mathcal{F}$ is a pair $(\textbf{x}, \textbf{F})$ satisfying the following conditions:

\begin{itemize}

\item $\textbf{x} = \{x_1, \ldots, x_n\}$ is a transcendence basis for $\mathcal{F}$ over $\text{Frac}(R)$.

\item $\textbf{F} = \{F_1, \ldots, F_n\}$ is a collection of irreducible polynomials in $R[x_1, \ldots, x_n]$ such that for each $i \in \{1, \ldots, n\}$, $F_i \notin \{x_1, \ldots, x_n\}$; and $F_i$ does not depend on $x_i$ .

\end{itemize} 

Adopting the terminology of cluster algebras, $\textbf{x}$ is called the \textit{\textbf{cluster}} and $x_1, \ldots, x_n$ the \textit{\textbf{cluster variables}}. $F_1, \ldots, F_n$ are called the \textit{\textbf{exchange polynomials}}. \newline

Recall that a cluster algebra seed of geometric type $(\textbf{x}, B)$ consists of a cluster $\textbf{x} = \{x_1, \ldots, x_n\}$ and an $m \times n$ integer matrix $B = (b_{ij})$ whose top $n \times n$ submatrix is skew-symmetrizable. We can recode the columns of this matrix as binomials defined by $F^B_j := \prod_{b_{ij}>0}x_i^{b_{ij}} + \prod_{b_{ij}<0}x_i^{-b_{ij}}$, so there is a strong similarity between the definition of cluster algebra and LP seeds. The key difference being that for LP algebras the exchange relations can be polynomial, not just binomial. However, unlike in cluster algebras, these polynomials are required to be irreducible. \newline
\indent To obtain an \textit\textbf{LP algebra} from a seed we imitate the construction of cluster algebras. Namely, we introduce a notion of mutation of seeds. The LP algebra will then be defined as the ring generated by all the cluster variables we obtain throughout the mutation process. Before we present the rules of mutation we first need to clarify notation and introduce the idea of normalising exchange polynomials.

\underline{\textbf{Notation:}} \begin{itemize}

\item Let $F$ and $G$ be Laurent polynomials in the variables $x_1, \ldots x_n$. We denote by $F \rvert_{x_j \leftarrow G}$ the expression obtained by substituting $x_j$ in $F$ for the Laurent polynomial $G$.

\item If $F$ is a Laurent polynomial involving a variable $x$ then we write $x \in F$. Similarly, $x \notin F$ indicates that $F$ does not involve $x$.

\end{itemize}

\begin{defn}

Given $ \mathbf{F} = \{F_1,\ldots,F_n\} $ from an LP seed, then for each $j\in \{1,\ldots,n\}$ we define $ \hat{F}_j := \frac{F_j}{x_1^{a_1} \ldots x_{j-1}^{a^{j-1}}x_{j+1}^{a_{j+1}} \ldots x_n^{a_n}}$ where $a_k \in \mathbb{Z}_{\geq 0}$ is maximal such that $F_k^{a_k} $ divides ${F}_j \rvert_{x_k \leftarrow \frac{F_k}{x}}$ as an element of $R[x_1, \ldots, x_{k-1},x^{-1},x_{k+1},\ldots,x_n]$.  The Laurent polynomials in $\mathbf{\hat{F}} := \{\hat{F}_1 , \ldots , \hat{F}_n\}$ are called the \textbf{\textit{normalised exchange polynomials}} of $\mathbf{F}$.
\end{defn}

\begin{exmp}
\label{norm}
Consider the following exchange polynomials in $\mathbb{Z}[a,b,c]$ $$F_a = 1+bc,\hspace{7mm} F_b = 1+a,\hspace{7mm} F_c = (1+a)^2 + ab^2.$$ \indent Since $F_b$ and $F_c$ both depend on $a$ then $F_{a}\rvert_{b \leftarrow \frac{F_b}{x}}$ and $F_{a}\rvert_{c \leftarrow \frac{F_c}{x}}$ are not divisible by $F_b$ and $F_c$ respectively. Consequently $\hat{F}_a = F_a$, and an analogous argument shows $\hat{F}_b = F_b$. Similarly, $c\in F_a$ implies $a\notin \frac{F_c}{\hat{F}_c}$. However, $2$ is the maximal power of $F_b$ that divides $F_{c}\rvert_{b \leftarrow \frac{F_b}{x}}$, so $\hat{F}_c = \frac{F_c}{b^2}$.

\end{exmp}

\begin{defn}

Let $(\mathbf{x}, \mathbf{F})$ be an LP seed and $i \in \{1,\ldots, n\}$. We define a new seed $\mu_i(\mathbf{x}, \mathbf{F}) := (\{x_1',\ldots,x_n'\},\{F_1',\ldots, F_n'\})$. Here $x_j' := x_j $ for $j \neq i$ and $x_i' := \hat{F}_i/x_i $.  The exchange polynomials change as follows:

\begin{itemize}

\item If $x_i \notin F_j$ then $F_j' := F_j$.

\item If $x_i \in F_j $ then $F_j' $ is obtained from the following 3 step process:

\begin{description}[align=left]
\item [(Step $\bf{1}$)] Define $G_j := F_j \rvert_{x_i\leftarrow \frac{\hat{F}_i\rvert_{x_j\leftarrow 0}}{x_i' }}$ 
\item [(Step $\bf{2}$)] Define $H_j := G_j$ divided out by all common factors with $\hat{F}_i \rvert_{x_j\leftarrow 0}$, so that none remain, i.e. we have $gcd(H_j,\hat{F}_i \rvert_{x_j\leftarrow 0})=1$.
\item [(Step $\bf{3}$)] Let $M$ be the unique monic Laurent monomial in $R[x_1'^{\pm 1},\ldots,x_n'^{\pm 1}]$ such that $F_j' := H_jM \in R[x_1',\ldots,x_n']$ and is not divisible by any of the variables $x_1',\ldots,x_n'$.
\end{description}

\end{itemize}

The new seed $\mu_i(\mathbf{x},\mathbf{F})$ is called the \textit{\textbf{mutation}} of $(\mathbf{x},\mathbf{F})$ in \textit{\textbf{direction $\boldsymbol{i}$}}. It is important to note that because of \textbf{Step $\bf{2}$} the new exchange polynomials are only defined up to a unit in $R$.

\end{defn}

It is certainly not clear a priori that $\mu_i(\mathbf{x},\mathbf{F})$ will be a valid LP seed due to the irreducibility requirement of the new exchange polynomials. Furthermore, due to the expression $\hat{F}_i\rvert_{x_j\leftarrow 0}$ appearing in \textbf{Step 1} it may not even be apparent that the process is well defined. These issues are resolved by the following two lemmas.

\begin{lem}[Proposition 2.7,\cite{lam2012laurent}]
\label{welldefined}
$x_i\in F_j\implies  x_j\notin \frac{F_i}{\hat{F}_i}$. In particular, $x_i\in F_j$ implies that $\hat{F}_i\rvert_{x_j\leftarrow 0}$ is well defined.

\end{lem}

\begin{lem}[Proposition 2.15, \cite{lam2012laurent}]

$F_j'$ is irreducible in $R[x_1',\ldots, x_n']$ for all $j \in \{1,\ldots,n\}$. In particular, $\mu_i(\mathbf{x}, \mathbf{F})$ is a valid LP seed.

\end{lem}

\begin{exmp}

We will perform mutation $\mu_{b}$ at $b$ on the LP seed $$(\{a,b,c\},\{F_a = 1+bc, F_b = 1+a, F_c = (1+a)^2 + ab^2\}).$$ Recall from Example \ref{norm} that $\hat{F}_b = F_b$. Both $F_a$ and $F_c$ depend on $b$ so we are required to apply the 3 step process on each of them. We shall denote the new variable $b' := \frac{\hat{F}_b}{b}$ by $d$. \newline
$$G_a = F_a \rvert_{b \leftarrow \frac{\hat{F_b} \rvert_{a \leftarrow 0}}{d}} = F_a \rvert_{b \leftarrow \frac{1}{d}} = 1+ \frac{c}{d}.$$ Nothing happens at Step 2 since $\hat{F_a} \rvert_{b \leftarrow 0} = 1$. Multiplying by the monomial $d$ gives us our new exchange polynomial $F_a' = d + c$. \newline
$$G_c = F_c \rvert_{b \leftarrow \frac{\hat{F_b} \rvert_{c \leftarrow 0}}{d}} = F_c \rvert_{b \leftarrow \frac{1+a}{d}} = (1+a)^2 + \frac{a(1+ a)^2}{d^2}.$$ Following Step 2 we divide $G_c$ by any of its common factors with $\hat{F_a} \rvert_{c \leftarrow 0} = 1+a$. This leaves us with $H_c = 1 + \frac{a}{d^2}$. Finally, multiplying by the monomial $d^2$ gives us our new exchange polynomial $F_c' = d^2 + a$. \newline
Hence, our new LP seed is $$(\{a,d,c\},\{F_a = d+c, F_d = 1+a, F_c = d^2 + a\}).$$

\end{exmp}

Recall that mutation in cluster algebras is an involution. In the LP algebra setting, because mutation of exchange polynomials is only defined up to a unit in $R$, it is clear we cannot say precisely the same thing for LP mutation. Nevertheless, we do have the following analogue.

\begin{prop}[Proposition 2.16, \cite{lam2012laurent}]

If $(\mathbf{x}', \mathbf{F}')$ is obtained from $(\mathbf{x}, \mathbf{F})$ by mutation at $i$, then $(\mathbf{x}, \mathbf{F})$ can be obtained from $(\mathbf{x}',\mathbf{F}')$ by mutation at $i $. It is in this sense that LP mutation is an involution.
\end{prop}

\begin{defn}

A \textbf{\textit{Laurent phenomenon algebra}} $(\mathcal{A},\mathcal{S})$ consists of a collection of seeds $\mathcal{S}$, and a subring $\mathcal{A}\subset \mathcal{F}$ that is generated by all the cluster variables appearing in the seeds of $\mathcal{S} $. This collection of seeds must be connected and closed under mutation. More formally, $\mathcal{S} $ is required to satisfy the following conditions: \begin{itemize}

\item Any two seeds in $\mathcal{S}$ are connected by a sequence of LP mutations.
\item$\forall$ $(\mathbf{x},\mathbf{F}) \in \mathcal{S}$ $\forall i \in \{1,\ldots,n\}$ there is a seed $(\mathbf{x}',\mathbf{F}') \in \mathcal{S}$ that can be obtained by mutating $(\mathbf{x},\mathbf{F})$ at $i$.

\end{itemize}

\end{defn}

\begin{defn}[Section 3.6, \cite{lam2012laurent}]

The \textit{\textbf{cluster complex }} $\Delta_{LP}(\mathcal{A})$ of an LP algebra $\mathcal{A}$ is the simplicial complex with the ground set being the cluster variables of $\mathcal{A}$, and the maximal simplices being the clusters.

\end{defn}

\begin{defn}[Subsection 3.6, \cite{lam2012laurent}]

The \textit{\textbf{exchange graph }} of an LP algebra $\mathcal{A}$ is the graph whose vertices correspond to the clusters of $\mathcal{A}$. Two vertices are connected by an edge if their corresponding clusters differ by a single mutation.

\end{defn}

\begin{defn}

A \textbf{\textit{specialised Laurent phenomenon algebra}} $(\mathcal{A}, \mathcal{S})_{sp}$ is the structure obtained from an LP algebra $(\mathcal{A}, \mathcal{S})$ when evaluating some elements in the coefficient ring $R$ at $1$.

\end{defn}

It is worth noting that, unlike in cluster algebras, this specialisation process does not generally produce another LP algebra. For $(\mathcal{A}, \mathcal{S})_{sp}$ to be an LP algebra the specialisation must commute with mutation. Namely, we would need $\mu_i(\Sigma)_{sp} = \mu_i(\Sigma_{sp})$ for each LP seed $\Sigma \in \mathcal{S}$. The following example shows this is not true in general.

\begin{exmp}
\label{not LP}

Consider the following seed where our coefficient ring is $R= \mathbb{Z}[X]$: $$\Sigma = (\{a,b,c\}, \{F_a = 1+Xb, F_b = a+c, F_c = 1+b\}).$$ Perfoming mutation at $a$ we obtain $$\mu_a(\Sigma) = (\{a'=\frac{1+Xb}{a},b,c\}, \{F_{a'} = 1+Xb, F_b = 1+a'c, F_c = 1+b\}).$$ However, if we specialise at $X=1$ and mutate the specialisation of $\Sigma$ at $a$, we get $$\mu_a(\Sigma_{sp}) = (\{a'=\frac{1+b}{ac},b,c\}, \{F_{a'} = 1+b, F_b = a'+1, F_c = 1+b\}).$$ Seeing as $\mu_a(\Sigma)_{sp} \neq \mu_a(\Sigma_{sp})$ we realise that the specialisation of the LP algebra generated by $\Sigma$ is not itself an LP algebra.

\end{exmp}

\section{Quasi-cluster algebras}

This section continues our previous paper \cite{wilson2017laurent}, which was based on the work of Dupont and Palesi \cite{dupont2015quasi}. Namely, we extend the construction of a quasi-cluster algebra to include punctured surfaces.

Let $S$ be a compact $2$-dimensional manifold. Fix a finite set $M$ of marked points of $S$ such that each boundary component contains at least one marked point - we will refer to marked points in the interior of $S$ as \textit{punctures}. The tuple $(S,M)$ is called a \textit{\textbf{bordered surface}}. We wish to exclude cases where $(S,M)$ does not admit a triangulation. As such, we do not allow $(S,M)$ to be an unpunctured or once-punctured monogon; digon; triangle; once or twice punctured sphere; M\"obius strip with one marked point on the boundary; or the once-punctured projective space. For technical reasons we also exclude the case where $(S,M)$ is the thrice-punctured sphere, the twice-punctured projective space and the once-punctured Klein bottle. \newline

To imitate the construction of cluster algebras arising from orientable surfaces we must first agree on which curves will form our notion of 'triangulation'. Our definitions are based on the theories developed by: Fock and Goncharov \cite{fock2007dual}, and Fomin, Shapiro and Thurston \cite{fomin2008cluster} on orientable surfaces; and Dupont and Palesi on non-orientable surfaces \cite{dupont2015quasi}. As in \cite{wilson2017laurent}, the key difference to our setup is the adjustment made to Dupont and Palesi's compatibility relations; this alteration facilitates the eventual connecting of quasi-cluster algebras to Laurent phenomenon algebras.

\begin{defn}

An \textit{\textbf{ordinary arc}} of $(S,M)$ is a simple curve in $S$ connecting two (not necessarily distinct) marked points of $M$, which is not homotopic to a boundary arc or a marked point.
\end{defn}

\begin{defn}

An \textit{\textbf{arc}} $\gamma$ is obtained from decorating ('tagging') an ordinary arc at each of its endpoints in one of two ways; \textit{\textbf{plain}} or \textit{\textbf{notched}}. This tagging is required to satisfy the following conditions:

\begin{itemize}

\item An endpoint of $\gamma$ lying on the boundary $\partial S$ must receive a plain tagging.

\item If the endpoints of $\gamma$ coincide they must receive the same tagging.

\end{itemize}

\end{defn}

\begin{defn}

A simple closed curve in $S$ is said to be \textit{\textbf{two-sided}} if it admits a regular neighbourhood which is orientable. Otherwise, it is said to be \textit{\textbf{one-sided}}.

\end{defn}

\begin{defn}

A \textit{\textbf{quasi-arc}} is either an arc or a one-sided closed curve. Throughout this paper we shall always consider quasi-arcs up to isotopy. Let $A^\otimes(S,M)$ denote the set of all quasi-arcs (considered up to isotopy).

\end{defn}

Recall that a closed non-orientable surface is homeomorphic to the connected sum of $k$ projective planes $\mathbb{R}P^2$. Such a surface is said to have (non-orientable) genus $k$. A \textit{\textbf{cross-cap}} is a cylinder where antipodal points on one of the boundary components are identified. In particular, note that a cross-cap is homeomorphic to $\mathbb{R}P^2$ with an open disk removed. An illustration of a cross cap in given in Figure \ref{crosscap} - throughout this paper we shall always represent it in this way.
For pictorial convenience we use the following alternative description: A compact non-orientable surface of genus $k$ (with boundary) is homeomorphic to a sphere where (more than) $k$ open disks are removed, and $k$ of them have been replaced with cross-caps.

\begin{figure}[H]
\begin{center}
\includegraphics[width=3cm]{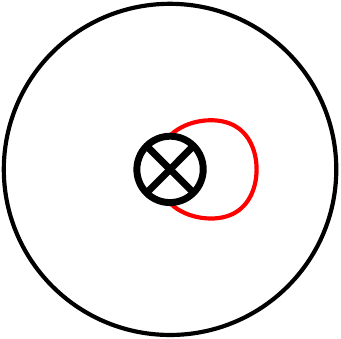}
\caption{A crosscap together with an example of a one-sided closed curve.}
\label{crosscap}
\end{center}
\end{figure}

\begin{defn}[Compatibility of arcs] 

Let $\alpha$ and $\beta$ be two arcs of $(S,M)$. We say $\alpha$ and $\beta$ are \textit{\textbf{compatible}} if and only if the following conditions are satisfied:

\begin{itemize}

\item There exist isotopic representatives of $\alpha$ and $\beta$ that don't intersect in the interior of $S$.

\item Suppose the untagged versions of $\alpha$ and $\beta$ do not coincide. If $\alpha$ and $\beta$ share an endpoint $p$ then the ends of $\alpha$ and $\beta$ at $p$ must be tagged in the same way.

\item Suppose the untagged versions of $\alpha$ and $\beta$ do coincide. Then precisely one end of $\alpha$ must be tagged in the same way as the corresponding end of $\beta$.

\end{itemize}

\end{defn}

To each arc $\gamma$ bounding a M\"obius strip with one marked point, $M_1^{\gamma}$, we uniquely associate the two quasi-arcs of $M_1^{\gamma}$. Namely, we associate the one-sided closed curve $\alpha_{\gamma}$ and the arc $\beta_{\gamma}$ enclosed in $M_1^{\gamma}$, see Figure \ref{intersectioncompatible}.

\begin{figure}[H]
\begin{center}
\includegraphics[width=3cm]{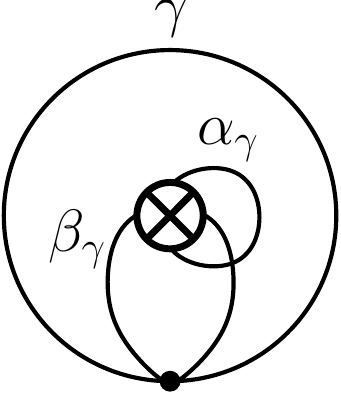}
\caption{The two quasi-arcs $\alpha_{\gamma}$ and $\beta_{\gamma}$ enclosed in the M\"obius strip, $M_1^{\gamma}$, cut out by an arc $\gamma$.}
\label{intersectioncompatible}
\end{center}
\end{figure}

\begin{defn}[Compatibility of quasi-arcs]
\label{newcompatibilitydef}
We say that two quasi-arcs $\alpha$ and $\beta$ are \textit{\textbf{compatible}} if either:
\begin{itemize}
\item $\alpha$ and $\beta$ are compatible arcs;
\item $\alpha$ and $\beta$ are not both arcs, and either $\alpha$ and $\beta$ do not intersect or $\{\alpha,\beta\} = \{\alpha_{\gamma},\beta_{\gamma}\}$ for some arc $\gamma$ bounding a M\"obius strip $M_1^{\gamma}$ - see Figure \ref{intersectioncompatible}.
\end{itemize}

\end{defn}

\begin{defn}
\label{quasitriangulationdef}
A \textit{\textbf{quasi-triangulation}} of $(S,M)$ is a maximal collection of pairwise compatible quasi-arcs of $(S,M)$ containing no arcs that cut out a once-punctured monogon or a M\"obius strip with one marked point on the boundary -- an example is shown on the left in Figure \ref{idealtriangulation}. A quasi triangulation is referred to as a \textbf{\textit{triangulation}} if it contains no one-sided closed curves.

\end{defn}

\begin{defn}
\label{ideal quasi-triangulation}
An \textit{\textbf{ideal quasi-triangulation}} of $(S,M)$ is a maximal collection of pairwise non-intersecting ordinary arcs and one-sided closed curves of $(S,M)$ -- an example is shown on the right in Figure \ref{idealtriangulation}. We shall refer to the curves comprising an ideal quasi-triangulation as \textit{\textbf{ordinary quasi-arcs}}.

\end{defn}

\begin{rmk}
After putting a hyperbolic metric on $(S,M)$ we need only ever consider the geodesic representatives of ordinary quasi-arcs to decide which collections form ideal quasi-triangulations. This is due to the fact that ordinary quasi-arcs have non-intersecting representatives \textit{if and only if} their geodesic representatives do not intersect. An analogous statement can be made when deciding which quasi-arcs form quasi-triangulations.
\end{rmk}

Let $T$ be a quasi-triangulation of $(S,M)$. As illustrated in Figure \ref{idealtriangulation}, we may associate an ideal quasi-triangulation $T^{\circ}$ to $T$ as follows:

\begin{itemize}

\item If $p$ is a puncture with more than one incident notch, then replace all these notches with plain taggings.

\item If $p$ is a puncture with precisely one incident notch, and this notch belongs to $\beta$, then replace $\beta$ with the unique arc $\gamma$ which encloses $\beta$ and $p$ in a monogon.

\item If $\alpha$ is a one-sided closed curve in $T$ then (by maximality of a quasi-triangulation) there exists a unique arc $\beta$ in $T$ which intersects $\alpha$. Replace $\beta$ with the unique arc $\gamma$ enclosing $\alpha$ and $\beta$ in a M\"obius strip with one marked point.

\end{itemize}

\begin{figure}[H]
\begin{center}
\includegraphics[width=11cm]{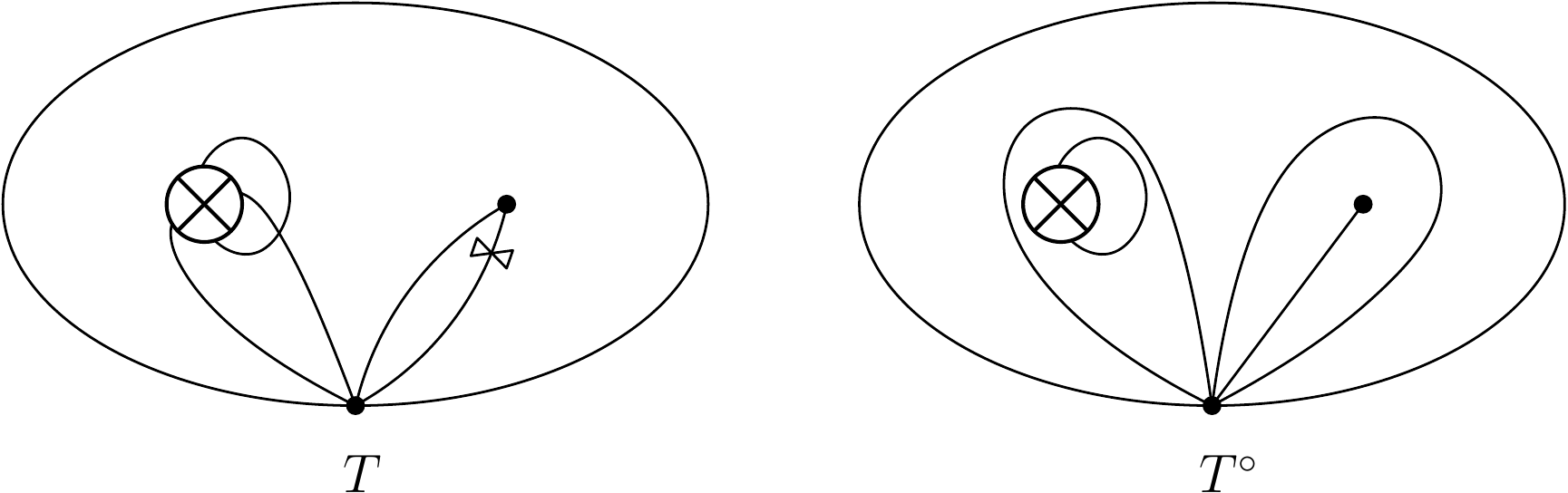}
\caption{Transforming a quasi-triangulation $T$ into an ideal triangulation $T^{\circ}$.}
\label{idealtriangulation}
\end{center}
\end{figure}

\begin{lem}
\label{puzzlepieces}
Let $T$ be a quasi-triangulation of $(S,M)$. Then $T^{\circ}$ cuts $(S,M)$ into triangles and annuli.

\end{lem}

\begin{proof}

Firstly, cut along all arcs in $T^{\circ}$ (i.e don't cut along any one-sided closed curves) to obtain a collection of connected components. Let $K$ be one of these connected components. Note that because we have cut along arcs, $K$ will have boundary with at least one marked point on each boundary component. Furthermore, we may assume $K$ has only one boundary component and no punctures as otherwise this contradicts the maximality of our quasi-triangulation. 

If $K$ is non-orientable then it contains a one-sided closed curve $\alpha \in T$. Let $\gamma$ be a curve that encloses $\alpha$ in a M\"obius strip with one marked point. By the maximality of the quasi-triangulation, $\gamma$ is either isotopic to the boundary of $K$, forcing $K$ to be the M\"obius strip with one marked point, or $\gamma$ is contractable resulting in $(S,M)$ being the once-punctured projective space. Since we have forbidden the later case then if $K$ is non-orientable it is the M\"obius strip with one marked point and a one-sided closed curve. Cutting along the one-sided closed curve yields the annulus with a marked point on one boundary component and the other empty of marked points.

What remains is to consider the case when $K$ is orientable. $K$ cannot be a monogon as then either $(S,M)$ itself is a monogon, or the boundary of $K$ is a contractable curve in $(S,M)$ and is therefore not a valid arc.
Similarly, $K$ cannot be a digon as then one of the following situations occur: $(S,M)$ is itself a digon; the two boundary segments of $K$ are isotopic; or $(S,M)$ is obtained from gluing together the boundary of $K$ with the result being the twice punctured sphere or the once-punctured projective space.
$K$ cannot have more than four marked points as this would contradict the maximality of the quasi-triangulation. Hence if $K$ is orientable it must be a triangle.

\end{proof}

\begin{prop}
\label{flip}
Let $T$ be a quasi-triangulation of $(S,M)$. Then for any $\gamma \in T$ there exists a unique $\gamma' \in A^\otimes(S,M)$ such that $\gamma' \neq \gamma$ and $\mu_{\gamma}(T) := T\setminus\{\gamma\}\cup \gamma'$ is a quasi-triangulation. We call $\gamma'$ the \textbf{\textit{flip}} of $\gamma$ with respect to $T$.

\end{prop}

\begin{proof}

For a quasi-triangulation $T$ of $(S,M)$ note that performing tag changing transformations at punctures has no effect on the flippability of quasi-arcs in $T$. Therefore, without loss of generality, we may assume that the only instance when a notched arc appears in T is when it is accompanied by its plain counterpart.  \newline \indent
To decide the flipability of an arc in $T$ we shall consider its local configuration. We achieve this by first considering the local configurations of quasi-arcs in the associated ideal quasi-triangulation $T^{\circ}$, and from here we will then discover the possible local pictures in $T$. \newline \indent
By Lemma \ref{puzzlepieces} we know that $T^{\circ}$ cuts $(S,M)$ into triangles and annuli - for convenience we shall refer to them as \textit{puzzle pieces}. Therefore any quasi-arc of $T^{\circ}$ is the glued side of two puzzle pieces. We list these gluings in Figure \ref{puzzlegluing} to obtain all possible neighbourhoods of a quasi-arc in $T^{\circ}$. When the configurations in Figure \ref{puzzlegluing} are pulled back to $T$ the only valid local configurations, shown in Figure \ref{flipregions}, are the quadrilateral, the punctured digon and the M\"{o}bius strip with two marked points - as by definition of a bordered surface we have forbidden the instance when $(S,M)$ is the thrice punctured sphere, the twice-punctured projective space, or the once-punctured Klein bottle. Each quasi-arc in the interior of the configurations in Figure \ref{flipregions} is uniquely flippable. An important point to add is that the boundary segments of these configurations may in fact be a substituted arc bounding a punctured monogon, or a M\"obius strip with one marked point. However, since this substituted arc, and the two quasi-arcs it bounds are compatible with precisely the same quasi-arcs, then this doesn't affect the existence or uniqueness of the flip in question.

\end{proof}

\begin{rmk}
The reason we have forbidden $(S,M)$ to be the thrice punctured sphere, the twice-punctured projective space, or the once-punctured Klein bottle should now be clear - the glued side of their corresponding configuration in Figure \ref{puzzlegluing} is pulled back to two arcs.
\end{rmk}

\begin{figure}[H]
\begin{center}
\includegraphics[width=13cm]{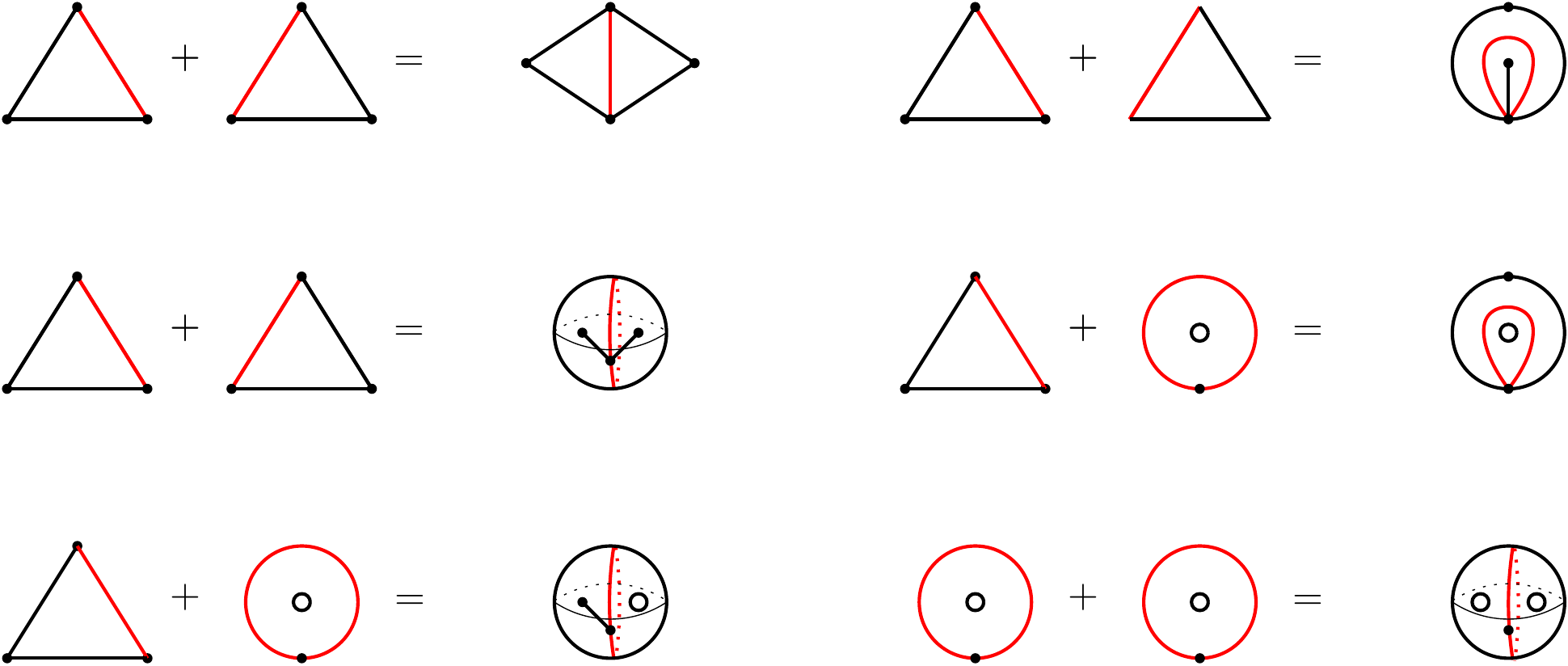}
\caption{The possible gluings of puzzle pieces.}
\label{puzzlegluing}
\end{center}
\end{figure}

\begin{figure}[H]
\begin{center}
\includegraphics[width=11cm]{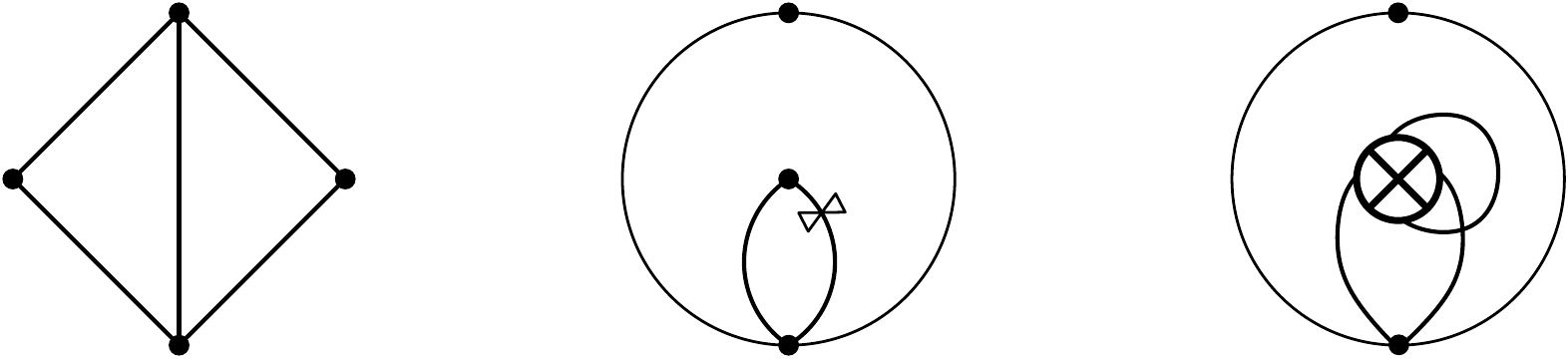}
\caption{The valid pullbacks obtained from gluings of puzzle pieces.}
\label{flipregions}
\end{center}
\end{figure}

\begin{defn}

The \textit{\textbf{flip graph}} of a bordered surface $(S,M)$ is the graph with vertices corresponding to quasi-triangulations and edges corresponding to flips.

\end{defn}

Harer \cite{harer1985stability} proved that two ideal triangulations on an orientable surface are connected via a sequence of flips. This result applies equally well to non-orientable surfaces; for a simple proof of this see Mosher \cite{mosher1988tiling}. The following proposition concerning the connectivity of the flip graph follows from the result of Harer, and arguments of Fomin, Shapiro and Thurston \cite{fomin2008cluster} regarding the ability to flip between plain and notched arcs in triangulations.

\begin{prop}
\label{flipconnected}
If $(S,M)$ is not a closed once-punctured surface then the flip graph of $(S,M)$ is connected. In the closed once-punctured case the flip graph has two isomorphic connected components: one containing only plain quasi-arcs, and the other containing only notched ones.

\end{prop}

Propositions \ref{flip} and \ref{flipconnected} tell us that the number of quasi-arcs in a quasi-triangulation is an invariant of $(S,M)$ - this number is called the \textit{\textbf{rank}} of $(S,M)$.
\newline

\indent We now introduce the notion of a seed of a bordered surface $(S,M)$.

\subsection*{Quasi-seeds and mutation.}

Suppose $(S,M)$ is a bordered surface of rank $n$ and let $b_1,\ldots, b_m$ consist of all the boundary segments of $(S,M)$. Denote by $\mathcal{F}$ the field of rational functions in $n+m$ independent variables over $\mathbb{Q}$.

A \textit{\textbf{quasi-seed}} of a bordered surface $(S,M)$ in $\mathcal{F}$ is a pair $(\mathbf{x},T)$ such that:

\begin{itemize}

\item $T$ is a quasi-triangulation of $(S,M)$.

\item $\mathbf{x} := \{x_{\gamma} \in \mathcal{F} | \gamma \in T\}$ is an algebraically independent set in $\mathcal{F}$ over $\mathbb{ZP} := \mathbb{Z}[x_{b_1},\ldots,x_{b_m}]$.

\end{itemize}

\indent We call $\mathbf{x}$ the \textit{\textbf{cluster}} of $(\mathbf{x},T)$ and the variables themselves are called \textit{\textbf{cluster variables}}.\newline

To define a cluster structure on $(S,M)$ we shall consider the \textit{decorated Teichm\"uller space}, $\tilde{\mathcal{T}}(S,M)$, as introduced by Penner \cite{penner2012decorated}. An element of $\tilde{\mathcal{T}}(S,M)$ consists of a complete finite-area hyperbolic structure of constant curvature $-1$ on $S\setminus M$ together with a collection of horocycles, one around each marked point. 

Fixing a decorated hyperbolic structure $\sigma \in \tilde{\mathcal{T}}(S,M)$ we may define the notion of \textit{lambda length}, $\lambda_{\sigma}(\gamma)$, for each quasi-arc $\gamma$ in $(S,M)$. More explicitly,

\[   
\lambda_{\sigma}(\gamma) = 
     \begin{cases}
       e^{\frac{l_{\sigma}(\gamma)}{{2}}},& \text{if $\gamma$ is an arc,}\\
       2\sinh(\frac{l_{\sigma}(\gamma)}{{2}}),& \text{if $\gamma$ is a one-sided closed curve,}\\
     \end{cases}
\]

\noindent where $l_{\sigma}(\gamma)$ is defined as follows. If $\gamma$ is a one-sided closed curve then $l_{\sigma}(\gamma)$ simply denotes the length of $\gamma$ in $\sigma$. If $\gamma$ is an arc then its endpoints are at cusps in $\sigma$, and so $\gamma$ will have infinite length. However, we define $l_{\sigma}(\gamma)$ to be the length of $\gamma$ between certain horocycles at its endpoints; the horocycle chosen at an endpoint will depend on how $\gamma$ is tagged. Recall that $\sigma$ comes equipped with a horocycle $h_k$ at each marked point $k$. If $\gamma$ has a plain tag at $k$ then we consider precisely the horocycle $h_k$. If $\gamma$ is notched at $k$ then we instead consider the \textit{conjugate horocycle} $\tilde{h}_k$, of $h_k$. (If $h_k$ has length $x$ then the \textit\textbf{{conjugate horocycle}} $\tilde{h}_k$ is defined to be the unique horocycle at $k$ with length $\frac{1}{x}$. \newline

The \textit{lambda length}, $\lambda({\gamma})$, of a quasi-arc $\gamma$ is the evaluation map on $\tilde{\mathcal{T}}(S,M)$ sending decorated hyperbolic structures $\sigma$ to $\lambda_{\sigma}(\gamma)$.

The theorem below follows from [Theorem 7.4, \cite{fomin2012cluster}] and [Remark 8.8, \cite{fomin2012cluster}].

\begin{thm}

For any quasi-triangulation $T$ with quasi-arcs and boundary arcs $\gamma_1, \ldots, \gamma_{n+b}$ there exists a homeomorphism 

 \begin{align*} 
\Lambda_T \colon   \tilde{\mathcal{T}}(S & ,M) \longrightarrow \mathbb{R}_{>0}^{n+b} \\
          &\sigma \mapsto (\lambda_{\sigma}(\gamma_1), \ldots, \lambda_{\sigma}(\gamma_{n+b}))
\end{align*} 

\end{thm}

As a consequence the lambda lengths of quasi-arcs and boundary arcs in a quasi-triangulation can be viewed as algebraically independent variables and we have a canonical isomorphism $$\mathbb{Q}(\{ \lambda(\gamma) | \gamma \in T\cup B(S,M) \}) \cong \mathcal{F}.$$

We may define a cluster structure by calculating how these lambda lengths are related under flips. We provide these precise relations below in Definition \ref{lambdarules}. Note that instead of working with the lambda lengths of quasi-arcs we shall instead always consider their corresponding elements in $\mathcal{F}$.

\begin{defn}
\label{lambdarules}
Given $\gamma \in T$ we define \textit{\textbf{mutation}} of $(\mathbf{x},T)$ in \textit{\textbf{direction \boldmath$\gamma$}} to be the pair $\mu_{\gamma}(\mathbf{x},T) := (\mathbf{x}',T')$ where $T' := \mu_{\gamma}(T)$ and $\mathbf{x}' := \mathbf{x} \setminus \{ x_{\gamma} \} \cup \{x_{\gamma'}\}$. The new variable $x_{\gamma'}$ depends on the combinatorial type of flip being performed. In Figure \ref{combinatorialflips} we list the possible flips and their corresponding exchange relations, which may be obtained using the combined results of \cite{dupont2015quasi} and \cite{fomin2012cluster}. \newline

\noindent (1). $\gamma$ is the diagonal of quadrilateral in which no two consecutive edges are identified.

\begin{figure}[H]
\begin{center}
\includegraphics[width=10cm]{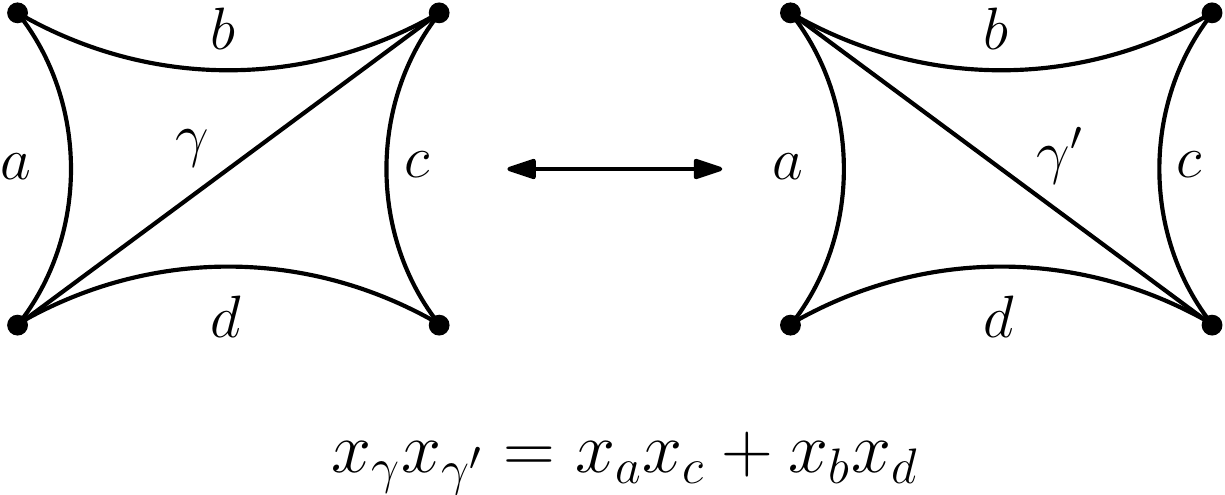}
\end{center}
\end{figure}

\noindent (2). $\gamma$ is an interior arc of a punctured digon.

\begin{figure}[H]
\begin{center}
\includegraphics[width=11cm]{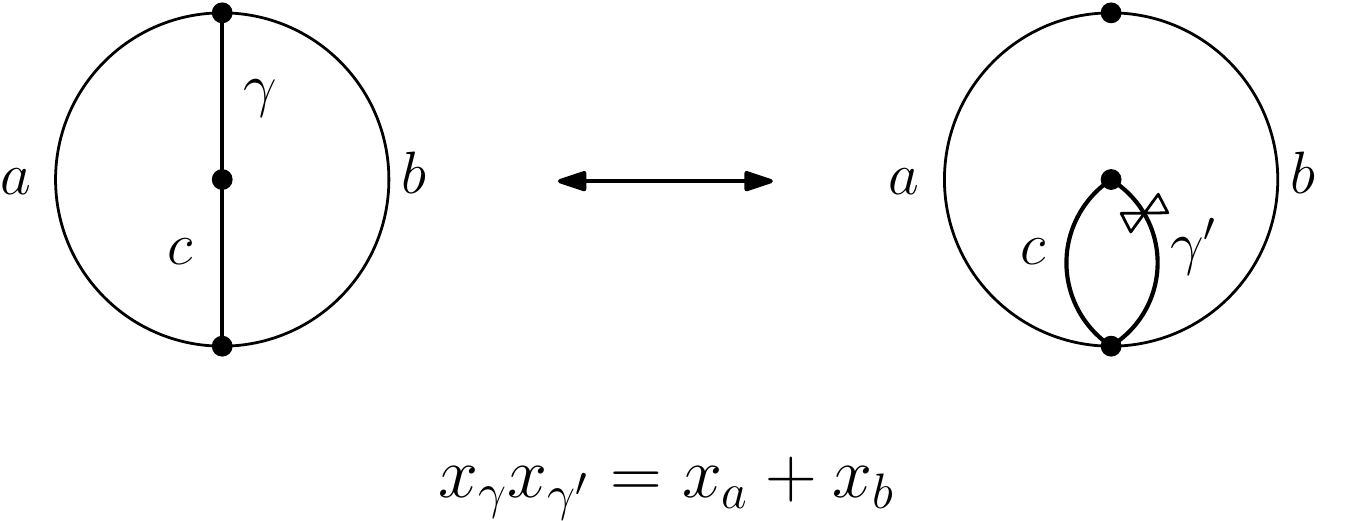}
\end{center}
\end{figure}

\noindent (3). $\gamma$ is an arc that flips to a one-sided closed curve, or vice verca.

\begin{figure}[H]
\begin{center}
\includegraphics[width=11cm]{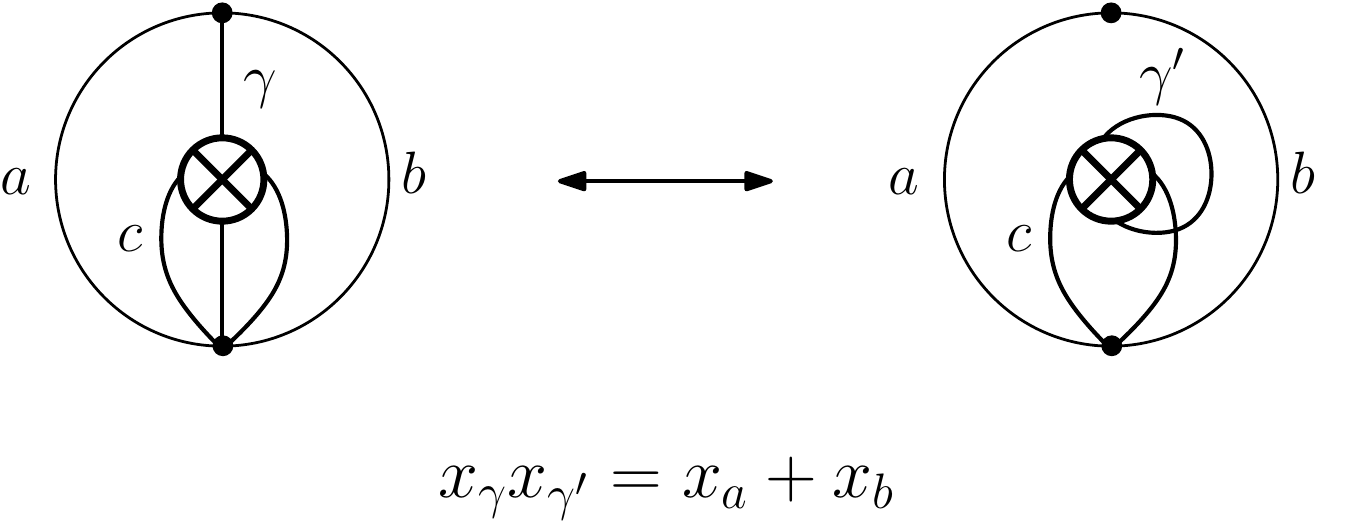}
\end{center}
\end{figure}

\noindent (4). $\gamma$ is an arc intersecting a one-sided close curve $c$.

\begin{figure}[H]
\begin{center}
\includegraphics[width=11cm]{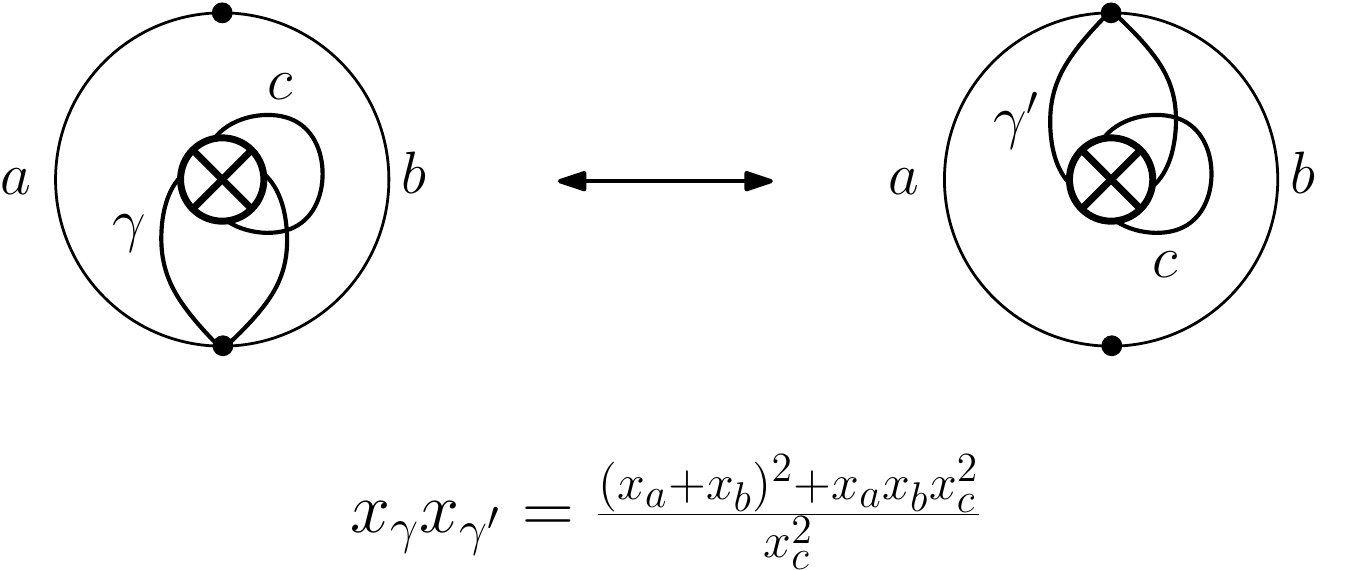}
\caption{Combinatorial types of flips together with their corresponding exchange relations.}
\label{combinatorialflips}
\end{center}
\end{figure}

\end{defn}

Let $(\mathbf{x},T)$ be a seed of $(S,M)$. If we label the cluster variables of $\mathbf{x}$ $1,\ldots, n$ then we can consider the labelled n-regular tree $\mathbb{T}_n$ generated by this seed through mutations. Each vertex in $\mathbb{T}_n$ has $n$ incident vertices labelled $1,\ldots,n$. Vertices represent seeds and the edges correspond to mutation. In particular, the label of the edge indicates which direction the seed is being mutated in. \newline

Let $\mathcal{X}$ be the set of all cluster variables appearing in the seeds of $\mathbb{T}_n$. $\mathcal{A}_{(\mathbf{x},T)}(S,M) := \mathbb{ZP}[\mathcal{X}]$ is the \textit{\textbf{quasi-cluster algebra}} of the seed $(\mathbf{x},T)$.

The definition of a quasi-cluster algebra depends on the choice of the initial seed. However, if we choose a different initial seed the resulting quasi-cluster algebra will be isomorphic to $\mathcal{A}_{(\mathbf{x},T)}(S,M)$. As such, it makes sense to talk about the quasi-cluster algebra of $(S,M)$.

\section{Anti-symmetric quivers}

\subsection{The double cover and anti-symmetric quivers}

Let $(S,M)$ be a bordered surface. We construct an orientable double cover of $(S,M)$ as follows. First consider the orientable surface $\tilde{S}$ obtained by replacing each cross-cap with a cylinder, see Figure \ref{surfaceandcylinder}.

\begin{figure}[H]
\begin{center}
\includegraphics[width=10cm]{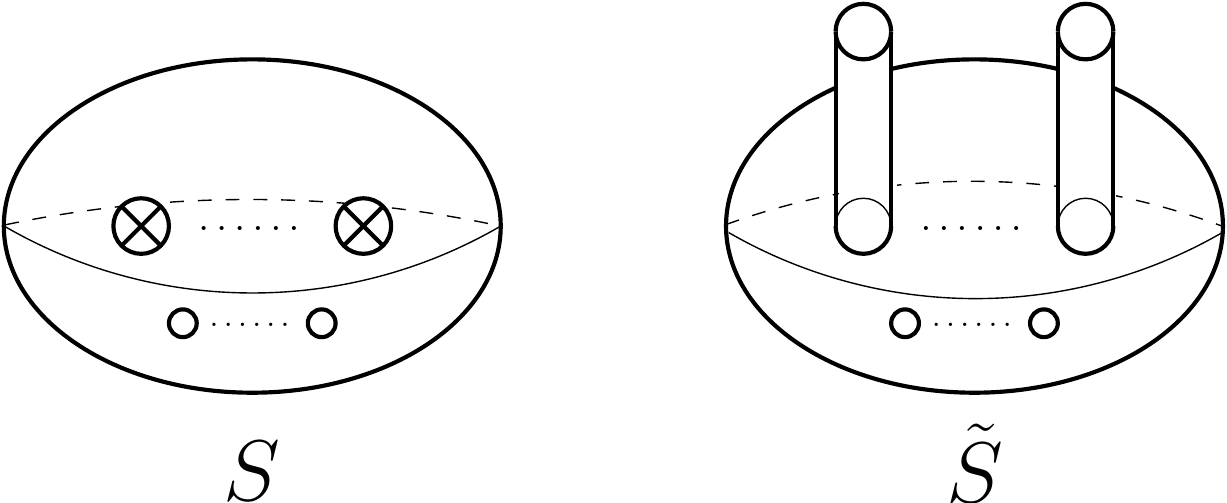}
\caption{An illustration of the non-orientable surface $S$ and the surface $\tilde{S}$ obtained by replacing each cross-cap with a cylinder. The small circles represent boundary components.}
\label{surfaceandcylinder}
\end{center}
\end{figure}

The orientable double cover $\overline{(S,M)}$ of $(S,M)$ is obtained by taking two copies of $\tilde{S}$ and glueing the boundary of each newly ajoined cylinder in the first copy, with a half twist, to the corresponding cylinder in the second copy. To clarify, each cylinder in the first copy is glued along their antipodal points in the second copy, see Figure \ref{surfaceglueing}. If $S$ is orientable then the double cover is two disjoint copies of $(S,M)$. In this case we endow the two disjoint copies with alternate orientations - this is to ensure its adjacency quiver is anti-symmetric, see Definition \ref{antisymmetric}.

\begin{figure}[H]
\begin{center}
\includegraphics[width=10cm]{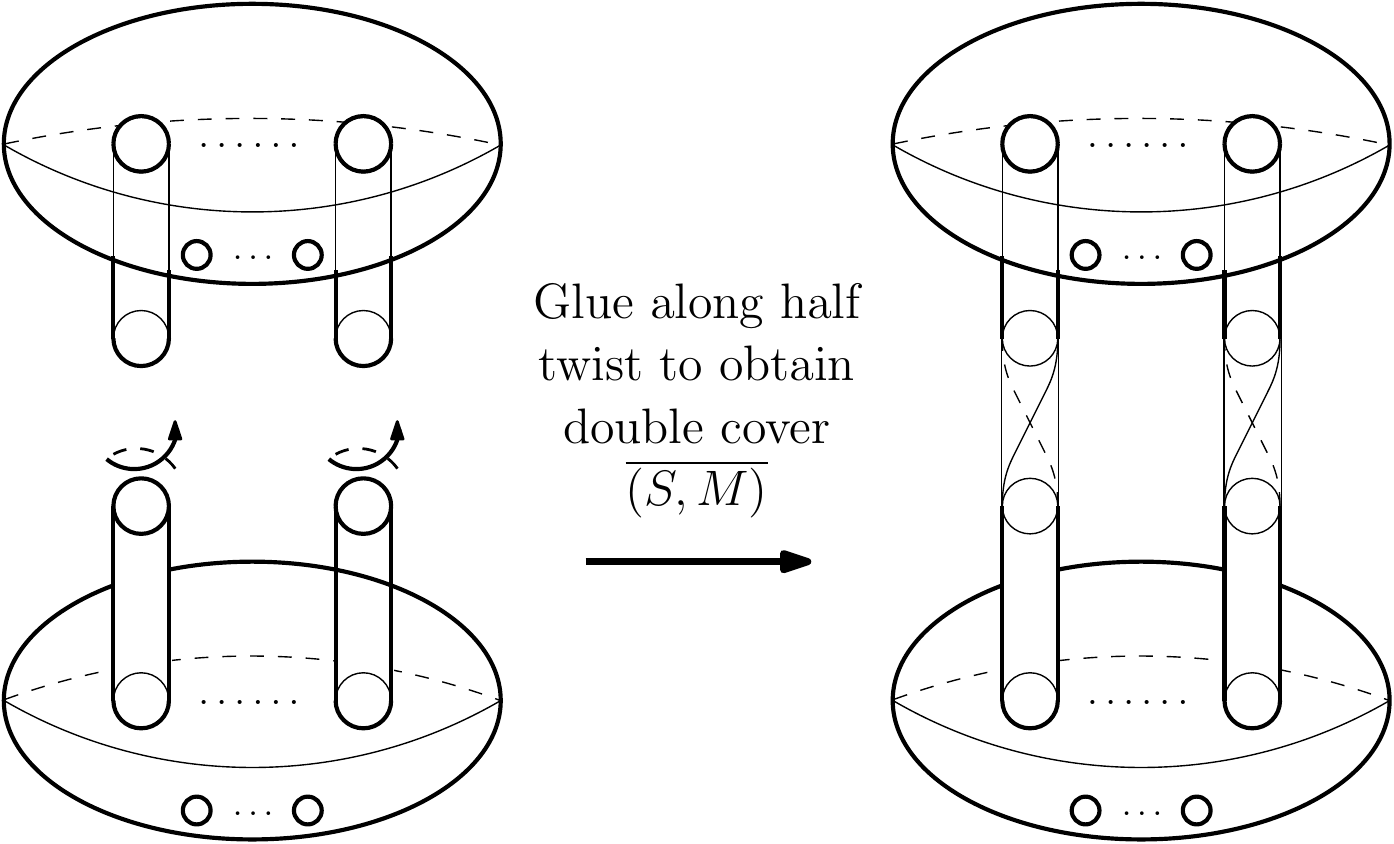}
\caption{The double cover is obtained by glueing two copies of $\tilde{S}$ along the boundaries of the newly adjoined cylinders.}
\label{surfaceglueing}
\end{center}
\end{figure}

If $T$ is a triangulation of $(S,M)$ then $T$ lifts to a triangulation $\overline{T}$ of the orientable double cover $\overline{(S,M)}$. Moreover, let $i$ be an arc in $T$ and, by abuse of notation, denote by $i$ and $\tilde{i}$ the two lifts in $\overline{T}$ of the arc $i \in T$ . Note that if $i$ and $j$ are arcs of a triangle $\Delta$ in $\overline{T}$, and $j$ follows $i$ in $\Delta$ under the agreed orientation of $\overline{(S,M)}$, then $\tilde{i}$ follows $\tilde{j}$ in the twin triangle $\tilde{\Delta}$. Hence in the quiver $Q_{\overline{T}}$ associated to $\overline{T}$ we have that $i \rightarrow j \iff \tilde{j} \rightarrow \tilde{i}$. Here we adopt the notation that $\tilde{\tilde{i}} = i$ for any $i \in \{1,\ldots, n\}$, and we shall use it throughout this paper.

Finally, note that there is no arrow $i \rightarrow \tilde{i}$ in $Q_{\overline{T}}$ as this would imply the existence of an anti-self-folded triangle in $T$, which is forbidden under our definition of triangulation, see Figure \ref{antiself}.

\begin{figure}[H]
\begin{center}
\includegraphics[width=10cm]{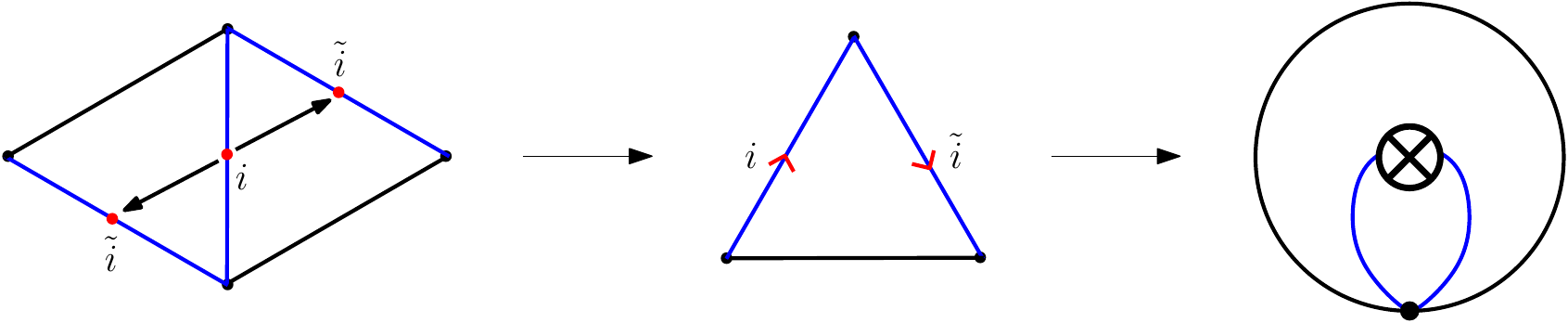}
\caption{On the left we show a segment of the double cover which admits the quiver $\tilde{i} \leftarrow i \rightarrow \tilde{i}$. On the right we show its $\mathbb{Z}_2$-quotient; an anti self-folded triangle. Such a configuration is forbidden under our definition of triangulation.}
\label{antiself}
\end{center}
\end{figure}

These two observations motivate the following definition.

\begin{defn}
\label{antisymmetric}

A quiver $Q$ on vertices $1,\ldots, n, \tilde{1},\ldots, \tilde{n}$ is called \textit{\textbf{anti-symmetric}} if:

\begin{itemize}

\item For any $i, j \in \{1,\ldots, n, \tilde{1},\ldots, \tilde{n}\}$ we have $i \rightarrow j \hspace{2mm} \textit{if and only if}  \hspace{2mm} \tilde{j} \rightarrow \tilde{i}$. 

\item  For any $i \in \{1,\ldots, n, \tilde{1},\ldots, \tilde{n}\}$ there are no arrows $i \rightarrow \tilde{i}$.

\end{itemize}

\end{defn}

The following proposition tells us that each flip between triangulations of $(S,M)$ corresponds to two flips in the double cover.

\begin{prop}
\label{quiverandflip}
Let $\gamma$ be an arc in a triangulation $T$, and by abuse of notation, denote its lifts in $\overline{T}$ by $\gamma$ and $\tilde{\gamma}$. If $\mu_{\gamma}(T)$ is a triangulation then $\mu_{\gamma}\circ\mu_{\tilde{\gamma}}(\overline{T}) = \mu_{\tilde{\gamma}}\circ\mu_{\gamma}(\overline{T}) = \overline{\mu_{\gamma}(T)}$.

\end{prop}

\begin{proof}

Consider the flip region of $\gamma$ in $T$. The interiors of the lifted flip regions will be disjoint, otherwise there would be arrows between the corresponding vertices of $\gamma$ and $\tilde{\gamma}$ in $Q_{\overline{T}}$, and Figure \ref{antiself} would then contradict the fact there are no anti-self-folded triangles in $T$. Finally, since $\mu_{\gamma}(T)$ is a triangulation, then for such triangulations, the definition of a flipping an arc will coincide on both non-orientable and orientable surfaces.

\end{proof}

\begin{rmk}

If $\mu_{\gamma}(T)$ is not a triangulation then $\gamma$ has flipped to a one-sided closed curve, meaning that $\overline{\mu_{\gamma}(T)}$ contains a closed curve and so is not a triangulation. In Section 5, by considering \textit{traditional triangulations}, we introduce an alternative flip for $\gamma$ in replacement of the one-sided closed curve. This alteration ensures that $\mu_{\gamma}\circ\mu_{\tilde{\gamma}}(\overline{T}) = \mu_{\tilde{\gamma}}\circ\mu_{\gamma}(\overline{T}) = \overline{\mu_{\gamma}(T)}$ is a triangulation for all arcs $\gamma$ in $T$. Moreover, this allows us to extend the existing theory of laminations on orientable surfaces to our cluster structure.

\end{rmk}

\subsection{Mutation of anti-symmetric quivers via LP mutation}

We shall now briefly leave the environment of triangulations and move to the more general setting of anti-symmetric quivers. In particular, we shall establish a connection between mutation of these quivers and LP-mutation. Recall that a quiver $Q$ can be equivalently encoded as a skew-symmetric matrix $B=(b_{ij})$. In what follows we shall interchange between the two viewpoints. 

Given an anti-symmetric quiver $Q = (b_{ij})$ we may assign an exchange polynomial to each pair of vertices $(j, \tilde{j})$ of $Q$. 

\begin{center}

 $F_j^Q := \displaystyle \prod_{b_{ij}+b_{\tilde{i}j }>0} x_i^{b_{ij}+b_{\tilde{i}j} } + \prod_{b_{ij}+b_{\tilde{i}j }<0} x_i^{-(b_{ij}+b_{\tilde{i}j})}$

\end{center}

As a result we arrive at the seed $\Sigma_Q := (\{x_1, \ldots , x_n\},\{F_1^Q, \ldots , F_n^Q\})$ associated to $Q$. Of course, this may not be a valid LP seed due to the requirement of irreducibility. We won't always get irreducibility, but the following proposition demonstrates there are plenty of cases where $Q$ does provide a valid LP seed.

\begin{prop}[Proposition 4.7, \cite{wilson2017laurent}]
\label{irreducible}
If $gcd(b_{1j}+b_{\tilde{1}j},\ldots, b_{nj}+b_{\tilde{n}j}) = 1$ then $F_j$ is irreducible in $\mathbb{Z}[x_1,\ldots,x_n]$.

\end{prop}

Note that when flipping between triangulations, if we want double mutation of our quiver to correspond to LP mutation then it is necessary for us to have $\hat{F}_i = F_i$ $ \forall i \in \{1,\ldots, n\}$. This is because the exchange polynomials of the arcs in the triangulations are polynomials (not strictly Laurent polynomials), so the normalisation process needs to be vacuous.

The following proposition from \cite{wilson2017laurent} tells us when LP mutation of a seed $\Sigma_Q$ corresponds to double mutation of $Q$.

\begin{prop}[Proposition 4.8, \cite{wilson2017laurent}]
\label{quiver LP}
Let $Q$ be an anti-symmetric quiver, and $\Sigma_Q = (\mathbf{x},\mathbf{F}^Q)$ its associated seed. Then $$ \mu_i(\{x_1, \ldots , x_n\},\{F_1^Q, \ldots , F_n^Q\}) = (\{x_1,\ldots,\frac{F_i^Q}{x_i} ,\ldots, x_n\},\{F_1^{\mu_i\circ\mu_{\tilde{i}}(Q)},\ldots, F_n^{\mu_i\circ\mu_{\tilde{i}}(Q)}\})$$ if the following conditions are satisfied:

\begin{itemize}

\item $(\mathbf{x}, \mathbf{F}^Q)$ is a valid LP seed, i.e. $F_j^{Q}$ is irreducible in $\mathbb{ZP}[\mathbf{x}]$ for each $j \in \{1,\ldots, n\}$.

\item $\hat{F}_i^{Q} = F_i^{Q}$.

\item There is no path $a \rightarrow i \rightarrow \tilde{a}$ for any vertex $a$ of $Q$.

\end{itemize}

\end{prop}

\section{Laminated surfaces and their (laminated) quasi-cluster algebra.}

In \cite{wilson2017laurent} we determined when the quasi-cluster algebra of a bordered surface $(S,M)$ has an LP structure. Specifically, when the boundary segments receive variables then the quasi-cluster algebra $\mathcal{A}(S,M)$ has an LP structure \textit{if and only if} $(S,M)$ is an unpunctured surface. If the boundary segments do not receive variables then the only unpunctured surfaces not admitting an LP structure are the $6$-gon, the cylinder $C_{2,2}$, the M\"obius strip with $4$ marked points, and the torus and the Klein bottle -- both with one boundary component and two marked points.\newline
\indent This classification was obtained by recognising the underlying characteristic preventing a bordered surface $(S,M)$ from having an LP structure. Namely, $(S,M)$ has no LP structure \textit{if and only if} there exists a quasi-triangulation of $(S,M)$ whose corresponding exchange polynomials are not all distinct. \newline
\indent The goal of this paper is to add laminations to the surface in an attempt to modify the exchange polynomials - our desire will then be to concoct a lamination which guarantees the uniqueness of exchange polynomials for any quasi-triangulation.

\subsection{Laminations and shear coordinates}

\begin{defn}
\label{laminationdef}
A \textbf{\textit{lamination}} on a bordered surface (S, M) is a finite collection of non-self-intersecting and pairwise non-intersecting curves in $(S,M)$, considered up to isotopy, and subject to the conditions outlined below -- see Figure \ref{laminations} for an example. Namely, each curve in a lamination must be one of the following:

\begin{itemize}

\item A curve connecting two unmarked points in $\partial S$. Though we do not allow the scenario when this curve is isotopic to a piece of boundary containing one or zero marked points;

\item A curve with one end being an unmarked point in $\partial S$, and whose other end spirals into a puncture;

\item A curve with both ends spiralling into (not necessarily distinct) punctures. We forbid the case when the curve has both ends spiralling into the same puncture, and does not enclose anything else;

\item A two-sided closed curve which does not bound a disk, a once-punctured disk, or a M\"obius strip.

\end{itemize}

\end{defn}

\begin{figure}[H]
\begin{center}
\includegraphics[width=11cm]{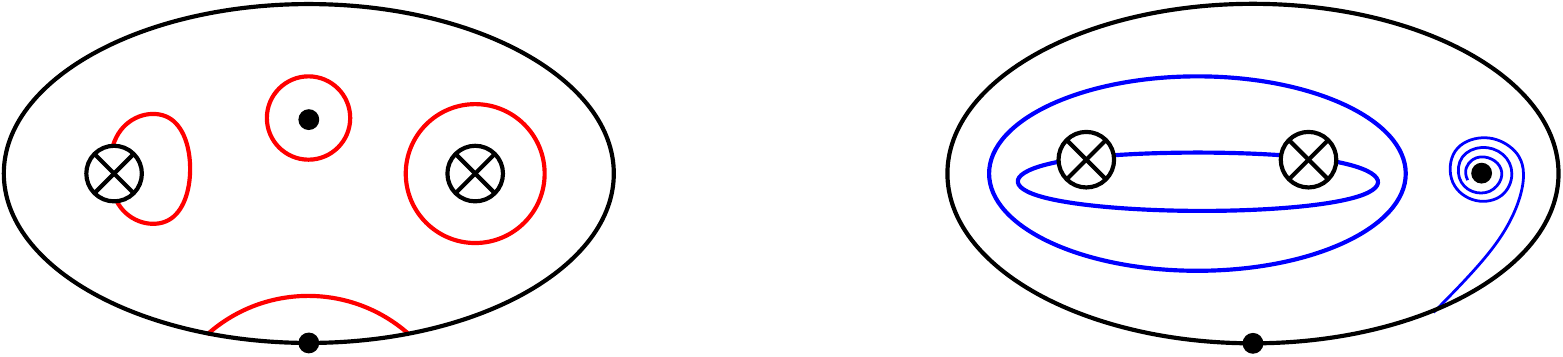}
\caption{None of the curves on the left are considered laminations. All curves on the right are legitimate laminations.}
\label{laminations}
\end{center}
\end{figure}

\begin{rmk}

Note that we have defined laminations in such a way that, when we lift to the double cover, we obtain laminations allowed by Fomin and Thurston \cite{fomin2012cluster}. For technical reasons discussed in Remark \ref{excluding laminations}, we do not allow laminations that bound a M\"{o}bius strip or are one-sided closed curves, even though their lifts would be legitimate laminations in the sense of Fomin and Thurston.

\end{rmk}

We now describe W. Thurston's \textit{shear coordinates} \cite{thurston1988geometry} with respect to a lamination of an ideal-triangulated orientable surface.

\begin{defn}[$S$-shape and $Z$-shape intersections] 
\label{s-z-intersections}

Let $Q_{\gamma}$ be a triangulated quadrilateral with diagonal $\gamma$. Suppose $C$ is a curve intersecting opposite sides of $Q_{\gamma}$ (and does not intersect the boundary of $Q_{\gamma}$ anywhere else). Denote these sides by $\alpha$ and $\beta$. If $\alpha$, $\beta$ and $\gamma$ form an $'S'$ (resp. $'Z'$), then call the intersection of $C$ with $Q_{\gamma}$ an \textit{\textbf{$S$-shape intersection}} (resp. \textit{\textbf{$Z$-shape intersection}}). See Figure \ref{s-z-figure}.

\end{defn}

\begin{defn}[Shear coordinates for ideal triangulations] 
\label{shear}
Let $T$ be an ideal triangulation of an orientable bordered surface $(S,M)$, and $L$ a lamination. Furthermore, let $\gamma$ be an arc of $T$ which is not the folded side of a self-folded triangle, and denote by $Q_{\gamma}$ the quadrilateral of $T$ whose diagonal is $\gamma$. The \textbf{\textit{shear coordinate}}, $b_{T}(L, \gamma)$, of $L$ and $\gamma$, with respect to $T$, is defined as:
$$ b_{T}(L, \gamma) := \# \Big\{ \stackanchor{\text{$S$-shape intersections}}{\text{of $L$ with $Q_{\gamma}$}}\Big\} - \# \Big\{ \stackanchor{\text{$Z$-shape intersections}}{\text{of $L$ with $Q_{\gamma}$}}\Big\}$$

\end{defn}

\begin{figure}[H]
\begin{center}
\includegraphics[width=11cm]{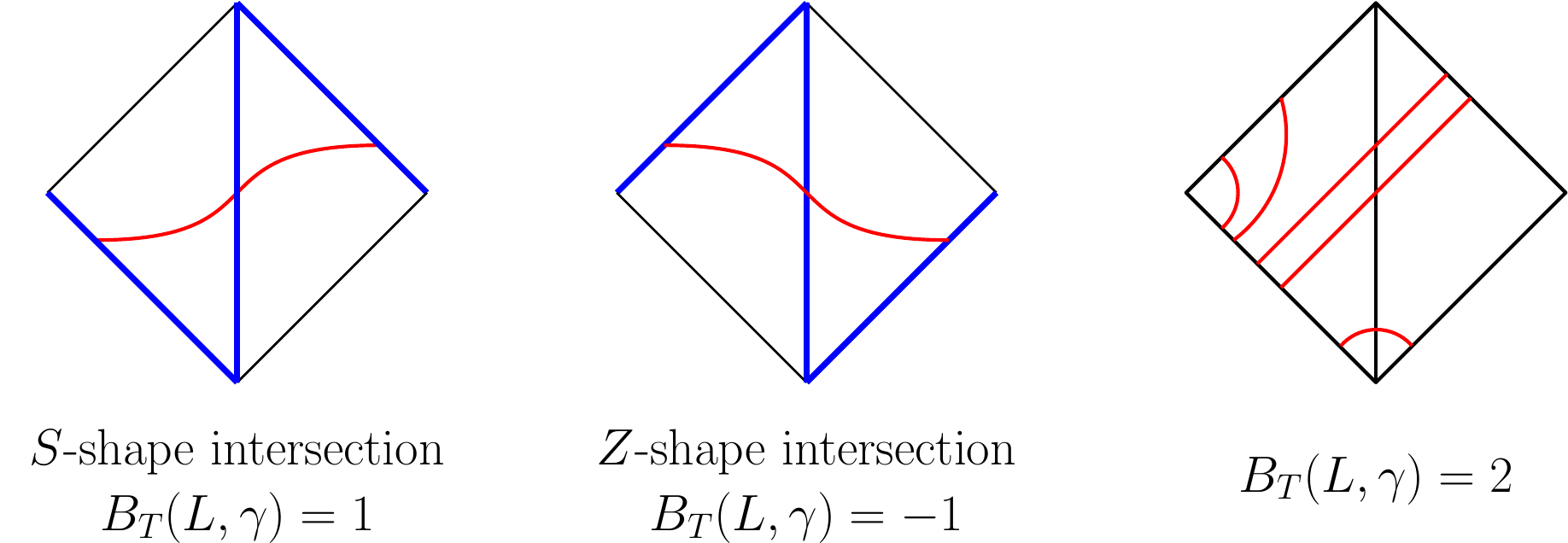}
\caption{S-shape and Z-shape intersections.}
\label{s-z-figure}
\end{center}
\end{figure}

\begin{rmk}

Note that even though a lamination spiralling into a puncture $p$ will intersect any arc incident to $p$ infinitely many times, $b_{T}(L, \gamma)$ will always be finite.

\end{rmk}

We explain below how Fomin and Thurston \cite{fomin2012cluster} extended the notion of shear coordinates to (tagged) triangulations of orientable bordered surfaces.

\begin{defn}[Shear coordinates for triangulations]

Let $T$ be a triangulation and $L$ a lamination. If $L$ spins into a puncture $p$, containing only arcs with notches at $p$, then reverse the direction of spinning of $L$ at $p$, and replace all these notched taggings with plain ones. \newline
\indent Using the rule above we may convert the lamination $L$ of $T$ into a lamination $L_1$ of a triangulation $T_1$, with the property that any notched arc in $T_1$ appears with its plain counterpart. As per usual, denote by $T^{\circ}$ the ideal triangulation associated to $T_1$ -- as hinted by the notation, this is also the ideal triangulation associated to $T$. \newline
\indent Let $\gamma$ be an arc of $T$, and denote by $\gamma^{\circ}$ the corresponding arc in $T^{\circ}$. We define $b_{T}(L, \gamma)$ as follows:

\begin{itemize}

\item If $\gamma^{\circ}$ is not the self-folded side of a triangle in $T^{\circ}$, then define $$b_{T}(L, \gamma):= b_{T^{\circ}}(L_1, \gamma^{\circ}).$$

\item If $\gamma^{\circ}$ is the self-folded side of a triangle in $T^{\circ}$, with puncture $p$, then reverse the direction of spinning of $L_1$ at $p$, and denote this new lamination by $L_2$. Furthermore, let $\beta$ denote the remaining side of the triangle in $T^{\circ}$ that is folded along $\gamma^{\circ}$. We define $$b_{T}(L, \gamma):= b_{T^{\circ}}(L_2, \beta)$$.

\end{itemize}

\end{defn}

\begin{rmk}

For a lamination $L$ of an ideal triangulation $T$, note that if $\gamma$ is the enclosing arc of a puncture $p$, then $b_{T}(L, \gamma)$ does not depend on the direction $L$ is spinning at any other puncture enclosed in any other monogon. In Figure \ref{laminationweights} we illustrate the algorithm of how to compute shear coordiantes of triangulations using ideal triangulations.

\end{rmk}

\begin{figure}[H]
\begin{center}
\includegraphics[width=13cm]{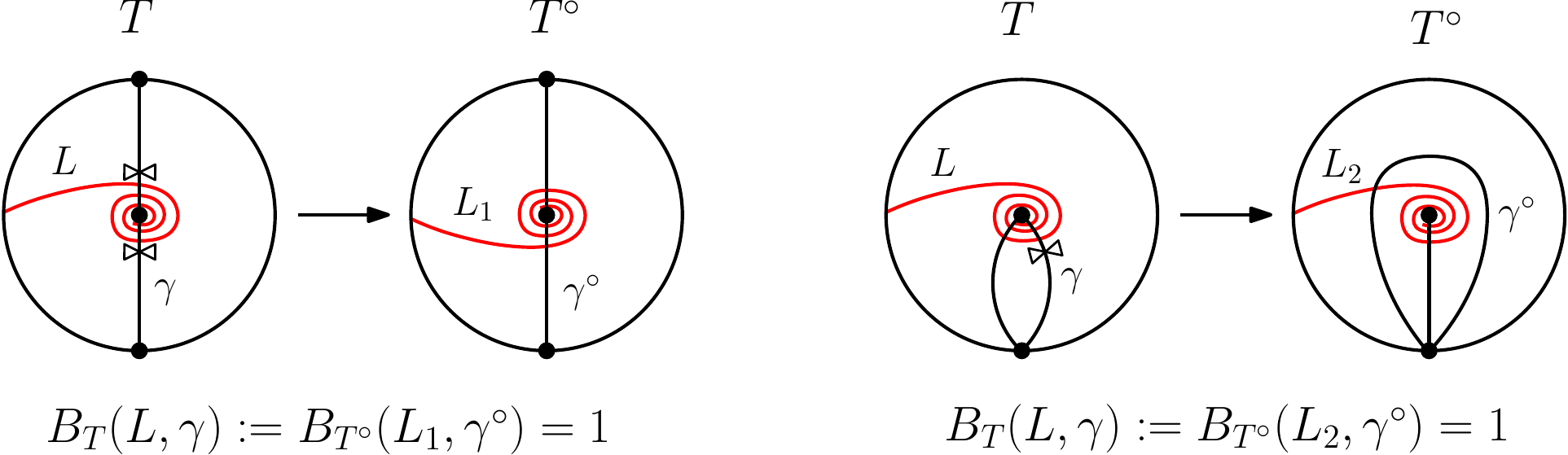}
\caption{Examples of how shear coordinates are defined for triangulations containing tagged arcs.}
\label{laminationweights}
\end{center}
\end{figure}

\begin{defn}

A \textbf{\textit{multi-lamination}}, $\mathbf{L}$, of a bordered surface $(S,M)$ consists of a finite collection of laminations of $(S,M)$.

\end{defn}

\begin{defn}(Adjacency quiver of laminated orientable surfaces).
Let $T$ be a triangulation of an orientable bordered surface $(S,M)$. For a multi-lamination $\textbf{L}$ of $(S,M)$ we extend the adjacency quiver $Q_T$ to a quiver $Q_{T, \mathbf{L}}$ as follows:

\begin{itemize}

\item For each lamination $L_i$ in $\mathbf{L}$ add a corresponding vertex to $Q_T$. Abusing notation, we shall also denote this vertex by $L_i$.

\item Let $\gamma$ denote a vertex in $Q_T$ and its corresponding arc in $T$. If $b_{T}(L_i, \gamma)$ is positive (resp. negative) add $|b_{T}(L_i, \gamma)|$ arrows $L_i \rightarrow \gamma$ (resp. $L_i \leftarrow \gamma$).

\end{itemize}

\end{defn}

\begin{prop}[Theorem 13.5, \cite{fomin2012cluster}]
\label{existing quiver flip}
Let $\mathbf{L}$ be a multi-lamination of an orientable bordered surface $(S,M)$. Then for any arc $\gamma$ in a triangulation $T$, $\mu_{\gamma}(Q_{T, \mathbf{L}}) = Q_{\mu_{\gamma}(T), \mathbf{L}}$.

\end{prop}

\begin{defn}(Quiver of laminated surfaces).
Let $T$ be a triangulation of a bordered surface $(S,M)$. For a multi-lamination $\mathbf{L}$ of $(S,M)$ let $\overline{\mathbf{L}}$ denote the lifted lamination on $\overline{(S,M)}$ -- the orientable double cover. We define the quiver associated to $(S,M,\mathbf{L})$ to be $Q_{\overline{T}, \mathbf{\overline{L}}}$.

\end{defn}

\begin{rmk}

Each lamination $L_i$ of $\mathbf{L}$ lifts to two laminations in $\overline{\mathbf{L}}$ -- abusing notation we shall denote these lifts by $L_i$ and $\tilde{L}_i$. Note that there is a choice to which of these lifts are marked $L_i$ and $\tilde{L}_i$. Moreover, if $L_i \in \mathbf{L}$ consists of more than one connected component then, as there is a choice of marking for each connected component, $Q_{\overline{T}, \mathbf{\overline{L}}}$ is not uniquely determined by $(S,M,\mathbf{L})$; it relies on the choice of which lifts of $L_i \in \mathbf{L}$ are marked $L_i$ and $\tilde{L}_i$. However, it will be important later on to note that the quantity $$b_{T}(L_i, \gamma) + b_{T}(\tilde{L}_i, \gamma)$$ is invariant under this choice.

\end{rmk}

\begin{prop}

For each triangulation $T$ of a multi-laminated bordered surface $(S,M,\mathbf{L})$, $Q_{\overline{T}, \mathbf{\overline{L}}}$ is an anti-symmetric quiver.

\end{prop}

\begin{proof}

We have already verified anti-symmetry between vertices corresponding to lifted arcs. It remains to check anti-symmetry for the rest of the quiver. \newline For each lamination $L_i$ of $\mathbf{L}$ we have two vertices in $Q_{\overline{T}, \mathbf{\overline{L}}}$ corresponding to the lifted versions of $L_i$. Abusing notation, we shall denote these vertices by $L_i$ and $\tilde{L}_i$. If the lift $L_i$ cuts through a triangulated quadrilateral in an $'S'$ (resp. $'Z'$) shape, the other lift $\tilde{L}_i$ cuts through the twin quadrilateral in a $'Z'$ (resp. $'S'$) shape. Hence we get an arrow $L_i \rightarrow \gamma$ (resp. $L_i \leftarrow \gamma$) \textit{if and only if} there is an arrow $\tilde{L}_i \leftarrow \tilde{\gamma}$ (resp. $\tilde{L}_i \rightarrow \tilde{\gamma}$). Furthermore, by definition of this quiver, there are no arrows between vertices corresponding to lifted laminations. In particular, there are no arrows $L_i \rightarrow \tilde{L_i}$ for any $i$.

\end{proof}

To utilise anti-symmetric quivers as much as possible, in certain triangulations, it will be helpful to contemplate an alternative choice of flip that our definitions had previously forbidden. This will involve considering \textit{traditional triangulations}, which are defined below. We follow up this definition with a discussion on how this notion of triangulation arises.

\begin{defn}

A \textit{\textbf{traditional triangulation}} consists of a maximal collection of pairwise compatible arcs, containing no arcs that cut out a once punctured monogon.

\end{defn}

\begin{figure}[H]
\begin{center}
\includegraphics[width=11cm]{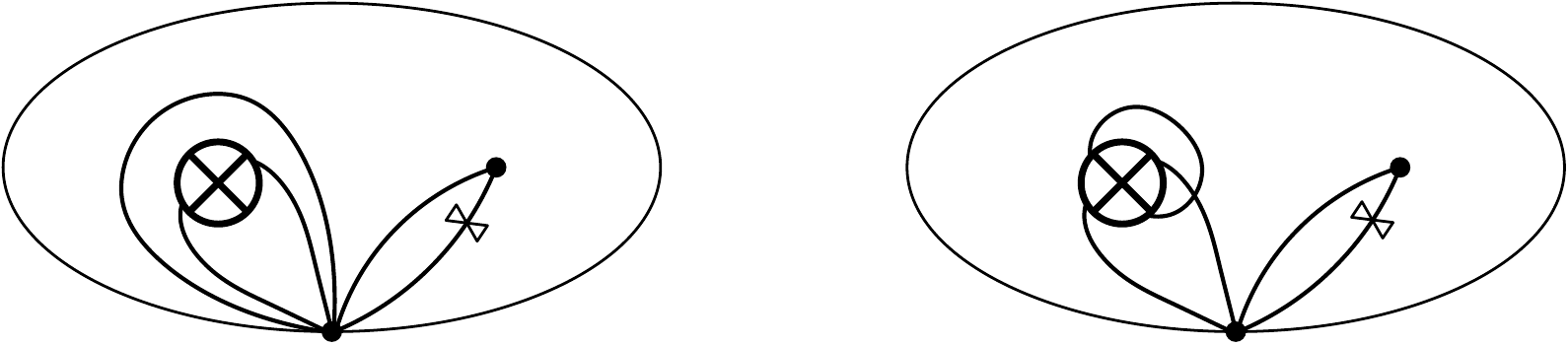}
\caption{To emphasise the differences, we provide an example of a traditional triangulation (left) and a triangulation (right).}
\label{example of traditional}
\end{center}
\end{figure}

\begin{rmk}

Note that traditional triangulations differ to triangulations in the sense that we allow arcs bounding a M\"{o}bius strip $M_1$, but do not allow one-sided closed curves -- see Figure \ref{example of traditional}.

\end{rmk}

Let $T$ be a triangulation of $(S,M)$ and $\alpha$ an arc in $T$. Proposition \ref{flip} tells us there exists a unique quasi-arc $\alpha'$ such that $T\cup\{\alpha'\}\setminus \{\alpha\}$ is a quasi-triangulation. However, when $\alpha'$ is a one-sided closed curve there is an alternative flip of $\alpha$ we can consider (which is forbidden under our current set-up). We shall describe this alternative flip and explain how it fits in with mutation of anti-symmetric quivers. Firstly, note that by Definition \ref{newcompatibilitydef}, if $\alpha'$ is a one-sided closed curve then it intersects precisely one arc $\beta \in T$. There exists a unique arc $\alpha^{*} \notin T$ enclosing $\alpha'$ and $\beta$ in $M_1$. If we choose to flip $\alpha$ to $\alpha^{*}$ (instead of $\alpha'$) then we will arise at a traditional triangulation, see Figure \ref{alternativeflip}. \newline

\begin{figure}[H]
\begin{center}
\includegraphics[width=13cm]{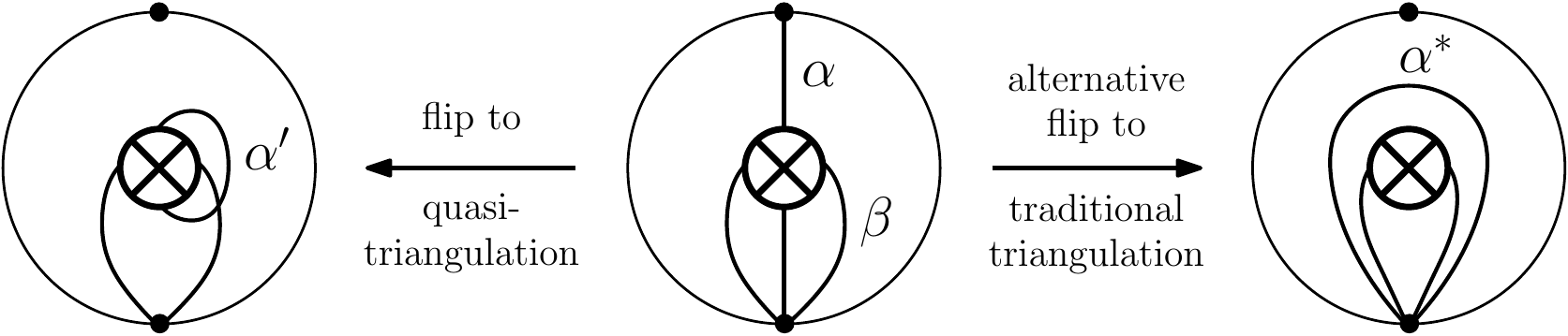}
\caption{On the left we show the flip of the arc $\alpha$ resulting in a quasi-triangulation. The alternative flip to a traditional triangulation is shown on the right.}
\label{alternativeflip}
\end{center}
\end{figure}

\indent In fact, analogous to the proof of Proposition \ref{flip}, for any triangulation $T$ and any arc $\gamma$ of $T$, there exists a unique arc $\gamma' \neq \gamma$ such that $T\cup\{\gamma'\}\setminus \{\gamma\}$ is a traditional triangulation. Turning our attention back to this alternative flip, consider the lift $\overline{T}$, of $T$, in the double cover $\overline{(S,M)}$. If we flip both of the lifts $\alpha$, $\tilde{\alpha}$ in $\overline{(S,M)}$ and take the $\mathbb{Z}_2$-quotient we will obtain precisely $T' := T\cup\{\alpha^*\}\setminus \{\alpha\}$. Therefore the existing theory of cluster algebras from surfaces, Proposition \ref{existing quiver flip} to be precise, tells us that
$$ \mu_{\alpha}\circ \mu_{\tilde{\alpha}}(Q_{\overline{T},\mathbf{\overline{L}}}) = Q_{\overline{T}',\mathbf{\overline{L}}}.$$

We conclude this short discussion with the following proposition. 

\begin{prop}
\label{traditional quiver mutation}
Fix a multi-lamination $\mathbf{L}$ of a bordered surface $(S,M)$. Let $T$ be a triangulation of $(S,M)$. Then for any arc $\gamma$ in $T$, flipping $\gamma$ (with respect to traditional triangulations) corresponds to double mutation of the anti-symmetric quiver $Q_{\overline{T}, \mathbf{\overline{L}}}$, at the vertices corresponding to the two lifts of $\gamma$.

\end{prop}

\begin{proof}

By Definition \ref{quasitriangulationdef}, since $\gamma$ is not bounded by an arc enclosing a M\"{o}bius strip, $M_1$, then the interiors of the flip regions containing the lifts of $\gamma$ are disjoint. Therefore, flipping $\gamma$ in $(S,M)$ corresponds to simultaneously flipping both of the lifts in $\overline{(S,M)}$. Finally, by the theory of orientable surfaces, flipping an arc in the double cover corresponds to mutating the vertex in $Q_{\overline{T}, \mathbf{\overline{L}}}$ representing that arc.

\end{proof}

\begin{rmk}

In general, when considering traditional triangulations, mutation does not preserve the anti-symmetric property of a quiver. In particular, after performing the flip of $\alpha$ to $\alpha^*$ discussed above, the corresponding quiver will contain (two) arrows between $\beta$ and $\tilde{\beta}$, depriving it of anti-symmetry -- see Figure \ref{brokensymmetry}. (Flips amongst triangulations will of course preserve the anti-symmetric property.)

\end{rmk}

\begin{figure}[H]
\begin{center}
\includegraphics[width=13cm]{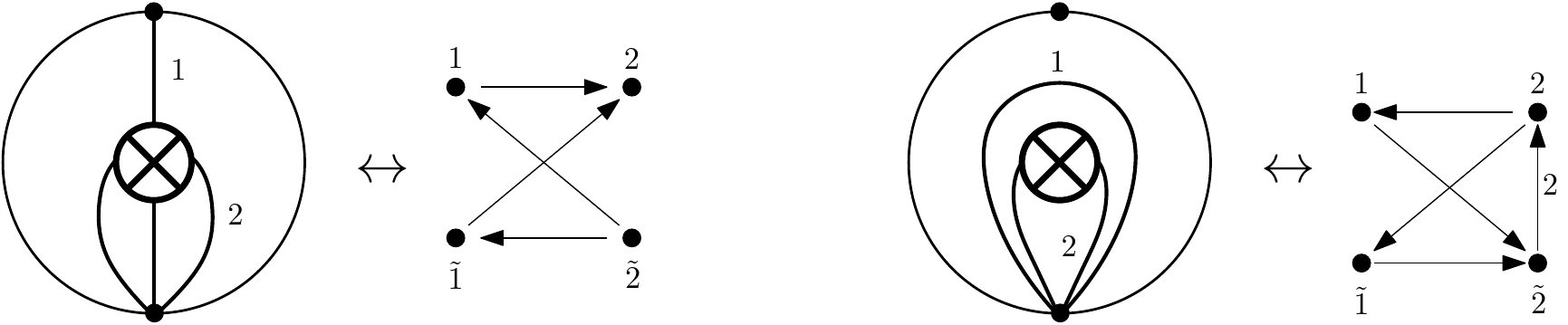}
\caption{Performing a flip to an arc bounding $M_1$ breaks anti-symmetry.}
\label{brokensymmetry}
\end{center}
\end{figure}

We have already seen that the lambda length of a quasi-arc can be viewed as a formal variable. Our goal now will be to introduce new variables, called \textit{laminated lambda lengths}, that take into account the multi-lamination $\mathbf{L}$ on the surface, not just the geometry. The procedure we shall use follows the approach taken by Fomin and Thurston \cite{fomin2012cluster}; it will involve rescaling the lambda length of each quasi-arc $\gamma$ with respect to the intersection numbers of $\gamma$ with $\mathbf{L}$. As it stands, this notion is currently ill-defined. Namely, when $\mathbf{L}$ spirals into a puncture $p$, it will intersect any arc incident to $p$ infinitely many times. To bypass this problem we shall open up the punctures.

\subsection{Opening the surface}

\begin{defn}

Let (S,M) be bordered surface and $P \subseteq M\setminus \partial S$ be a set of punctures. The \textit{\textbf{(partially) opened bordered surface}}, $(S_P,M_P)$ is defined as follows. $S_P$ is obtained from $S$ by removing a small open neighbourhood around each $p\in P$. Furthermore, to each newly created boundary component, $C_p$, we add a marked point $m_p$. We then set $M_P := (M\setminus P) \cup \{m_p\}_{p\in P}$.

\end{defn}

It is crucial to note that our treatment of a partially opened bordered surface $(S_P,M_P)$ throughout this paper will differ from that of a bordered surface. I.e. we will care whether a boundary segment was the consequence of opening a puncture. In particular, the set of quasi-arcs of $A^{\otimes}(S_P,M_P)$ is defined as before, except now:

\begin{itemize}

\item We allow arcs to be notched at $m_p$ for $p \in P$.

\item We don't allow an arc to cut out a monogon containing $C_p$ for $p \in P$.

\end{itemize}

\noindent With this in mind there is a canonical projection map

$$ \kappa_P : A^{\otimes}(S_P,M_P) \longrightarrow A^{\otimes}(S,M)$$

\noindent that amounts to collapsing each boundary component $C_p$ in $(S_P,M_P)$. Any quasi-arc $\overline{\gamma} \in A^{\otimes}(S_P,M_P)$ that projects to a quasi-arc $\gamma \in A^{\otimes}(S,M)$ will be referred to as a \textit{lift} of $\gamma$ -- we give an example of this in Figure \ref{openingpunctures}.

\begin{figure}[H]
\begin{center}
\includegraphics[width=12cm]{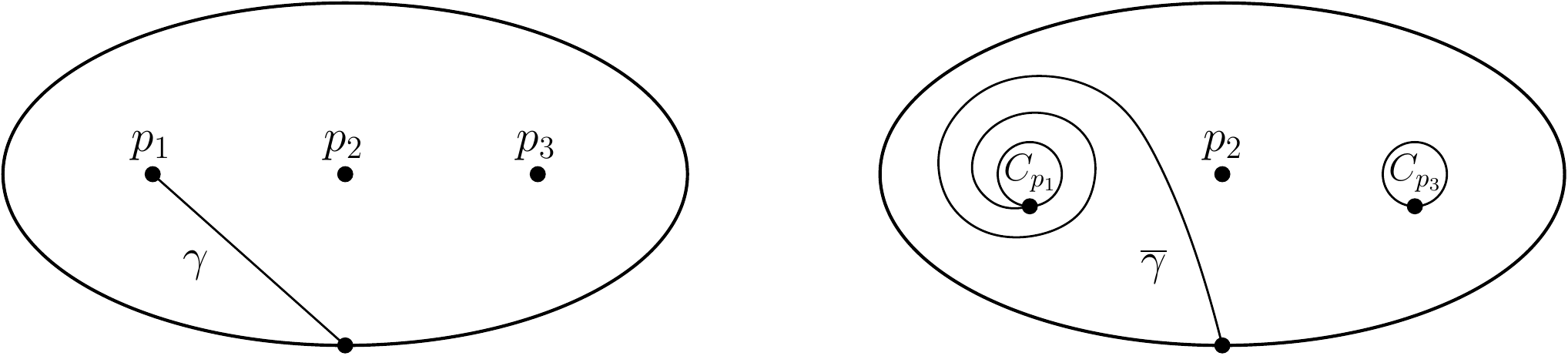}
\caption{Here $P = \{p_1,p_3\}$ meaning we have opened the punctures $p_1$ and $p_3$, but not $p_2$. The arc $\overline{\gamma} \in (S_P,M_P)$ is a lift of $\gamma \in (S,M)$.}
\label{openingpunctures}
\end{center}
\end{figure}

\begin{defn}

The \textit{\textbf{opened bordered surface}}, $(S^*,M^*)$, is the result of opening up all the punctures. Note that $\kappa_{M\setminus \partial S} := \kappa^*$ factors through every other map $\kappa_P$.

\end{defn}

We now describe what will be the overarching notion of a Teichm\"uller space with regards to opening surfaces. For those familiar with the work of Fomin and Thurston \cite{fomin2012cluster}, our definitions of Teichm\"{u}la space are analogous. Moreover, the way we define the signed and lambda lengths of (tagged) arcs is exactly the same. For one-sided closed curves the length of arc is independent of the choice of horocycles chosen, and is simply defined as the hyperbolic length of the curve; the lambda length is given in Definition \ref{lambda length of quasi-arcs}.

\begin{defn}

A \textit{\textbf{decorated set of punctures}}, $\tilde{P}$, is a subset $P \subseteq M\setminus \partial S$ together with a choice of 'orientation' on $C_p$ for each $p \in P$.

\end{defn}

\begin{rmk}

To clarify, our usage of 'orientation' means that we are choosing a direction of flow on each boundary component $C_p$. Being on a non-orientable surface just means that we cannot globally speak about whether this flow is clockwise or counter-clockwise.

\end{rmk}

\begin{defn}

For a decorated set of punctures $\tilde{P}$ we define the \textit{\textbf{partially opened Teichm\"uller space}}, $\mathcal{T}_{\tilde{P}}(S_P,M_P)$, to be the space of all finite volume, complete hyperbolic metrics on $S_P\setminus (M\setminus P)$ with geodesic boundary, up to isotopy. \newline
The \textit{\textbf{decorated partially opened Teichm\"uller space}}, $\tilde{\mathcal{T}}_{\tilde{P}}(S_P,M_P)$, consists of the same metrics as in $\mathcal{T}_{\tilde{P}}(S_P,M_P)$, except now they are considered up to isotopy relative to $\{m_p\}_{p \in P}$. Additionally, there is a choice of horocycle around each point in $M \setminus P$.
\end{defn}

Given a decorated set of punctures $\tilde{P}$ and $\sigma \in \mathcal{T}_{\tilde{P}}(S_P,M_P)$, for any quasi-arc $\gamma \in A^{\otimes}(S_P,M_P)$ we can associate a unique non-intersecting geodesic $\gamma_{\sigma}$ on $S_P$. If $\gamma$ is a one-sided closed curve then $\gamma_{\sigma}$ is just the usual geodesic representative of $\gamma$ with respect to $\sigma$. If $\gamma$ is an arc we define $\gamma_{\sigma}$ as follows: 

\begin{itemize}

\item For an endpoint of $\gamma$ not in $P$, $\gamma_{\sigma}$ runs out to the corresponding cusp.

\item For an endpoint of $\gamma$ in $P$ that is tagged plain $\gamma_{\sigma}$ should spiral (infinitely) around $C_p$ in the chosen orientation of $C_p$. Otherwise the endpoint is notched, and it should spiral \textit{against} the chosen orientation -- an example of this is given in Figure \ref{infinitegeodesic}.

\end{itemize}

\begin{figure}[H]
\begin{center}
\includegraphics[width=12cm]{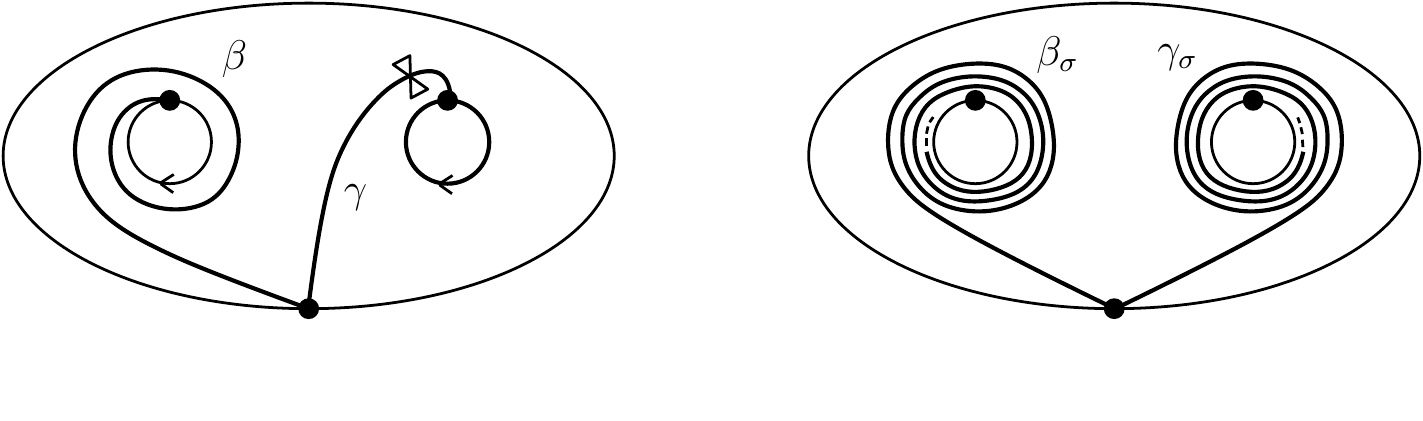}
\caption{Arcs $\beta$ and $\gamma$ and their associated geodesics $\beta_{\sigma}$ and $\gamma_{\sigma}$.}
\label{infinitegeodesic}
\end{center}
\end{figure}

\begin{defn}

Let $\tilde{P}$ be a decorated set of punctures and $\sigma \in \mathcal{T}_{\tilde{P}}(S_P,M_P)$. For each $p \in P$ consider a small segment of the horocycle originating from $m_p$ which is both perpendicular to $C_p$ and all geodesics $\gamma_{\sigma}$ that spiral into $C_p$ in the chosen orientation of $C_p$. Such a segment is called the \textit{\textbf{perpendicular horocycle segment}}, and is denoted $h_p$.

\end{defn}

\begin{defn}[Length of plain arcs on opened surface]
\label{arc length}
Let $\sigma \in \mathcal{\tilde{T}}_{\tilde{P}}(S_P,M_P)$. We will eventually define the lengths of all plain arcs $\gamma$ in $(S_P,M_P)$, however, for now we shall only concentrate on those whose ends twist sufficiently far around opened punctures $C_p$ in the direction consistent with the chosen orientation of each $C_p$. \newline
At the ends of $\gamma_{\sigma}$ that spiral around an open puncture $C_p$ there will be infinitely many intersections with the horocyclic segment $h_p$ at $m_p$. We describe how we pick one of these intersections:  \newline

For $\gamma$ with endpoints $m_p$ and $m_q$ (that twists sufficiently far around the corresponding boundaries) choose the unique intersections between $\gamma_{\sigma}$ and each horocyclic segment, $h_p$ and $h_q$, such that the path running from

\begin{itemize}

\item $m_p$ to an intersection of $h_p$ with $\gamma_{\sigma}$ (along $h_p$), then from

\item $\gamma_{\sigma}$ to an intersection of $\gamma_{\sigma}$ with $h_q$ (along $\gamma_{\sigma}$), then from

\item $h_q$ to $m_q$ (along $h_q$)

\end{itemize}

is homotopic to the original arc $\gamma$. In the less complicated case of $\gamma$ not having both endpoints in $P$, we leave $\gamma_{\sigma}$ unmodified at the ends not in $P$, and, as usual, choose the unique intersection between $\gamma_{\sigma}$ and the corresponding horocycle. The \textit{\textbf{length}} of $\gamma$, $l_{\sigma}(\gamma)$, is defined to be the signed distance of $\gamma_{\sigma}$ between the horocycles at its endpoints (with respect to the intersections described above) -- an illustration of this is given in Figure \ref{perphorocycle}.

This definition is extended to all plain arcs (not just those twisting sufficiently far around open punctures) by defining,

\begin{equation}
\label{twist1}
l_{\sigma}(\psi^{\pm 1}_p(\gamma)) := \pm n_p(\gamma) l_{\sigma}(p) + l_{\sigma}(\gamma),
\end{equation}

\noindent where $\psi_p(\gamma)$ denotes the twist of $\gamma$ around $C_p$ in the direction consistent with $C_p$'s orientation ($\psi_p^{- 1}(\gamma)$ being the twist against $C_p$'s orientation); $n_p(\gamma)$ is the number of endpoints $\gamma$ has at $m_p$; and $l_{\sigma}(p)$ is the length of $C_p$ if $p \in P$, and $0$ otherwise.

\end{defn}

\begin{rmk}

In Definition \ref{arc length}, (\ref{twist1}) is well defined as the distance between successive intersections of $\gamma_{\sigma}$ with $h_p$ is $l_{\sigma}(p)$. A proof of this can be found in [Lemma 10.7, \cite{fomin2012cluster}].

\end{rmk}

\begin{figure}[H]
\begin{center}
\includegraphics[width=10cm]{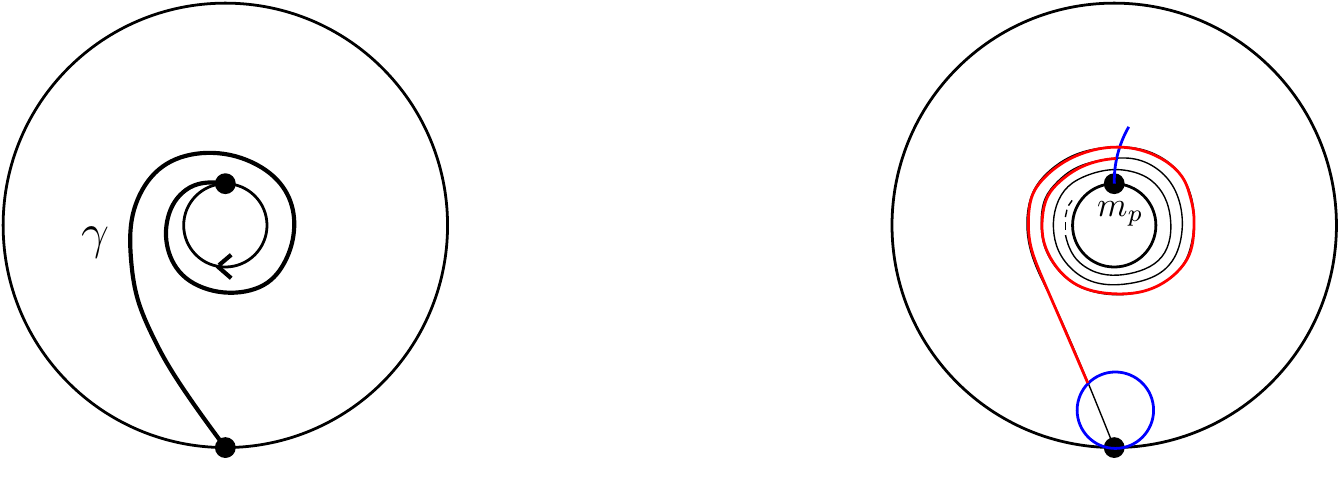}
\caption{On the left we draw an arc $\gamma$. On the right we draw the associated geodesic $\gamma_{\sigma}$; the perpendicular horocyclic segment at $m_p$; and a horocycle around the other marked point of $\gamma$. The length of $\gamma$, $l_{\sigma}(\gamma)$, is the length of the red part of $\gamma_{\sigma}$.}
\label{perphorocycle}
\end{center}
\end{figure}

\begin{defn}

Let $\tilde{P}$ be a decorated set of punctures and $\sigma \in \mathcal{\tilde{T}}_{\tilde{P}}(S_P,M_P)$. For each $p \in P$ consider the point $\overline{m}_p$ on $C_p$ that is a (signed) distance $v(p) : = 2ln|\lambda(p) - \lambda(p)^{-1}|$ from $m_p$ in the direction against the orientation of $C_p$.
The \textit{\textbf{conjugate perpendicular horocycle segment}}, $\overline{h}_p$, is the segment of the horocycle originating from $\overline{m}_p$ which is both perpendicular to $C_p$ and all geodesics $\gamma_{\sigma}$ that spiral into $C_p$ \textit{against} the chosen direction. 

\end{defn}

\begin{defn}[Length of arcs on opened surface]
\label{tagged length}
Let $\sigma \in \mathcal{\tilde{T}}_{\tilde{P}}(S_P,M_P)$ and $\gamma$ is an arc whose endpoints twist sufficiently far around open punctures; namely, if $\gamma$ is tagged plain at $m_p$ then it must twist sufficiently far in the direction of $C_p$'s orientation, and if it is notched it must twist sufficiently far \textit{against} $C_p$'s orientation.

For such an arc $\gamma$, the length $l_{\sigma}(\gamma)$ is defined as in Definition \ref{arc length}, except now, when there is a notched endpoint at $m_p$, at the corresponding endpoint of $\gamma$ we consider the intersection of $\gamma_{\sigma}$ with the conjugate perpendicular horocycle segment $\overline{h}_p$. In particular, the way we choose horocycle intersections is the the same as Definition \ref{arc length}, except for notched endpoints at $m_p$ we now run from $m_p$ to $\overline{m}_p$ (along $C_p$) \textit{against} the orientation of $C_p$, and then run from $\overline{m}_p$ to an intersection of $\overline{h}_p$ with $\gamma_{\sigma}$ (along $\overline{h}_p$) -- an illustration of this is given in Figure \ref{conjperohorocycle}.

The definition is again extended to all arcs by using:

\begin{equation}
\label{twist2}
l_{\sigma}(\psi^{\pm 1}_p(\gamma)) := \pm n_p(\gamma) l_{\sigma}(p) + l_{\sigma}(\gamma)
\end{equation}

Here $\psi_p$ and $l_{\sigma}(p)$ are as in (\ref{twist1}). However, we extend $n_p(\gamma)$ to all arcs by setting it as minus (resp. plus) the number of notched (resp. plain) ends of $\gamma$ at $m_p$.
\end{defn}

\begin{figure}[H]
\begin{center}
\includegraphics[width=10cm]{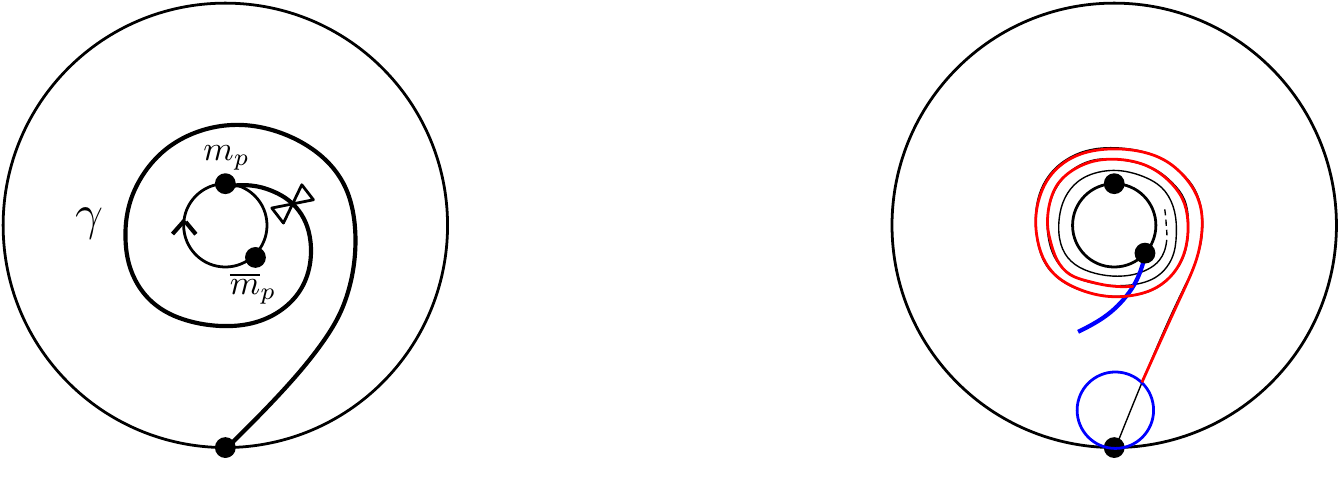}
\caption{On the left we draw an arc $\gamma$ notched arc at $m_p$, and also indicate the point $\overline{m}_p$ which is a distance $v(p)$ away from $m_p$. On the right we draw the associated geodesic $\gamma_{\sigma}$; the conjugate perpendicular horocycle segment $\overline{h}_p$; and a horocycle around the other marked point of $\gamma$. The length of $\gamma$, $l_{\sigma}(\gamma)$, is the length of the red part of $\gamma_{\sigma}$.}
\label{conjperohorocycle}
\end{center}
\end{figure}

\begin{defn}[Length of one-sided closed curves]

Let $\sigma \in \tilde{\mathcal{T}}_{\tilde{P}}(S_P,M_P)$. If $\gamma$ is a one-sided closed curve then we denote by $l_{\sigma}(\gamma)$ the hyperbolic length of the geodesic representation of $\gamma$ in $\sigma$.

\end{defn}

\begin{defn}
\label{lambda length of quasi-arcs}
Let $\sigma \in \tilde{\mathcal{T}}_{\tilde{P}}(S_P,M_P)$. We define the lambda length of a quasi-arc $\gamma$ in $(S_P,M_P)$ as:

\[   
\lambda_{\sigma}(\gamma) = 
     \begin{cases}
       e^{\frac{l_{\sigma}(\gamma)}{{2}}},& \text{if $\gamma$ is an arc,}\\
       2\text{sinh}(\frac{l_{\sigma}(\gamma)}{{2}}),& \text{if $\gamma$ is a one-sided closed curve,}\\
     \end{cases}
\]

In addition to this, for each puncture $p$ of $(S,M)$, we define $\lambda_{\sigma}(p) := e^{\frac{l_{\sigma}(p)}{{2}}}$.

\end{defn}

\begin{rmk}
\label{non-normalised structure}
Let us fix a lift $\overline{\gamma} \in (S^{*},M^{*})$ for each arc $\gamma$ in $(S,M)$. [Corollary 10.16, \cite{fomin2012cluster}] tells us that for each triangulation $T$ of $(S,M)$, the cluster $$\mathbf{x}(T) := \{ \lambda(\overline{\gamma}) | \gamma \in T\}$$ may be viewed as a set of algebraically independent variables. Furthermore, [Theorem 11.1, \cite{fomin2012cluster}] reveals that the exchange relations between these clusters are the relations of the corresponding flips on the pre-opened surface $(S,M)$, that have been rescaled at the situations where a flip region in $(S,M)$ has not lifted to a flip region in $(S^{*},M^{*})$. In the terminology of \cite{fomin2012cluster}, this collection of clusters, together with the corresponding rescaled exchange relations, form a \textit{non-normalised exchange pattern} on $E^{\circ}(S,M)$.

\end{rmk}

\subsection{Transverse measure and tropical lambda lengths}

With regards to defining a notion of length (\textit{laminated lambda length}) that takes into account the laminations as well as the geometry, we introduce \textit{tranverse measures} and \textit{tropical lambda lengths} for quasi-arcs on opened surfaces. Those familiar with \cite{fomin2012cluster} should note that, for (tagged) arcs, our definitions are inherited from there. For one-sided closed curves the transverse measure is simply defined as the intersection number between a given lamination; the tropical lambda length is defined in the same way as arcs. \newline \indent The laminated lambda length of a lifted arc will then be defined as a rescaling of lambda length by the tropical lambda length.

\begin{defn}

A \textit{\textbf{lifted lamination}}, $L^{*}$, of $(S^*,M^*)$ consists of a choice of orientation on each opened puncture $C_p$ together with a finite number of non-intersecting curves, with endpoints in $\partial S^* \setminus M^*$, considered up to isotopy relative to $M^*$. We forbid the following types of curves:

\begin{itemize}

\item one-sided closed curves;

\item two-sided closed curves that bound a disk, a M\"obius strip, or a disk containing a single opened puncture;

\item curves with endpoints in $\partial S^*$ which are isotopic to a piece of boundary containing one or zero marked points.

\end{itemize}

\end{defn}

\begin{rmk}

Observe that we can construct a canonical projection map taking lifted laminations, $L^*$, of $(S^*,M^*)$ to laminations, $L$, of $(S,M)$. Namely, $L$ is obtained from $L^*$ by closing the opened punctures and demanding that endpoints of $L^*$ which end on opened punctures, $C_p$, now spiral around $p$ in the direction \textit{opposite} to the orientation chosen on $C_p$ (with respect to $L^*$). The reason why we demand the spiralling to oppose the orientation on $C_p$ is to produce equation (\ref{transverse twist}) - if the orientation agreed we would have to replace '$\pm$' with '$\mp$' on the RHS.

\end{rmk}

\begin{defn}[Transverse measures for plain arcs]
\label{arc transverse}
Let $L^*$ be a lifted lamination of an opened surface $(S^*,M^*)$ and let $\gamma$ be a plain arc or a boundary segment of $(S^*,M^*)$. The \textit{\textbf{transverse measure}} of $\gamma$ with respect to $L^*$ is the integer $l_{L^*}(\gamma)$ defined as follows:

\begin{itemize}

\item If $\gamma$ does not have ends at any $m_p$ then $l_{L^*}(\gamma)$ is the minimal number of intersection points between $L^*$ and any arcs homotopic to $\gamma$.

\item If $\gamma$ has one or two ends at opened punctures, and $\gamma$ twists sufficiently far around these opened punctures compared to $L^*$, then $l_{L^*}(\gamma)$ is again defined to be the minimal number of intersection points between $L^*$ and any arcs homotopic to $\gamma$. By '$\gamma$ twists sufficiently far around $C_p$ compared to $L^*$' we mean that $\gamma$ wraps around $C_p$ more than $L^*$ does with respect to the orientation on $C_p$. \newline  We extend this definition to all plain arcs, not just to those twisting sufficiently far, by setting:

\begin{equation}
\label{transverse twist}
l_{L^*}(\psi^{\pm 1}_p(\gamma)) := \pm n_p(\gamma)l_{L^*}(p) + l_{L^*}(\gamma)
\end{equation}

Here $\psi_p$ and $n_p$ are as in Definition \ref{tagged length}, although it is key to note $\psi_p$ is now defined with respect to the orientation on each $C_p$ coming from $L^*$. $l_{L^*}(p)$ is the number of intersections of $L^*$ with $C_p$.
\end{itemize}

\end{defn}

\begin{defn}[Transverse measures for all arcs]
\label{tagged transverse}
For plain arcs $\gamma$, $l_{L^*}(\gamma)$ is defined as in Definition \ref{arc transverse}. For an arc $\gamma$ which is notched at an endpoint $m_p$, and is such that $L^*$ twists sufficiently far around these opened punctures compared to $\gamma$, define $l_{L^*}(\gamma)$ to be the minimal number of intersection points between $L^*$ and (any arcs homotopic to) $\gamma$, plus $l_{L^*}(p)$. By `$L^*$ twists sufficiently far around $C_p$ compared to $\gamma$' we mean that $L^*$ wraps around $C_p$ more than $\gamma$ does with respect to the orientation on $C_p$ -- note that the requirements of `sufficient wrapping' for notched arcs is the \textit{opposite} of those demanded for plain arcs in Definition \ref{arc transverse}. \newline
\indent We extend the definition to all arcs using equation (\ref{transverse twist}), defined in Definition \ref{arc transverse}.
\end{defn}

\begin{rmk}

The additional term $l_{L^*}(p)$ appearing for notched arcs in the above definition ensures that the laminated lambda lengths (defined later on in Definition \ref{laminated lambda length}) of arcs in the punctured digon satisfy the same exchange relation (2) appearing in Figure \ref{combinatorialflips}.

\end{rmk}

\begin{defn}

For a one-sided closed curve $\gamma$ we define $l_{L^*}(\gamma)$ as the minimal number of intersection points between $L^*$ and any one-sided closed curves homotopic to $\gamma$.

\end{defn}

\begin{defn}[Tropical semi-field associated with a multi lamination]

Let $\mathbf{L}$ be a multi lamination of a bordered surface $(S,M)$. For each lamination $L_i$ in $\mathbf{L}$ we introduce a variable $q_i$. We consider the tropical semifield $\mathbb{P}_{\mathbf{L}}$ over these variables. More specifically, $\mathbb{P}_{\mathbf{L}} := \text{Trop}(q_i : i \in I)$. Note that $I$ is just the indexing set for the laminations $L_i$ in $\mathbf{L}$.

\end{defn}

\begin{defn}[Tropical lambda lengths]

Let $\mathbf{L^*} = \{L^*_i\}_{i \in I}$ be a lifted multi-lamination on an opened surface $(S^{*},M^{*})$. Let $\gamma$ be a quasi-arc or boundary component of $(S^{*},M^{*})$. We define the \textit{\textbf{tropical lambda length}}, $c_{\mathbf{L^*}}(\gamma)$, of $\gamma$ as follows:

\begin{equation}
c_{\mathbf{L^*}}(\gamma) = \prod_{i \in I} q_i^{-\frac{l_{L^*_i}(\gamma)}{2}}
\end{equation}

Note that by (\ref{transverse twist}) these tropical lambda lengths satisfy

\begin{equation}
\label{tropical twist}
c_{\mathbf{L^*}}(\psi^{\pm 1}_p(\gamma)) = c_{\mathbf{L^*}}(p)^{\pm n_p(\gamma)}c_{\mathbf{L^*}}(\gamma)
\end{equation}

\end{defn}

\begin{rmk} 
The transverse measure of an arc considers the number of intersections between the laminations, whereas shear coordinates depend on a triangulation and are only concerned with counting $'S'$ and $'Z'$ intersections. Nevertheless, the two notions are closely related. Fomin and Thurston \cite{fomin2012cluster} showed that for any arc $\gamma$ in a triangulation $T$ we have:

\begin{equation}
\label{shear and tropical}
r_{\gamma}(T,\mathbf{L^*}) := \prod_{i \in I}q_i^{-b_{\gamma}(T,L^*_i)} =  \frac{p_{\gamma}^{+}}{p_{\gamma}^{-}}\prod_{\beta \in T}c_{L^*}(\overline{\beta})^{B(T)_{\beta \gamma}},
\end{equation}

\noindent where the term $\frac{p_{\gamma}^{+}}{p_{\gamma}^{-}}$ accounts for the instances where $\overline{\gamma}$ is not the interior of a flip region.
Moreover, they showed that, for arcs, the exchange relations between the $c_{\mathbf{L^*}}$'s are the tropical versions of the exchange relations between the corresponding lambda lengths (in fact we shall see later this statement extends to all quasi-arcs). In turn, using this fact together with equation (\ref{shear and tropical}) yields an elegant proof of Proposition \ref{existing quiver flip}.
\end{rmk}

\subsection{Laminated lambda lengths and the laminated quasi-cluster algebra}

Recall that we began to consider the opened surface with the intention of rescaling lambda lengths of quasi-arcs using transverse measures. (Transverse measure is generally ill-defined on un-opened surfaces due to possible infinite intersections of arcs with the multi-lamination.) Our approach so far requires us to fix a lift $\overline{\gamma}$ in $(S^{*},M^{*})$ for each quasi-arc $\gamma$ in $(S,M)$. As we already noted in Remark \ref{non-normalised structure}, the clusters $\mathbf{x}(T) := \{l(\overline{\gamma}) : \gamma \in T\}$ arising from triangulations form a non-normalised exchange pattern on $E^{\circ}(S,M)$. If the arcs of a flip region in $(S,M)$ lift to another flip region in $(S^{*},M^{*})$ then the exchange relations coincide. However, if they do not lift to a flip region, the exchange relations will differ. In particular, when they do not, the exchange relation on the opened surface will be a rescaled version of the original. This rescaled relation is obtained by finding a new collection of lifts such that, with respect to these new lifts, the flip region does lift to a flip region. Since the new lifts only differ from the old via spiralling at opened punctures, this rewriting is obtained using (\ref{twist2}).
The issue with the current standings is that it is quite hard to keep track of these particular rescalings. The following definition shows that by putting boundary conditions on the opened punctures, we may both achieve our goal of defining \textit{laminated lambda lengths} that take into account the lamination on the surface, and eliminate the nasty rescaling process required when flip regions do not lift to flip regions.

\begin{defn}

The \textbf{\textit{complete decorated Teichm\"uller space}}, $\overline{\mathcal{T}}(S,M)$, is the disjoint union of the $\mathcal{T}_{\tilde{P}}(S_P,M_P)$ over all $3^{|\partial S \setminus M|}$ partially decorated sets $\tilde{P}$.

\end{defn}

\begin{defn}

Let $\mathbf{L} = \{L_i\}_{i \in I}$ be a multi-lamination of $(S,M)$. Fix a lift $\mathbf{L^*}$, of $\mathbf{L}$, on the opened surface $(S^*,M^*)$. A point $(\sigma,q)$ of the \textbf{\textit{laminated Teichm\"uller space}}, $\overline{\mathcal{T}}(S,M,\mathbf{L^*})$, consists of a decorated hyperbolic structure $\sigma \in \overline{\mathcal{T}}(S,M)$ and a collection of positive real numbers $q := (q_1,\ldots q_{|I|})$ subject to the following condition on all punctures $p \in \partial S \setminus M$:

$$ \lambda(p) = c_{\mathbf{L^*}}(p).$$

\end{defn}

\begin{defn}
\label{laminated lambda length}
Let $(S,M)$ be a bordered surface, and $\mathbf{L}$ a multi-lamination. Fix a lift $\mathbf{L^*}$ of $\mathbf{L}$. For each quasi-arc $\gamma$ of $(S,M)$ choose a lift $\overline{\gamma}$. We define the \textit{\textbf{laminated lambda length}}, $x_{\mathbf{L^*}}(\gamma)$, of $\gamma$ to be:

\begin{equation}
\label{laminated length}
x_{\mathbf{L^*}}(\gamma) := \frac{\lambda(\overline{\gamma})}{c_{\mathbf{L^*}}(\overline{\gamma})}
\end{equation}

Due to the enforced 'boundary' condition $\lambda(p) = c_{\mathbf{L^*}}(p)$ for each puncture $p$, from equations (\ref{twist2}) and (\ref{tropical twist}) we realise that $x_{\mathbf{L^*}}(\gamma)$, as the notation suggests, is independent of the choice of lift $\overline{\gamma}$. It is worth noting that this definition \textit{does} depend on the choice of lift $\mathbf{L^*}$.

\end{defn}

The following theorem follows from [Corollary 15.5, \cite{fomin2012cluster}].

\begin{thm}
\label{main homeo}
Let $\mathbf{L} = \{L_i\}_{i \in I}$ be a multi-lamination of $(S,M)$ and $\mathbf{L^*}$ a lift. For any quasi-triangulation $T$ with quasi-arcs and boundary arcs $\gamma_1, \ldots, \gamma_{n+b}$ there exists a homeomorphism 
\begin{align*} 
\Lambda_T \colon   {\overline{\mathcal{T}}}(S , &M,\mathbf{L^*}) \longrightarrow \mathbb{R}_{>0}^{n+b+|I|} \\
          &(\sigma,q) \mapsto (x_{\mathbf{L^*}}(\gamma_1), \ldots, x_{\mathbf{L^*}}(\gamma_1),q_1,\ldots q_{|I|})
\end{align*} 

\end{thm}

Theorem \ref{main homeo} allows us to simultaneously view the laminations in a multi-lamination, and the laminated lambda lengths of any quasi-triangulation, as algebraically independent variables. With this in mind, given a laminated bordered surface $(S,M,\mathbf{L})$, we can consider a seed, $(\textbf{x},T)$, consisting of a quasi triangulation $T$ and a collection of algebraically independent \textit{cluster variables} $\textbf{x} := \{x_{\gamma}| \gamma \in T\}$. Furthermore, consider the coefficient ring $\mathbb{ZP}$ generated (over $\mathbb{Z}$) by the algebraically independent \textit{frozen variables} $x_{b}$ and $x_{L_i}$ corresponding to each boundary segment $b$ of $(S,M)$ and each lamination $L_i$ of $\mathbf{L}$.\newline
\indent Performing flips of quasi-arcs, and using the exchange relations in Definition \ref{lambdarules} coupled with the equation (\ref{laminated length}) of a laminated lambda length, we can generate all other seeds with respect to our initial seed $(\textbf{x},T)$. \newline
\indent Let $\mathcal{X}$ be the set of all cluster variables appearing in all of these seeds. $\mathcal{A}_{(\mathbf{x},T)}(S,M,\mathbf{L^*}) := \mathbb{ZP}[\mathcal{X}]$ is the \textit{\textbf{laminated quasi-cluster algebra}} of the seed $(\mathbf{x},T)$. \newline
\indent The definition of a quasi-cluster algebra depends on the choice of the initial seed and of the lift $\mathbf{L^*}$. However, [Definition 15.3, \cite{fomin2012cluster}] reassures us that if we choose a different initial seed, or a different lift, the resulting laminated quasi-cluster algebra will be isomorphic to $\mathcal{A}_{(\mathbf{x},T)}(S,M,\mathbf{L^*})$. As such, it makes sense to talk about the laminated quasi-cluster algebra, $\mathcal{A}(S,M,\mathbf{L})$, of $(S,M,\mathbf{L})$.

\begin{defn}

The \textit{\textbf{laminated quasi-arc complex}} $\Delta^{\otimes}(S,M,\mathbf{L})$ of the laminated quasi-cluster algebra $\mathcal{A}(S,M,\mathbf{L})$ is the simplicial complex with the ground set being the cluster variables of $\mathcal{A}(S,M,\mathbf{L})$, and the maximal simplices being the clusters.

\end{defn}

\begin{defn}

The \textit{\textbf{exchange graph}} $E^{\otimes}(S,M)$ of the laminated quasi-cluster algebra $\mathcal{A}(S,M,\mathbf{L})$ is the graph whose vertices correspond to the clusters of $\mathcal{A}(S,M,\mathbf{L})$. Two vertices are connected by an edge if their corresponding clusters differ by a single mutation.

\end{defn}

\section{Connecting laminated quasi-cluster algebras to LP algebras}

\subsection{Finding exchange relations of quasi-arcs via quivers}

The following proposition, which is a subcase of [Theorem 15.6, \cite{fomin2012cluster}], tells us that when we are looking at flips between traditional triangulations, then the exchange polynomials of arcs can be obtained by looking at the ingoing and outgoing arrows of the associated quiver.

\begin{prop}
\label{flip variable change} 
Let $T$ be a triangulation of $(S,M)$, $\mathbf{L}$ a multi-lamination and $\overline{\mathbf{L}}$ a lift of $\mathbf{L}$ to $\overline{(S,M)}$. Label the arcs, boundary segments and laminations $1,\ldots, m$ and consider the associated quiver $Q_{\overline{T},\mathbf{\overline{L}}}$. 

Let $\gamma$ be an arc in $T$ and consider the unique arc $\gamma' \neq \gamma$ such that $T\cup\{\gamma'\}\setminus\{\gamma\}$ is a traditional triangulation. Suppose the lifts of $\gamma$ receive the labels $j$ and $\tilde{j}$. Then the exchange polynomial of $\gamma$ with respect to this flip is:

$$ F_j = \displaystyle \prod_{\substack{ b_{ij} >0\\ i \in \{1,\ldots, m, \tilde{1} \ldots, \tilde{m}\}}} x_i^{b_{ij}} \hspace{4mm} + \prod_{\substack{ b_{ij} <0\\ i \in \{1,\ldots, m, \tilde{1} \ldots, \tilde{m}\}}} x_i^{-b_{ij}}$$ 

\end{prop}

\begin{proof}

Follows from [Theorem 15.6, \cite{fomin2012cluster}].

\end{proof}

We are now at the stage where we know that for a laminated surface $(S,M,\mathbf{L})$, flipping arcs in a triangulation $T$ corresponds to double-mutation of the associated anti-symmetric quiver $Q_{\overline{T},\mathbf{\overline{L}}}$. Moreover, we know that the associated exchange relations of each vertex of $Q_{\overline{T},\mathbf{\overline{L}}}$ describe how the laminated lambda lengths change under flip. It is crucial to note that to get the correspondence above we have been allowing flips to arcs bounding $M_1$ instead of to one-sided closed curves. It turns out that if we make an adjustment to how we 'read off' polynomials from $Q_{\overline{T},\mathbf{\overline{L}}}$ then we can obtain the exchange relations regarding the flip to a one-sided closed curve instead of the arc bounding $M_1$.

\begin{defn}
\label{shortened}
Let $Q$ be an anti-symmetric quiver with $2m$ vertices, of which $m-n$ pairs are frozen. The \textit{\textbf{shortened exchange matrix}} of $Q$ is the matrix $\overline{B} = (\overline{b}_{ij})_{\substack{1 \leq i \leq m \\ 1 \leq j \leq n}}$, where $\overline{b}_{ij} := b_{ij} + b_{\tilde{i}j}$. Each column $1 \leq j \leq n$ of $\overline{B}$ is naturally associated to the polynomial

\begin{center}

 $\overline{F}_j^Q := \displaystyle \prod_{\substack{\overline{b}_{ij} >0\\ i \in \{1,\ldots, m\}}} x_i^{\overline{b}_{ij}} + \prod_{\substack{\overline{b}_{ij} <0\\ i \in \{1,\ldots, m\}}} x_i^{-\overline{b}_{ij}}$. 

\end{center}

\end{defn}

We wish to show that these exchange relations from $\overline{B}$ describe how laminated lambda lengths change when flipping arcs. To achieve this we require the following two lemmas.

\begin{lem}
\label{t-mutable path}
Let $T$ be a triangulation of $(S,M,\mathbf{L})$ and $Q_{\overline{T},\mathbf{\overline{L}}}$ its associated quiver. Furthermore, let $i$ be a vertex of $Q_{\overline{T},\mathbf{\overline{L}}}$ corresponding to an arc. Then there is a path $k \rightarrow i \rightarrow \tilde{k}$ in $Q_{\overline{T},\mathbf{\overline{L}}}$ for some vertex $k$ \textit{if and only if} $i$ flips to a one-sided closed curve and $k$ is an arc.

\end{lem}

\begin{proof}
Note that for any lamination $L \in \overline{\mathbf{L}}$ and any arc $\gamma \in \overline{T}$, $b_{L\gamma}$ and $b_{\tilde{L}\gamma}$ must both be non-negative or non-positive since $L$ and $\tilde{L}$ do not intersect. Therefore, if there is a path $k \rightarrow i \rightarrow \tilde{k}$ in $Q_{\overline{T},\mathbf{\overline{L}}}$ then $k$ must be an arc and not a lamination. Furthermore, by anti-symmetry there is also the path $i \leftarrow k \rightarrow \tilde{i}$. This implies the existence of the quadrilateral $(a,\tilde{i},\tilde{b},i)$ shown in Figure \ref{square}, where $a$ and $\tilde{b}$ may not be arcs in $T$, but the associated arc bounding an arc and its notched counterpart. We see that $a, \tilde{b} \notin \{i, \tilde{i}\}$ as this would then imply $T$ contains either a punctured monogon or $M_1$, both of which are forbidden. Applying anti-symmetry again we find the existence of the quadrilateral $(\tilde{a},i,b,\tilde{i})$. Glueing these two quadrilaterals together and taking the $\mathbb{Z}_2$-quotient yields the picture in Figure \ref{badflip}, confirming that $i$ flips to a one-sided closed curve. The proof of the other direction is trivial.

\end{proof}

\begin{figure}[H]
\begin{center}
\includegraphics[width=13cm]{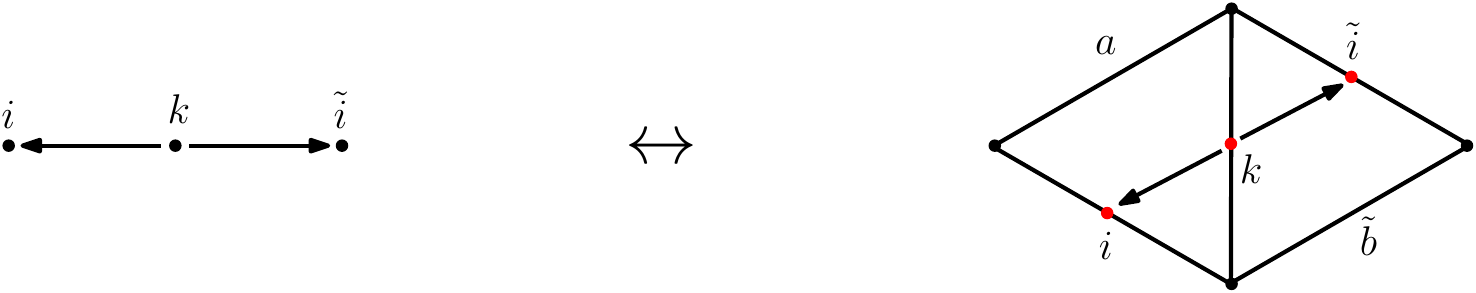}
\caption{The local configuration of the surface if $i \leftarrow k \rightarrow \tilde{i}$ is a path in $Q_{\overline{T},\mathbf{\overline{L}}}$.}
\label{square}
\end{center}
\end{figure}

\begin{figure}[H]
\begin{center}
\includegraphics[width=10cm]{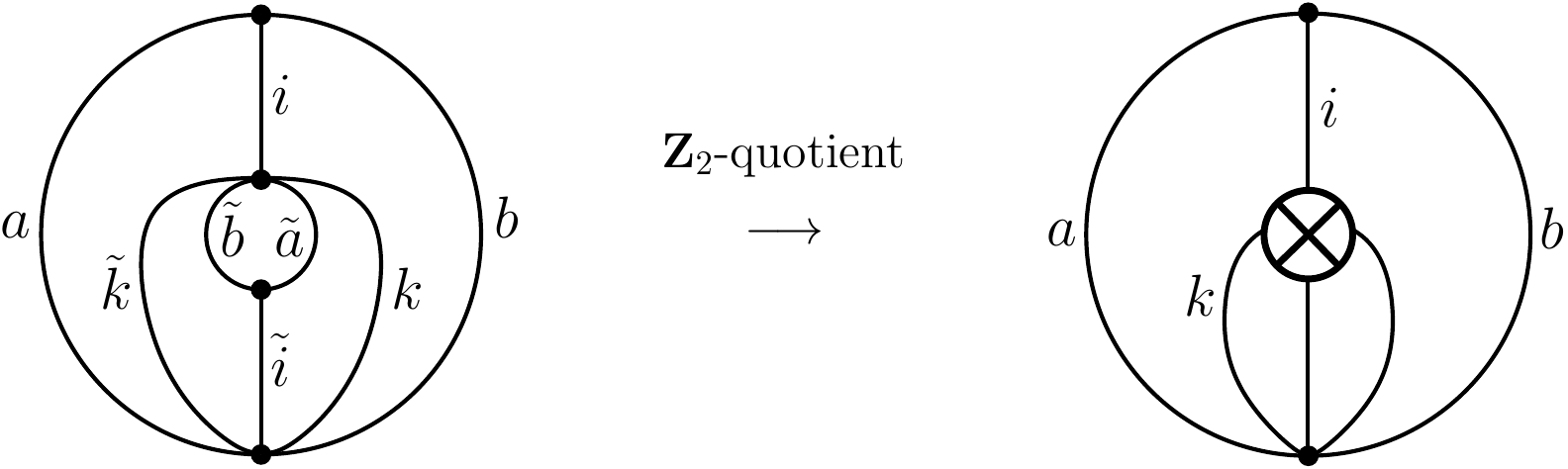}
\caption{The quasi-triangulation induced by the path $k \rightarrow i \rightarrow \tilde{k}$ in $Q_{\overline{T},\mathbf{\overline{L}}}$.}
\label{badflip}
\end{center}
\end{figure}

\begin{lem}
\label{M1 lamination}
Let $\alpha^*$ be an arc bounding a M\"obius strip with one marked point, $M_1$, and $\beta$ the unique arc in $M_1$. Consider the flip of $\beta$ to the one-sided closed curve $\alpha$. Then $x_{\mathbf{L}^*}(\alpha)x_{\mathbf{L}^*}(\beta) = x_{\mathbf{L}^*}(\alpha^*)$.

\end{lem}

\begin{proof}

There are two elementary laminations of $M_1$ - these are shown in Figure \ref{elementarylaminations}. \newline

\noindent For the lamination on the left in Figure \ref{elementarylaminations}:  $$c_{\mathbf{L}^*}(\beta) = c_{\mathbf{L}^*}(\alpha) = q^{-\frac{1}{2}}, \hspace{5mm} c_{\mathbf{L}^*}(\alpha^*) = q^{-1}.$$
\noindent For the lamination on the right in Figure \ref{elementarylaminations}: $$c_{\mathbf{L}^*}(\beta) = c_{\mathbf{L}^*}(\alpha^*) = q^{-1}, \hspace{11mm} c_{\mathbf{L}^*}(\alpha) = 1.$$

Recall that by Definition \ref{lambdarules} the lambda lengths of $\alpha, \beta, \alpha^*$ are related by $\lambda(\alpha)\lambda(\beta) = \lambda(\alpha^*)$. Therefore, employing equation (\ref{laminated length}), for any multi lamination $\mathbf{L}$ we obtain $$x_{\mathbf{L}^*}(\alpha)x_{\mathbf{L}^*}(\beta) = \frac{\lambda(\alpha)\lambda(\beta)}{c_{\mathbf{L}^*}(\alpha)c_{\mathbf{L}^*}(\beta)} = \frac{\lambda(\alpha^*)}{c_{\mathbf{L}^*}(\alpha^*)}= x_{\mathbf{L}^*}(\alpha^*).$$

\end{proof}

\begin{rmk}
\label{excluding laminations}
Note that the truth of Lemma \ref{M1 lamination} crucially depends on our exclusion, in Definition \ref{laminationdef}, of closed curves that are: one-sided, or bound a M\"obius strip. If $L$ is one of these forbidden curves contained in $M_1$, then $c_{\mathbf{L}^*}(\alpha)c_{\mathbf{L}^*}(\beta) = q^{-1} \neq 1 = c_{\mathbf{L}^*}(\alpha^*)$. Consequently, $x_{\mathbf{L}^*}(\alpha)x_{\mathbf{L}^*}(\beta) \neq x_{\mathbf{L}^*}(\alpha^*)$.
\end{rmk}

\begin{figure}[H]
\begin{center}
\includegraphics[width=10cm]{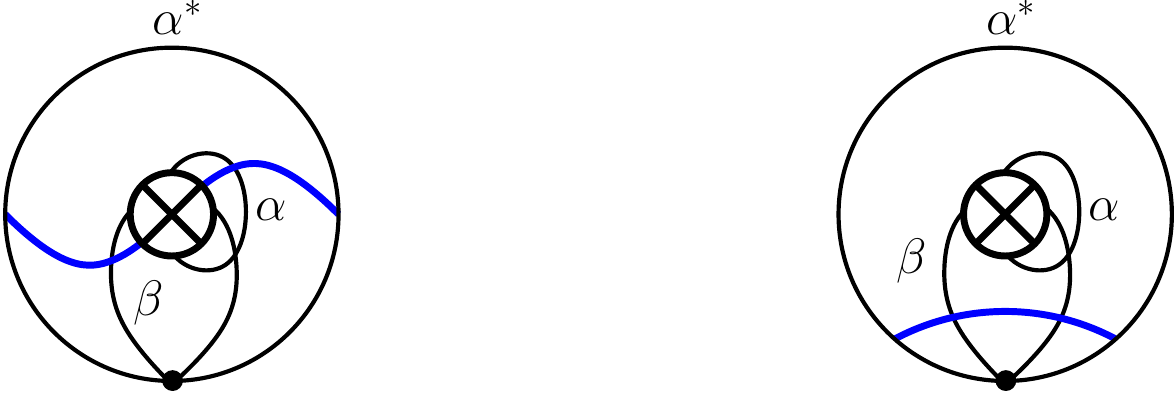}
\caption{The two \textit{elementary} laminations of $M_1$ (meaning every lamination of $M_1$ will be some union of these laminations). We also include the quasi-arcs $\alpha, \beta, \alpha^*$ so the reader can verify the tropical lambda lengths occurring in the proof of Lemma \ref{M1 lamination}.}
\label{elementarylaminations}
\end{center}
\end{figure}

\noindent \textbf{\underline{Notation}:} From here onwards, by abuse of notation, for each quasi-arc $\gamma$ of a laminated bordered surface $(S,M,\mathbf{L})$, we shall denote the laminated lambda length $x_{\mathbf{L}^*}(\gamma)$ by $\gamma$ itself. Previously we had also denoted the laminated lambda length of a quasi-arc $\gamma$ by $x_{\gamma}$, however this notation would prove cumbersome in what follows.

\begin{prop}
\label{correct polys}
Let $Q_{\overline{T},\mathbf{\overline{L}}}$ be an anti-symmetric quiver of a triangulation $T$ of $(S,M,\mathbf{L})$. Then the polynomials $\overline{\mathbf{F}}$ from Definition \ref{shortened} are the exchange relations describing how laminated lambda lengths change under flips of arcs in $T$.

\end{prop}

\begin{proof}

Currently by Proposition \ref{flip variable change} we know that the polynomial 

$$F_j^{Q_{\overline{T},\mathbf{\overline{L}}}} = \displaystyle \prod_{\substack{ b_{ij} >0\\ i \in \{1,\ldots, m, \tilde{1} \ldots, \tilde{m}\}}} x_i^{b_{ij}} \hspace{4mm} + \prod_{\substack{ b_{ij} <0\\ i \in \{1,\ldots, m, \tilde{1} \ldots, \tilde{m}\}}} x_i^{-b_{ij}}$$

\noindent describes how the laminated lambda length of an arc $\gamma_j$ in $T$ changes under flip when we allow flips to arcs bounding $M_1$ instead of one-sided closed curves. By Lemma \ref{t-mutable path}, if $\gamma_j$ does not flip to a one-sided closed curve then $F_j^{Q_{\overline{T},\mathbf{\overline{L}}}} = \overline{F}_j^{Q_{\overline{T},\mathbf{\overline{L}}}}$. \newline
If $\gamma_j$ does flip to a one-sided closed curve $\alpha$, then by Lemma \ref{t-mutable path} we know $\overline{b}_{kj} := b_{kj} + b_{\tilde{k}j} = 0$ for some $k$ (where $b_{kj} = -b_{\tilde{k}j} = \pm 1$), and $b_{ij},b_{\tilde{i}j}$ are both simultaneously non-positive or non-negative for all $i \in [m]\setminus\{k\}$. Hence $\overline{F}_j^{Q_{\overline{T},\mathbf{\overline{L}}}} = \frac{F_j^{Q_{\overline{T},\mathbf{\overline{L}}}}}{\beta}$, where $\beta$ is the arc corresponding to $k$. Proposition \ref{flip variable change} tells us that when $\gamma_j$ flips instead to the arc $\alpha^*$ enclosing $M_1$, then $\gamma_j \alpha^* = F_j^{Q_{\overline{T},\mathbf{\overline{L}}}}$. Moreover, by Lemma \ref{M1 lamination} we know that $\alpha\beta = \alpha^*$. This gives us the desired relation $\gamma_j \alpha = \frac{F_j^{Q_{\overline{T},\mathbf{\overline{L}}}}}{\beta} = \overline{F}_j^{Q_{\overline{T},\mathbf{\overline{L}}}}$.

\end{proof}

Proposition \ref{correct polys} tells us that for a triangulation $T$ of a laminated bordered surface $(S,M,\mathbf{L})$, for any arc $\gamma \in T$, the exchange polynomial $\overline{F}_{\gamma}$ is obtained from considering (sums of) the ingoing and outgoing arrows of the vertex $\gamma$ (or equivalently $\tilde{\gamma}$) in $Q_{\overline{T},\mathbf{\overline{L}}}$. However, currently we have no such combinatorial method that provides us with the exchange polynomials of quasi-arcs in quasi-triangulations containing one-sided closed curves. The following lemma addresses this. \newline Before we state the lemma let us fix some notation. Recall that a one-sided closed curve $\alpha$ in a quasi-triangulation $T$ will intersect precisely one arc $\beta \in T$. As it always will throughout this chapter, $\alpha^*$ denotes the unique arc enclosing $\alpha$ and $\beta$ in $M_1$. For each quasi-triangulation $T$ we can therefore uniquely associate a traditional triangulation $T^{*}$ by replacing each one-sided closed curve $\alpha \in T$ with $\alpha^*$, see Figure \ref{alternativeflip}. \newline
For the rest of this chapter, by an abusive of notation, for each quasi-arc $\gamma$ we shall also denote its laminated lambda length by $\gamma$ -- previously written as $x_{\mathbf{L}^*}(\gamma)$. Similarly, for a lamination $L_i$ of a multi-lamination $\mathbf{L}$, we also denote its corresponding variable by $L_i$ -- previously written as $q_i$.

\begin{lem}
\label{correct polys quasi}
Let $T$ be a quasi-triangulation containing a one-sided closed curve $\alpha$, and let $\beta$ denote the unique arc in $T$ intersecting $\alpha$. Moreover, suppose $x$ and $y$ are the arcs in $T$ enclosing $\alpha$ and $\beta$ in a M\"{o}bius strip with 2 marked points. Consider the associated traditional triangulation $T^*$ and its corresponding quiver $Q_{\overline{T}^*}$, shown in Figure \ref{traditionalquiver}. Furthermore, denote by $b_{ij}$ and $b_{ij}'$ the coefficients of $Q_{\overline{T}^*}$ and $\mu_{\alpha^*}\circ \mu_{\tilde{\alpha}^*}(Q_{\overline{T}^*})$, respectively. Then the exchange polynomials of the quasi-arcs $\alpha$ and $\beta$ in $T$ are given by:

$$\alpha \alpha' = \displaystyle \big( \prod_{\overline{b}_{L_i\alpha^{*}} >0} L_i^{\overline{b}_{L_i\alpha^{*}}} \big)y +  \big(\prod_{\overline{b}_{L_i\alpha^{*}} <0} L_i^{-\overline{b}_{L_i\alpha^{*}}} \big)x$$

\begingroup\makeatletter\def\f@size{10}\check@mathfonts
$$\beta\beta' = \frac{\displaystyle \big( \prod_{\overline{b}'_{L_i\beta} >0} L_i^{\overline{b}'_{L_i\beta}} \big){\bigg(\big( \prod_{\overline{b}_{L_i\alpha^{*}} >0} L_i^{\overline{b}_{L_i\alpha^{*}}} \big)y +  \big(\prod_{\overline{b}_{L_i\alpha^{*}} <0} L_i^{-\overline{b}_{L_i\alpha^{*}}} \big)x\bigg)}^2 +  \big(\prod_{\overline{b}'_{L_i\beta} <0} L_i^{-\overline{b}'_{L_i\beta}} \big)xy\alpha^2}{\alpha^2}.$$\endgroup

\end{lem}

\begin{proof}

The exchange relation of $\alpha$ follows from Proposition \ref{traditional quiver mutation} and Proposition \ref{correct polys}.\newline
To obtain the exchange polynomial of $\beta$ we (first) need to consider $\mu_{\alpha^*}\circ \mu_{\tilde{\alpha}^*}(Q_{\overline{T}^*})$ instead of $Q_{\overline{T}^*}$ -- this is because $\beta$ flips to $\beta'$ in $\mu_{\alpha^*}(T^*)$, but not in $T^*$. By Proposition \ref{correct polys} we get that:

$$\beta \beta' = \displaystyle \big( \prod_{\overline{b}'_{L_i\beta} >0} L_i^{\overline{b}'_{L_i\beta}} \big)\alpha'^2 +  \big(\prod_{\overline{b}'_{L_i\beta} <0} L_i^{-\overline{b}'_{L_i\beta}} \big)xy.$$

Rewriting $\alpha'$ using the exchange relation already obtained for $\alpha$ yields

\begingroup\makeatletter\def\f@size{10}\check@mathfonts
$$\beta\beta' = \frac{\displaystyle \big( \prod_{\overline{b}'_{L_i\beta} >0} L_i^{\overline{b}'_{L_i\beta}} \big){\bigg(\big( \prod_{\overline{b}_{L_i\alpha^{*}} >0} L_i^{\overline{b}_{L_i\alpha^{*}}} \big)y +  \big(\prod_{\overline{b}_{L_i\alpha^{*}} <0} L_i^{-\overline{b}_{L_i\alpha^{*}}} \big)x\bigg)}^2 +  \big(\prod_{\overline{b}'_{L_i\beta} <0} L_i^{-\overline{b}'_{L_i\beta}} \big)xy\alpha^2}{\alpha^2}.$$\endgroup

\end{proof}

\begin{rmk} More generally, for any quasi-arc $\gamma$ in a quasi-triangulation $T$, the exchange polynomial for $\gamma$ is still obtained by the formulae of Propositions \ref{correct polys} and \ref{correct polys quasi} -- we just have to remember that if a variable $\alpha^*$ appears in the exchange relation we must replace it with $\alpha\beta$.
\end{rmk}

\begin{figure}[H]
\begin{center}
\includegraphics[width=12cm]{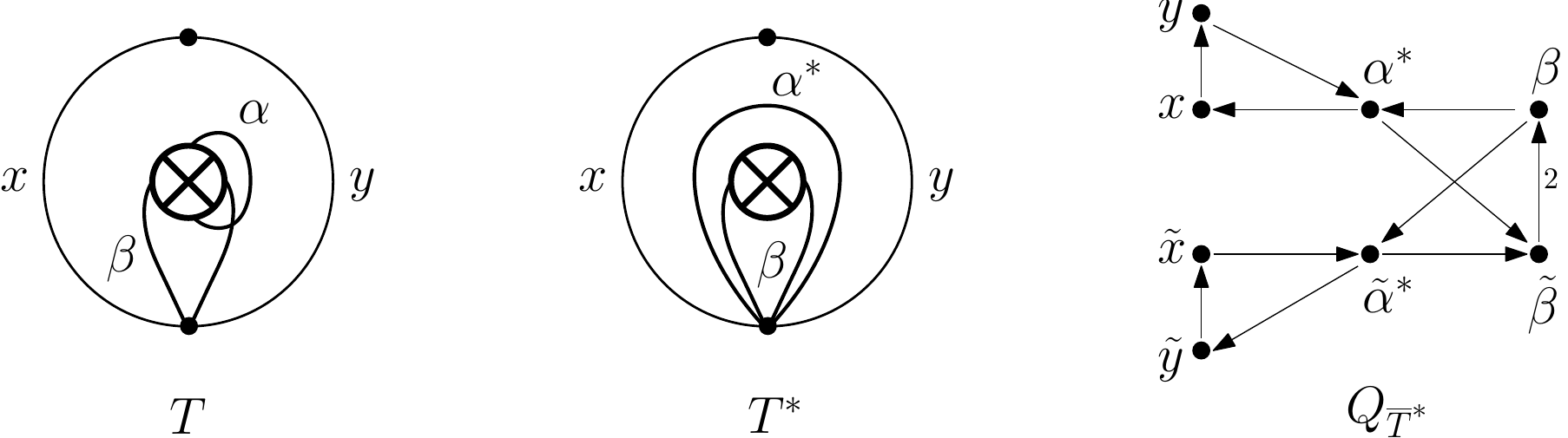}
\caption{The quasi-triangulation $T$, its traditional triangulation $T^*$, and the associated quiver $Q_{\overline{T}^*}$ arising from the lift of $\overline{T}^*$ to the double cover.}
\label{traditionalquiver}
\end{center}
\end{figure}

\subsection{Laminated quasi-cluster algebras via LP mutation}

Let $\gamma$ be a quasi-arc in a quasi-triangulation $T$. We have seen in Propositions \ref{correct polys} and \ref{correct polys quasi} that the exchange relation of $\gamma$ is a Laurent polynomial; we denote by $F_{\gamma}$ the numerator of this polynomial. To each quasi-triangulation $T$ we assign the 'LP' seed $(\mathbf{x},\mathbf{F}_T)$ where $\mathbf{x} := \{x_{\mathbf{L}^*}(\gamma) | \gamma \in T\}$ and $\mathbf{F}_T := \{F_{\gamma} | \gamma \in T\}$. Of course, due to the irreducibility conditions, $(\mathbf{x},\mathbf{F}_T)$ may not be a valid LP seed -- this will be addressed later. \newline
The following lemma assures us that, if the polynomials in $\mathbf{F}_T$ are distinct, then the normalisations of these polynomials are the exchange relations of their corresponding quasi-arcs.

\begin{lem}
\label{normalisation assumption}
Let $T$ be a quasi-triangulation and suppose $F_{\gamma_i} \neq F_{\gamma_j}$ for any quasi-arcs $\gamma_i$ and $\gamma_j$ in $T$ $(i\neq j)$. If $\gamma \in T$ intersects a one-sided closed curve $\alpha \in T$ then $\hat{F}_{\gamma} = \frac{F_{\gamma}}{\alpha^2}$, otherwise $\hat{F}_{\gamma} = F_{\gamma}$.

\end{lem}

\begin{proof}

Let $\gamma_1, \ldots, \gamma_n$ be the quasi-arcs in $T$. Recall that $$\hat{F}_{\gamma_j} := \frac{F_{\gamma_j}}{\gamma_1^{a_1}\ldots \gamma_{j-1}^{a_{j-1}}\gamma_{j+1}^{a_{j+1}}\ldots \gamma_n^{a_n}}$$ \noindent where $a_k \in \mathbb{Z}_{\geq 0}$ is maximal such that $F_{\gamma_k}^{a_k}$ divides ${F}_{\gamma_j} \rvert_{\gamma_k \leftarrow \frac{F_{\gamma_k}}{x}}$.

Hence, $a_k > 0$ \textit{if and only if} $F_{\gamma_k}$ divides the constant term of $F_{\gamma_j}$ when viewed as a polynomial in $\gamma_k$.

If $\gamma_j$ does not intersect a one-sided closed curve in $T$ then $F_{\gamma_j}$ is a binomial. As a consequence, when viewed as a polynomial in $\gamma_k$, the constant term is either a monomial ($\gamma_k \in F_{\gamma_j}$), or the whole binomial $F_{\gamma_j}$ ($\gamma_k \notin F_{\gamma_j}$). If it is a monomial then it is not divisible by $F_{\gamma_k}$. From our assumptions in the lemma we know that $F_{\gamma_j}$ is irreducible and $F_{\gamma_j} \neq F_{\gamma_k}$. So if the constant term is $F_{\gamma_j}$, this also cannot be divisible by $F_{\gamma_k}$. Hence $\hat{F}_{\gamma_j} = F_{\gamma_j}$.

If $\gamma_j$ intersects a one-sided closed curve $\alpha \in T$, then $\gamma_j$ has the flip region shown in Figure \ref{yjflipregion}. Moreover, by Lemma \ref{correct polys quasi}, it has the exchange polynomial $$F_{\gamma_j} = \displaystyle \big( \prod_{\overline{b}_{i\gamma_j} >0} L_i^{\overline{b}_{i\gamma_j}} \big)F_{\alpha}^2 +  \big(\prod_{\overline{b}_{i\gamma_j} <0} L_i^{-\overline{b}_{i\gamma_j}} \big)xy\alpha^2 $$

where 

$$F_{\alpha} = \displaystyle \big( \prod_{\overline{b}_{i\alpha} >0} L_i^{\overline{b}_{i\alpha}} \big)y +  \big(\prod_{\overline{b}_{i\alpha} <0} L_i^{-\overline{b}_{i\alpha}} \big)x $$.

Accordingly, for any quasi-arc $\gamma_k \in T\setminus \{\alpha\}$, the constant term of $F_{\gamma_j}$, when viewed as a polynomial in $\gamma_k$, is a monomial or $F_{\gamma_j}$. Just as before, this implies $\gamma_k \notin \frac{\hat{F}_{\gamma_j}}{F_{\gamma_j}}$. However, when $F_{\gamma_j}$ is viewed as a polynomial in $\alpha$ the constant term is $\displaystyle \big( \prod_{\overline{b}_{i\gamma_j} >0} L_i^{\overline{b}_{i\gamma_j}} \big)F_{\alpha}^2$, and the degree $1$ term is $0$. Thus, $\hat{F}_{\gamma_j} = \frac{F_{\gamma_j}}{\alpha^2}$.

\end{proof}

\begin{figure}[H]
\begin{center}
\includegraphics[width=4cm]{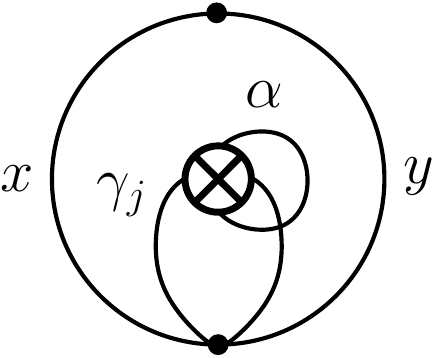}
\caption{The flip region of $\gamma_j$ if it intersects a one-sided closed curve.}
\label{yjflipregion}
\end{center}
\end{figure}

\begin{lem}
\label{lamination restriction}
Let $T$ be the traditional triangulation of $M_2$ obtained from glueing together a triangle and an anti-self-folded triangle. If we label the lifted arcs as in Figure \ref{traditionallift}, then for any lamination $L$ of $M_2$, in the quiver $Q_{\overline{T},\overline{L}}$ we have $\overline{b}_{L\beta} \geq 0$ and either 
\begin{center}

$b_{L\alpha^*} \geq b_{\tilde{L}\beta}$ and $b_{\tilde{L}\alpha^*} \geq b_{L\beta} \hspace{7mm} or \hspace{7mm} b_{L\alpha^*} \leq b_{\tilde{L}\beta}$ and $b_{\tilde{L}\alpha^*} \leq b_{L\beta}$.

\end{center}

\end{lem}

\begin{figure}[H]
\begin{center}
\includegraphics[width=10cm]{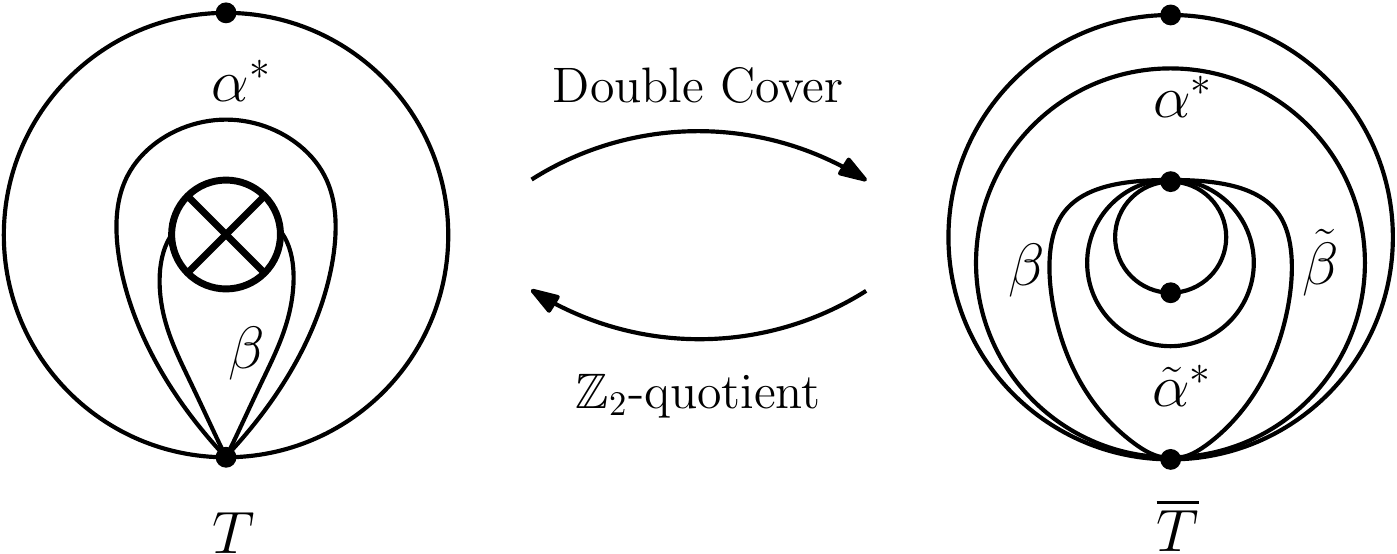}
\caption{The traditional triangulation $T$ obtained from glueing a triangle with an anti-self-folded triangle, and its lift $\overline{T}$. See Figure \ref{antiself} to recall the definition of an anti-self-folded triangle.}
\label{traditionallift}
\end{center}
\end{figure}

\begin{proof}

Let us consider the lifted triangulation $\overline{T}$ of $T$, and suppose we have labelled the arcs as shown in Figure \ref{traditionallift}. We shall first determine when $L$ adds weight to $\beta$ or $\alpha^*$.\newline Recall that for a lamination $L$ to add positive (resp. negative) weight to $\beta$ it needs to cut the quadrilateral in $\overline{T}$ containing $\beta$ in an $'S'$ (resp. $'Z'$) shape. In Figure \ref{betashapes}, for each shape type, we show the local configuration of the lamination within the quadrilateral containing $\beta$, and we denote its accompanying twin lamination with a dotted line. 

\begin{figure}[H]
\begin{center}
\includegraphics[width=10cm]{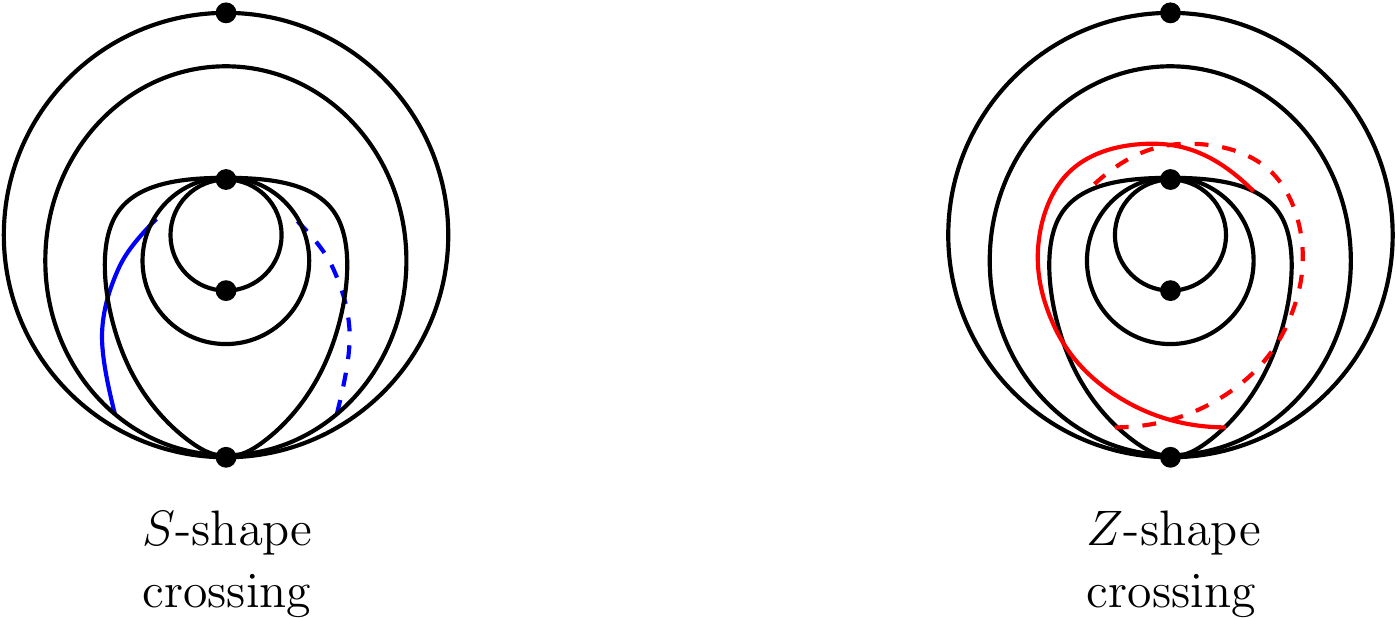}
\caption{The instances where the lamination cuts $\beta$ in an '$S$' or '$Z$' shape.}
\label{betashapes}
\end{center}
\end{figure}

For the case when $L$ cuts $\beta$ in an $'S'$ shape the partial lamination shown in Figure \ref{betashapes} can be extended (without self-intersections) in three ways. The $\mathbb{Z}_2$-quotients of these extensions are shown in Figure \ref{betalaminations}. However, note that when $L$ cuts $\beta$ in a $'Z'$ shape the partial 'lamination' shown in Figure \ref{betashapes} is self intersecting, and therefore will not form a legitimate lamination. As a consequence, for any lamination $L$, when we label the arcs as in Figure \ref{traditionallift}, then we can only ever have $\overline{b}_{L\beta} \geq 0$. (If we labelled $\beta$ and $\tilde{\beta}$ the other way round we would only ever have $\overline{b}_{L\beta} \leq 0$. See Definition \ref{shortened} for a recap on how we define $\overline{b}_{ij}$.)

\begin{figure}[H]
\begin{center}
\includegraphics[width=13cm]{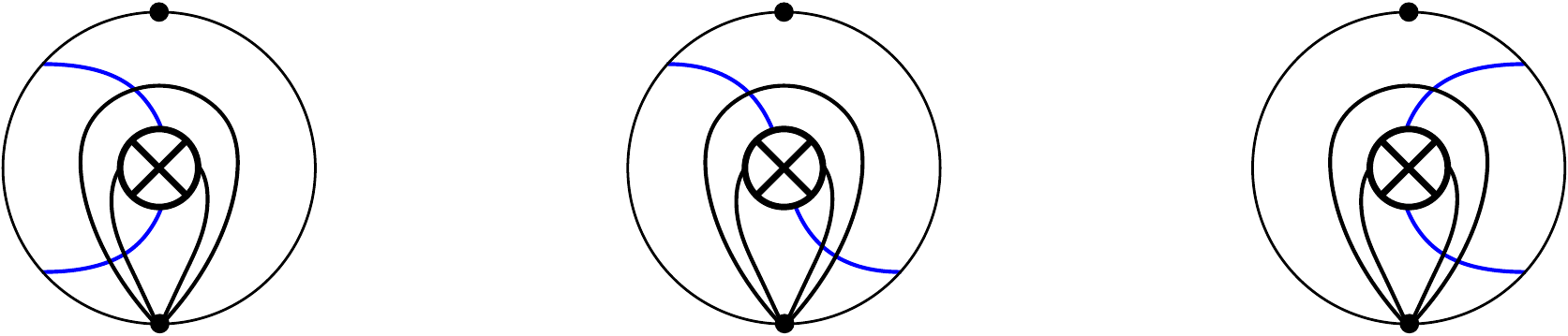}
\caption{The possible (elementary) laminations adding weight to $\beta$.}
\label{betalaminations}
\end{center}
\end{figure}

Now suppose $L$ adds weight to $\alpha^*$. Locally within the quadrilateral containing $\alpha^*$, depending on which shape $L$ cuts $\alpha^*$, $L$ will have one of the configurations shown in Figure \ref{alphashapes}.

\begin{figure}[H]
\begin{center}
\includegraphics[width=10cm]{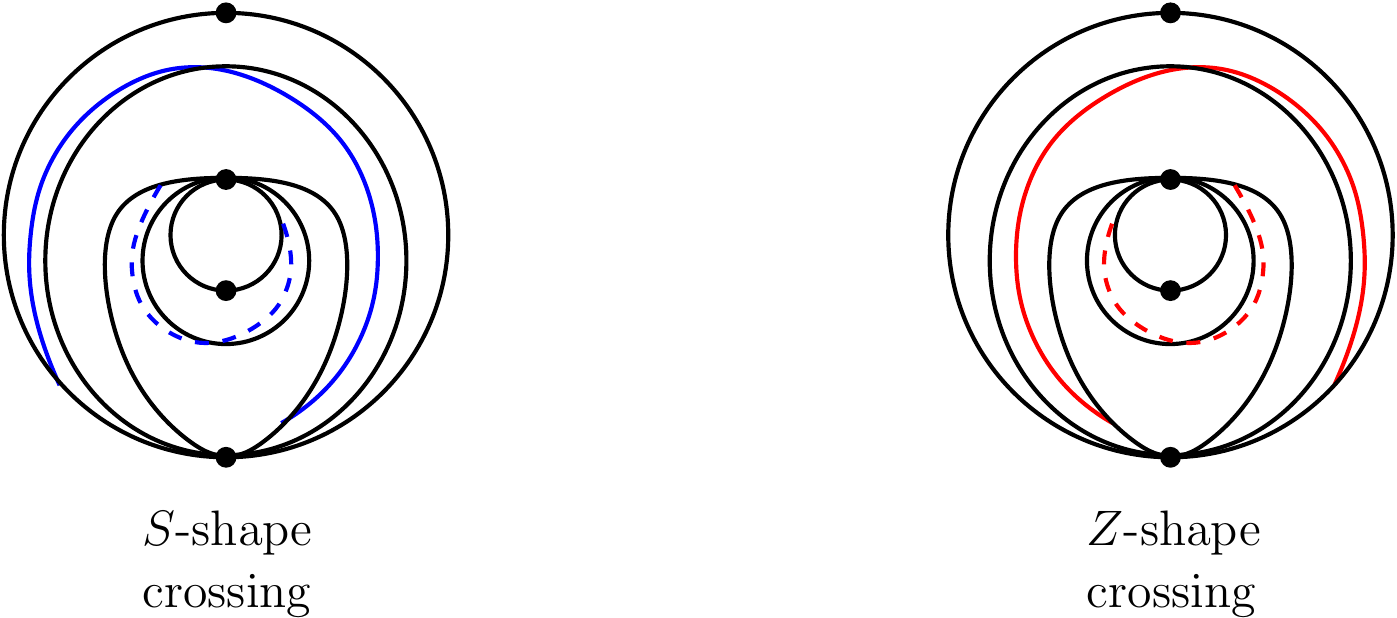}
\caption{The instances where the lamination cuts $\alpha^*$ in an '$S$' or '$Z$' shape.}
\label{alphashapes}
\end{center}
\end{figure}

We can see that each configuration can be extended to a (non-intersecting) lamination in precisely two ways. Taking the $\mathbb{Z}_2$-quotient leaves us with the laminations shown in Figure \ref{alphalaminations}.

\begin{figure}[H]
\begin{center}
\includegraphics[width=13cm]{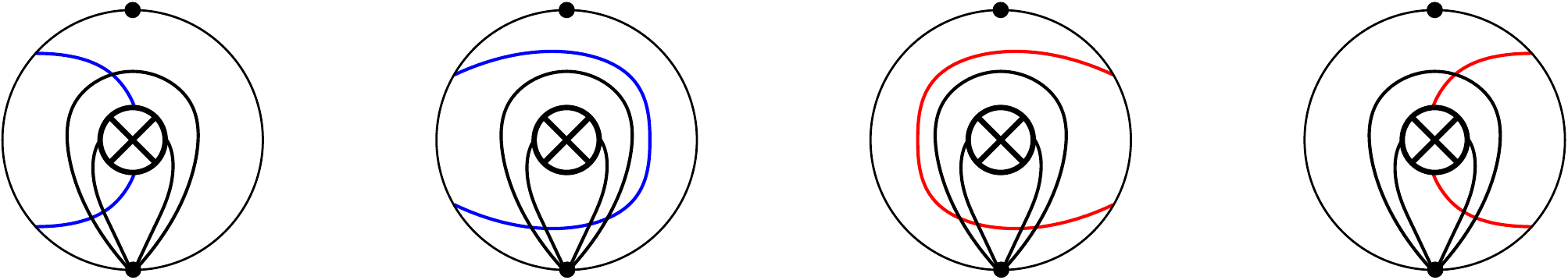}
\caption{The possible (elementary) laminations adding weight to $\alpha^*$.}
\label{alphalaminations}
\end{center}
\end{figure}

In Figure \ref{laminationcoefficients} we list the double covers of all elementary laminations which add weight to $\alpha^*$ and $\beta$ -- in each case we indicate the corresponding lamination weights. Note that for: $$(1) \hspace{2mm} \text{we get} \hspace{2mm}  b_{L\alpha^*} = b_{\tilde{L}\beta} \hspace{2mm}  \text{and} \hspace{2mm}  b_{\tilde{L}\alpha^*} = b_{L\beta},$$ 
$$(4) \hspace{2mm} \text{we get} \hspace{2mm}  b_{L\alpha^*} \geq b_{\tilde{L}\beta} \hspace{2mm}  \text{and} \hspace{2mm}  b_{\tilde{L}\alpha^*} \geq b_{L\beta},$$
$$(2), (3) \text{ and } (5) \hspace{2mm} \text{we get} \hspace{2mm}  b_{L\alpha^*} \geq b_{\tilde{L}\beta} \hspace{2mm}  \text{and} \hspace{2mm}  b_{\tilde{L}\alpha^*} \geq b_{L\beta}.$$

Recall that, by definition, any lamination is the union of non-intersecting elementary laminations. Since the lamination in $(4)$ intersects the laminations in $(2)$,$(3)$ and $(5)$ then any lamination will indeed satisfy $$b_{L\alpha^*} \geq b_{\tilde{L}\beta} \text{ and } b_{\tilde{L}\alpha^*} \geq b_{L\beta} \hspace{5mm} \text{ or } \hspace{5mm} b_{L\alpha^*} \leq b_{\tilde{L}\beta} \text{ and } b_{\tilde{L}\alpha^*} \leq b_{L\beta}.$$

\end{proof}

\begin{figure}[H]
\begin{center}
\includegraphics[width=13cm]{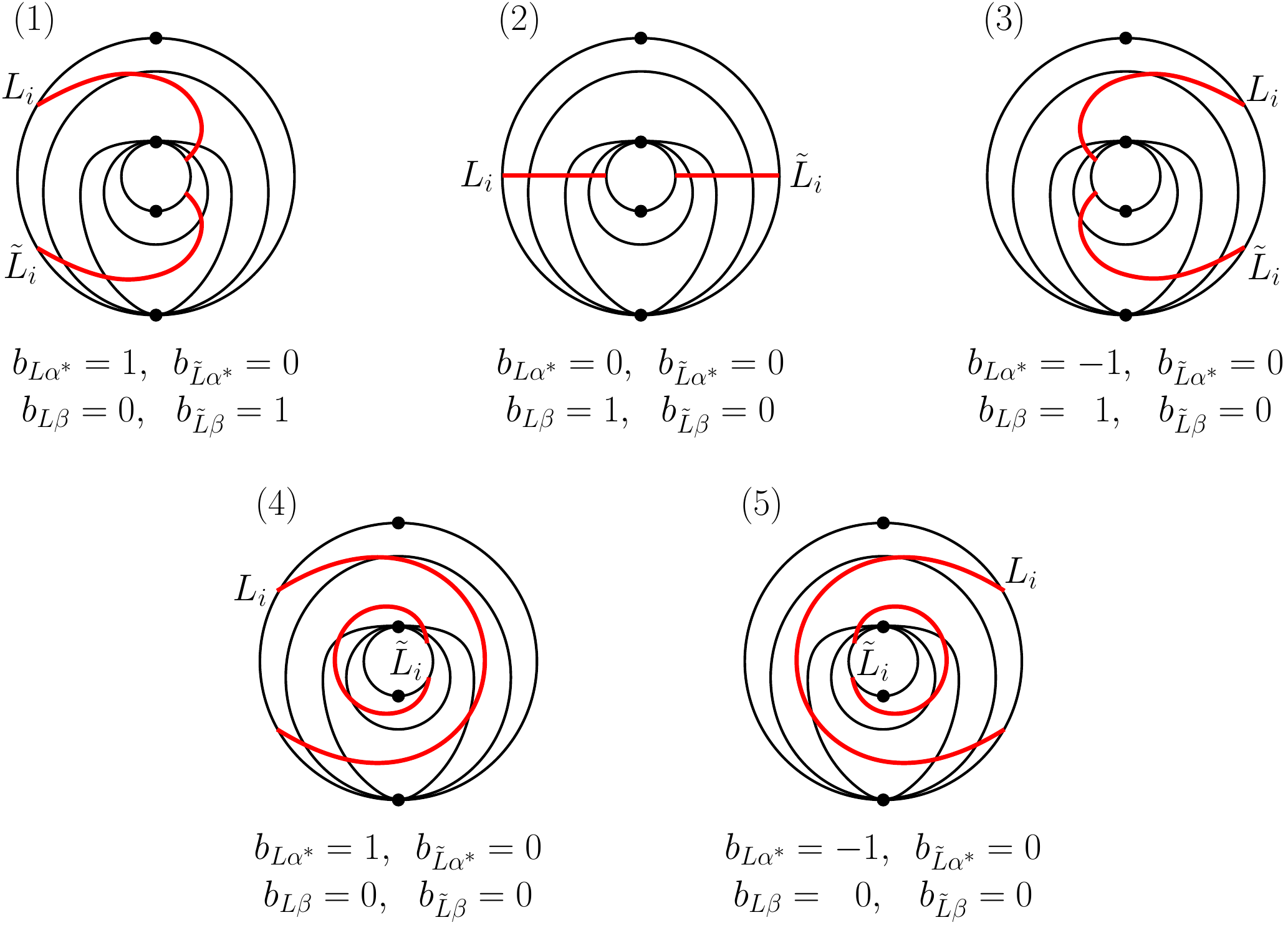}
\caption{Here we draw the double covers of all elementary laminations adding weight to $\alpha^*$ or $\beta$, and we also include the corresponding lamination weights.}
\label{laminationcoefficients}
\end{center}
\end{figure}

\begin{rmk} 
In the above proof note that relabelling $L \leftrightarrow \tilde{L}$ in any of the configurations in Figure \ref{laminationcoefficients} does not alter the corresponding inequalities.
\end{rmk}

Having realised how to obtain lamination coefficients of exchange polynomials of any quasi-triangulation $T$ we will now describe how these change under flips. If $T$ is a triangulation then these coefficients will change in accordance to usual quiver mutation formulae. We are therefore left with the task of describing how coefficients change when we perform flips in regions containing a one-sided closed curve. For a quiver $Q_{\overline{T}}$ arising from a traditional triangulation $T$ containing anti-self-folded triangles, Lemma \ref{lamination restriction} puts a restriction on the possible extended quivers, $Q_{\overline{T},\mathbf{\overline{L}}}$, that can arise from a multi lamination $\mathbf{L}$ on the surface. We shall use this lemma to sift out these 'obvious' impossible extended quivers. After this sifting, on the surviving extended quivers, we will describe in the following lemma how the quiver changes with respect to flips of arcs in $T$. \newline

\begin{lem}
\label{exceptional relations}
Let $(S,M,\mathbf{L})$ be a laminated bordered surface. Consider the quasi-triangulation $T$ and its associated lift $\overline{T}^*$, both shown in Figure \ref{tandtraditionallift}. Let us denote the coefficients of $Q_{\overline{T}^{*},\overline{\mathbf{L}}}$ by $\overline{b}_{ij}$. Then: 
\begin{enumerate}[label=(\alph*)]
\item the lamination coefficients $\overline{b}'_{ij}$, corresponding to $\alpha^*$, $\beta$ and $x$ in $Q_{\overline{\mu_{\alpha^*}(T)}^{*},\overline{\mathbf{L}}}$ can be written as: $$ \overline{b}_{L_ix}' = \overline{b}_{L_ix} + \max(0,\overline{b}_{L_i\alpha^*}), \hspace{9mm} \overline{b}_{L_i\alpha^*}' = -\overline{b}_{L_i\alpha^*}, \hspace{9mm}\overline{b}_{L_i\beta}' = \overline{b}_{L_i\beta} - |\overline{b}_{L_i\alpha^*}| $$
\item The lamination coefficients $\overline{b}'''_{ij}$, of $\alpha^*$, $\beta$ and $x$ in $Q_{\overline{\mu_{\beta}(T)}^{*},\overline{\mathbf{L}}}$ can be written as: 
$$\overline{b}'''_{L_i\alpha^*} = \overline{b}_{L_i\alpha^*}, \hspace{10mm} \overline{b}'''_{L_i\beta} = -\overline{b}_{L_i\beta},$$
$$ \overline{b}'''_{L_ix} = \overline{b}_{L_ix} + \max(0,\overline{b}_{L_i\alpha^*} + \overline{b}_{L_i\beta}) + \max(0,\overline{b}_{L_i\alpha^*} - \overline{b}_{L_i\beta})$$

\end{enumerate}

\end{lem}

\begin{proof}

To validate these formulae we must:
\begin{itemize}

\item perform mutation, $\mu_{\alpha^*}$, at $\alpha^*$ and $\tilde{\alpha}^*$ in $Q_{\overline{T}^{*},\overline{\mathbf{L}}}$ to obtain $Q_{\overline{\mu_{\alpha}(T)}^{*},\overline{\mathbf{L}}}$

\item perform the sequence of mutations $\mu_{\alpha^*}\circ\mu_{\beta}\circ\mu_{\alpha^*}$ on $Q_{\overline{T}^{*},\overline{\mathbf{L}}}$ to obtain $Q_{\overline{\mu_{\beta}(T)}^{*},\overline{\mathbf{L}}}$.

\end{itemize}

By Lemma \ref{lamination restriction}, we can divide our task into four natural cases:

\begin{enumerate}

\item $b_{L\alpha},b_{L\tilde{\alpha}} \geq 0$, \hspace{3mm} $b_{L\alpha^*} \geq b_{\tilde{L}\beta}$, \hspace{3mm} $b_{\tilde{L}\alpha^*} \geq b_{L\beta}$

\item $b_{L\alpha},b_{L\tilde{\alpha}} \geq 0$, \hspace{3mm} $b_{L\alpha^*} \leq b_{\tilde{L}\beta}$, \hspace{3mm} $b_{\tilde{L}\alpha^*} \leq b_{L\beta}$

\item $b_{L\alpha},b_{L\tilde{\alpha}} \leq 0$, \hspace{3mm} $b_{L\alpha^*} \geq b_{\tilde{L}\beta}$, \hspace{3mm} $b_{\tilde{L}\alpha^*} \geq b_{L\beta}$

\item $b_{L\alpha},b_{L\tilde{\alpha}} \leq 0$, \hspace{3mm} $b_{L\alpha^*} \leq b_{\tilde{L}\beta}$, \hspace{3mm} $b_{\tilde{L}\alpha^*} \leq b_{L\beta}$

\end{enumerate}

Using the matrix mutation exchange relation, $b'_{kj} = \sgn(b_{ij})[b_{ki}b_{ij}]_{+}$, it is easily verified that, in each of the four cases, the resulting coefficients agree with the claimed formulae.

\end{proof}

\begin{figure}[H]
\begin{center}
\includegraphics[width=13cm]{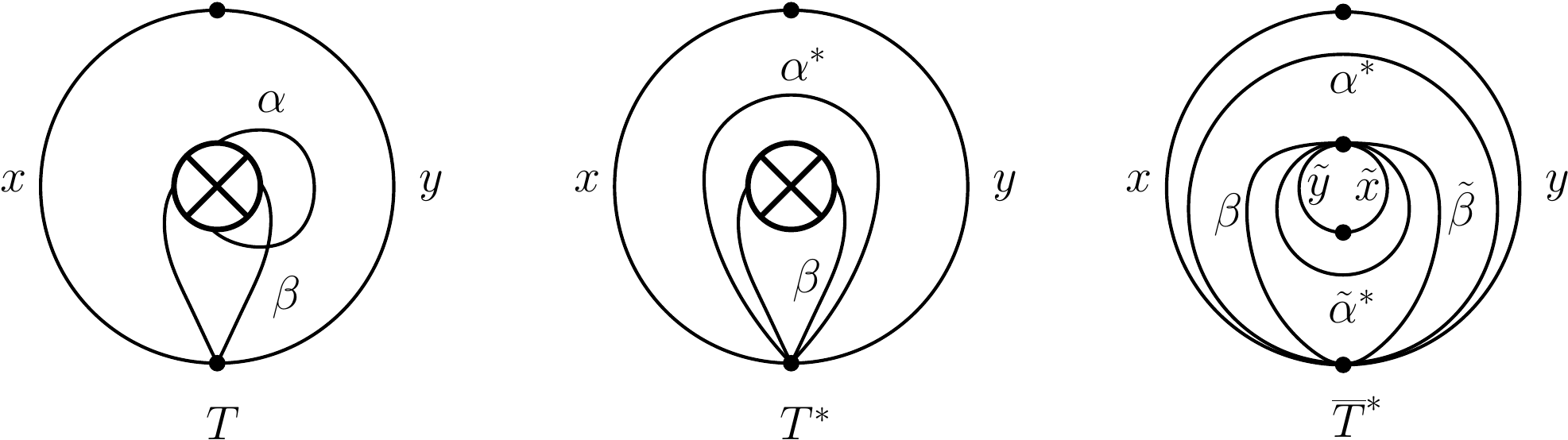}
\caption{A quasi-triangulation $T$ together with its associated traditional triangulation $T^*$ and lift $\overline{T}^*$.}
\label{tandtraditionallift}
\end{center}
\end{figure}

Having discovered how exchange polynomials of any quasi-triangulation change via the consideration of quivers, we are now ready to show that (under certain conditions) LP mutation is also describing how these polynomials change under flips.

\begin{thm}
\label{partial main theorem}
Suppose $(S,M,\mathbf{L})$ is a laminated bordered surface such that for each quasi-triangulation $T$, $(\mathbf{x}, \mathbf{F}_T)$ is a valid seed, and $F_{\gamma_i} \neq F_{\gamma_j}$ for any quasi-arcs $\gamma_i$ and $\gamma_j$ in $T$ $(i \neq j)$. Then LP mutation amongst seeds corresponds to flipping quasi-arcs. Specifically, for any seed, $(\mathbf{x}, \mathbf{F}_T)$, and quasi-arc $\gamma \in T$ we have that $\mu_{\gamma}(\mathbf{x}, \mathbf{F}_T) = (\mu_{\gamma}(\mathbf{x}), \mathbf{F}_{\mu_{\gamma}(T)})$. \newline Here $\mu_{\gamma}(\mathbf{x}):= \mathbf{x}\setminus\{x_{\mathbf{L}^*}(\gamma)\}\cup \{x_{\mathbf{L}^*}(\gamma')\}$ where $\gamma'$ is the flip of $\gamma$ with respect to $T$.

\end{thm}

\begin{proof}

In Proposition \ref{flip} we classified the type of flip regions of quasi-triangulations $T$ - these are shown in Figure \ref{flipregions}. It is crucial to note that the sides of these flip regions may not be arcs (or boundary segments) in $T$, but rather an arc bounding $M_1$ or a punctured digon. In which case this arc is representing the two quasi-arcs it bounds. Propositions \ref{traditional quiver mutation}, \ref{correct polys} and Lemma \ref{normalisation assumption} tells us that LP mutation describes how the exchange polynomials of arcs change when flipping amongst triangulations. It remains to check this is the case when mutating to, and amongst, quasi-triangulations containing one-sided closed curves. \newline \indent In this proof we shall only consider flip regions whose sides are arcs in $T$, however, our arguments can easily be extended to the case when they are not. Note that when we perform a flip only the exchange polynomials of the interior and boundary quasi-arcs of the flip region can change. Therefore, for each flip, we just need to show that LP mutation describes how these polynomials change (as LP mutation will also leave all other exchange polynomials unchanged).

\begin{figure}[H]
\begin{center}
\includegraphics[width=13.8cm]{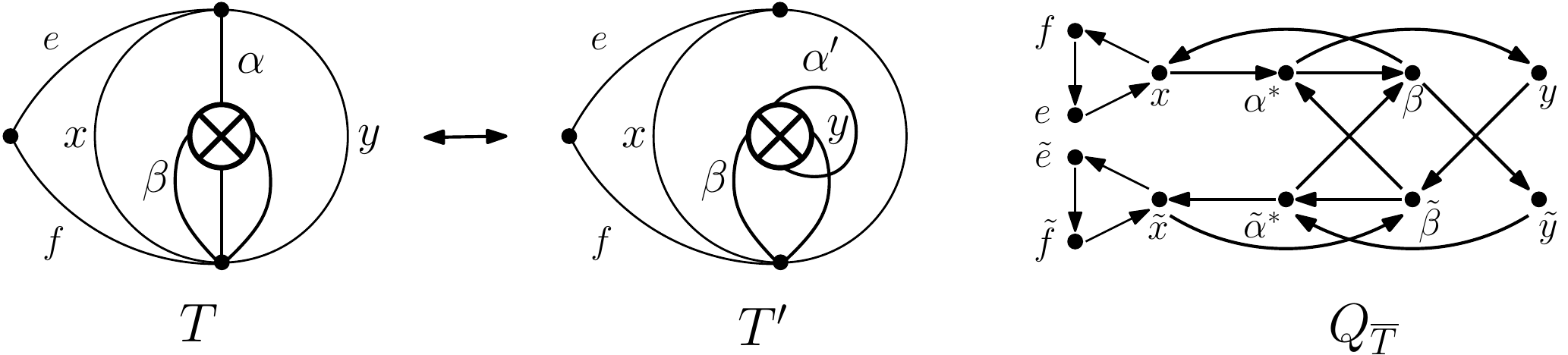}
\caption{A triangulation $T$, in which $\alpha$ flips to a one-sided closed curve $\alpha'$, and the corresponding quiver $Q_{\overline{T}}$.}
\label{fliptoonesided}
\end{center}
\end{figure}

\underline{Case 1}: The flip of an arc $\alpha$ to a one-sided closed curve.

Suppose an arc $\alpha$ in a triangulation of $(S,M)$ flips to a one-sided closed curve $\alpha'$. Let $x$ and $y$ be the boundaries of this region, and $\beta$ the other interior arc. Without loss of generality it suffices to show that LP mutation describes how $F_x$ and $F_{\beta}$ change under this flip. Furthermore we may assume that $x$ is not a boundary segment of $(S,M)$, as otherwise it has no exchange polynomial, and there would be nothing to check. We will therefore have the local picture shown in Figure \ref{fliptoonesided}, moreover, for our chosen labelling of $Q_{\overline{T}}$, by Proposition \ref{correct polys}, we will get:

$$ F_{\alpha} = \displaystyle \big( \prod_{\overline{b}_{L_i\alpha} >0} L_i^{\overline{b}_{L_i\alpha}} \big)x +  \big(\prod_{\overline{b}_{L_i\alpha} <0} L_i^{-\overline{b}_{L_i\alpha}} \big)y $$

$$ F_{\beta} = \displaystyle \big( \prod_{\overline{b}_{L_i\beta} >0} L_i^{\overline{b}_{L_i\beta}} \big)\alpha^2 +  \big(\prod_{\overline{b}_{L_i\beta} <0} L_i^{-\overline{b}_{L_i\beta}} \big)xy $$

$$ F_{x} = \displaystyle \big( \prod_{\overline{b}_{L_i x} >0} L_i^{\overline{b}_{L_i x}} \big)\beta e +  \big(\prod_{\overline{b}_{L_i x} <0} L_i^{-\overline{b}_{L_i x}} \big)\alpha f .$$

Let us consider the quasi-triangulation $T' := \mu_{\alpha}(T)$ and denote the coefficients in $Q_{\overline{T}'}$ by $\overline{b}'_{ij}$. By Propositions \ref{correct polys} and \ref{correct polys quasi}, we are required to show that LP mutation changes $F_{\beta}$ and $F_x$ to the following polynomials: 

$$ F'_{\beta} = \displaystyle \big( \prod_{\overline{b}_{L_i\beta} >0} L_i^{\overline{b}_{L_i\beta}} \big)F_{\alpha}^2 +  \big(\prod_{\overline{b}_{L_i\beta} <0} L_i^{-\overline{b}_{L_i\beta}} \big)xy\alpha'^2 $$

$$ F'_{x} = \displaystyle \big( \prod_{\overline{b}'_{L_i x} >0} L_i^{\overline{b}'_{L_i x}} \big)\beta e\alpha' +  \big(\prod_{\overline{b}'_{L_i x} <0} L_i^{-\overline{b}'_{L_i x}} \big)yf .$$

Since $F_{\alpha} \rvert_{\beta \leftarrow 0 } = F_{\alpha}$ then,

$$ G_{\beta} = F_{\beta} \rvert_{\alpha \leftarrow \frac{F_{\alpha}}{\alpha'}} = \frac{\displaystyle \big( \prod_{\overline{b}_{L_i\beta} >0} L_i^{\overline{b}_{L_i\beta}} \big)F_{\alpha}^2 +  \big(\prod_{\overline{b}_{L_i\beta} <0} L_i^{-\overline{b}_{L_i\beta}} \big)xy\alpha'^2}{\alpha'^2} $$

Hence $MG_{\beta} = F_{\beta}'$, as required. \newline Since $\hat{F}_{\alpha} \rvert_{x \leftarrow 0 } = \big(\prod_{\overline{b}_{L_i\alpha} <0} L_i^{-\overline{b}_{L_i\alpha}} \big)y$ we obtain:

$$ G_{x} = F_{x} \rvert_{\alpha \leftarrow \frac{\hat{F}_{\alpha} \rvert_{x \leftarrow 0 }}{\alpha'}} = \frac{\displaystyle \big( \prod_{\overline{b}_{L_i x} >0} L_i^{\overline{b}_{L_i x}} \big)\beta e\alpha' +  \big(\prod_{\overline{b}_{L_i x} <0} L_i^{-\overline{b}_{L_i x}} \big)\big(\prod_{\overline{b}_{L_i\alpha} <0} L_i^{-\overline{b}_{L_i\alpha}} \big) yf}{\alpha'}$$

From here we see that the exponent of $L_i$ in $MG_x$ will be $| \overline{b}_{L_i x} - \max(0,-\overline{b}_{L_i\alpha})|$. Moreover, $L_i$ will appear in the left or right monomial of $MG_x$ respective of whether $ \overline{b}_{L_i x} - \max(0,-\overline{b}_{L_i\alpha})$ is positive or negative. Lemma \ref{exceptional relations} tells us that $\overline{b}_{L_i x} = \overline{b}'_{L_i x} + \max(0, \overline{b}_{L_i\alpha})$ and $\overline{b}_{L_i \alpha} = - \overline{b}'_{L_i \alpha}$. So $\overline{b}'_{L_i\beta} = \overline{b}_{L_i x} - \max(0,-\overline{b}_{L_i\alpha})$ and LP mutation indeed describes how the exchange polynomials change for this flip. \newline

\begin{figure}[H]
\begin{center}
\includegraphics[width=13.8cm]{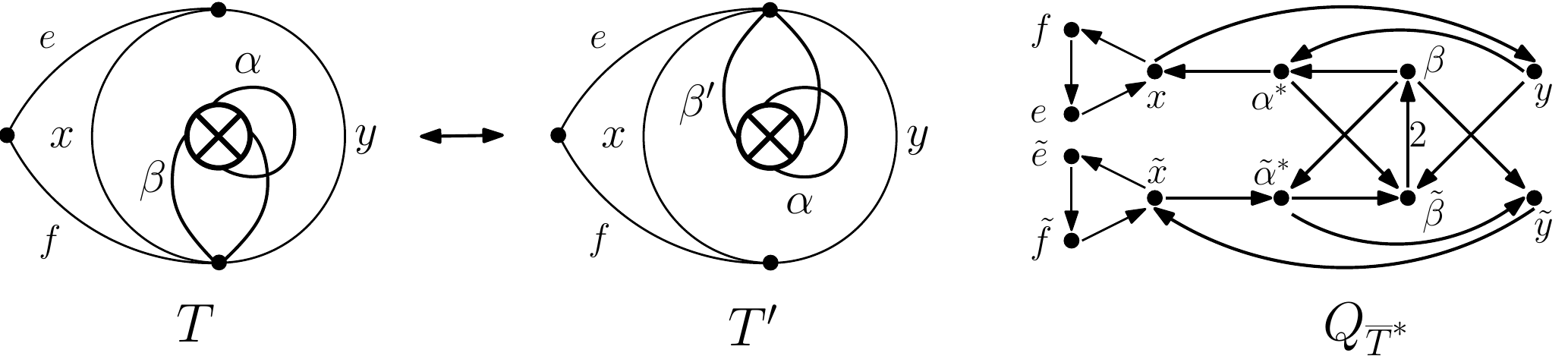}
\caption{A triangulation $T$, in which $\beta$ intersects a one-sided closed curve $\alpha$, and the corresponding quiver $Q_{\overline{T}^*}$.}
\label{flipintersectingonesided}
\end{center}
\end{figure}

\underline{Case 2}: The flip of an arc $\beta$ intersecting a one-sided closed curve.

Suppose an arc $\beta$ in a quasi-triangulation of $(S,M)$ intersects a one-sided closed curve $\alpha$. Let $x$ and $y$ be the boundary segments of the flip region, and denote by $\beta'$ the arc $\beta$ flips to. As before, it is enough to show that LP mutation describes how $F_{\alpha}$ and $F_x$ change under this flip, and we may assume $x$ is not a boundary segment of $(S,M)$. We therefore arrive at the quasi-triangulation $T$, and quiver $Q_{\overline{T}^{*}}$ shown in Figure \ref{flipintersectingonesided}.  For our chosen labelling of $Q_{\overline{T}^{*}}$, by Propositions \ref{correct polys} and \ref{correct polys quasi}, we obtain:

$$ F_{\alpha} = \displaystyle \big( \prod_{\overline{b}_{L_i\alpha^*} >0} L_i^{\overline{b}_{L_i\alpha^*}} \big)y +  \big(\prod_{\overline{b}_{L_i\alpha^*} <0} L_i^{-\overline{b}_{L_i\alpha^*}} \big)x $$

$$F_{\beta} = \displaystyle \big( \prod_{\overline{b}'_{L_i\beta} >0} L_i^{\overline{b}'_{L_i\beta}} \big)F_{\alpha}^2 +  \big(\prod_{\overline{b}'_{L_i\beta} <0} L_i^{-\overline{b}'_{L_i\beta}} \big)xy\alpha^2.$$

$$ F_{x} = \displaystyle \big( \prod_{\overline{b}_{L_i x} >0} L_i^{\overline{b}_{L_i x}} \big)\alpha\beta e +  \big(\prod_{\overline{b}_{L_i x} <0} L_i^{-\overline{b}_{L_i x}} \big)fy $$

Note that here we represent the coefficients of $Q_{\overline{\mu_{\alpha}(T)}} = \mu_{\alpha^*}\circ \mu_{\tilde{\alpha}^*}(Q_{\overline{T}^*})$ by $\overline{b}'_{L_i\beta}$. Furthermore, for $T' := \mu_{\beta}(T)$, if we denote the coefficients of $Q_{\overline{T'}^*}$ by $\overline{b}'''_{ij}$ then, by Propositions \ref{correct polys} and \ref{correct polys quasi}, we are required to show that LP mutation changes $F_{\alpha}$ and $F_x$ to the following polynomials:

$$ F_{\alpha}' = \displaystyle \big( \prod_{\overline{b}'''_{L_i\alpha^*} >0} L_i^{\overline{b}'''_{L_i\alpha^*}} \big)y +  \big(\prod_{\overline{b}'''_{L_i\alpha^*} <0} L_i^{-\overline{b}'''_{L_i\alpha^*}} \big)x $$

$$ F_{x}' = \displaystyle \big( \prod_{\overline{b}'''_{L_i x} >0} L_i^{\overline{b}'''_{L_i x}} \big)ey +  \big(\prod_{\overline{b}'''_{L_i x} <0} L_i^{-\overline{b}'''_{L_i x}} \big)f\alpha\beta' $$

Since $\beta \notin F_{\alpha}$ then we need $F_{\alpha} = F'_{\alpha}$. This is the case since Lemma \ref{exceptional relations} tells us that $\overline{b}'''_{L_i \alpha^*} = \overline{b}_{L_i \alpha^*}$. It remains to check how $F_{\beta}$ changes under LP mutation.

$$\hat{F}_{\beta} \rvert_{x \leftarrow 0} = \frac{ \displaystyle \big( \prod_{\overline{b}'_{L_i\beta} >0} L_i^{\overline{b}'_{L_i\beta}} \big)\big( \prod_{\overline{b}_{L_i\alpha^*} >0} L_i^{2\overline{b}_{L_i\alpha^*}} \big)y^2}{\alpha^2}$$ $$= \frac{ \displaystyle \Big( \prod_{\max(0,\overline{b}_{L_i\alpha^*} + \overline{b}_{L_i\beta}) + \max(0,\overline{b}_{L_i\alpha^*} - \overline{b}_{L_i\beta}) >0} L_i^{\max(0,\overline{b}_{L_i\alpha^*} + \overline{b}_{L_i\beta}) + \max(0,\overline{b}_{L_i\alpha^*} - \overline{b}_{L_i\beta})} \Big)y^2}{\alpha^2}$$ 

The last equality follows from Lemma \ref{exceptional relations} which tells us that $\overline{b}'_{L_i\beta} = \overline{b}_{L_i\beta} - |\overline{b}_{L_i\alpha^*}|$, and the inequalities of Lemma \ref{lamination restriction}. For convenience, let us define $K_i := \max(0,\overline{b}_{L_i\alpha^*} + \overline{b}_{L_i\beta}) + \max(0,\overline{b}_{L_i\alpha^*} - \overline{b}_{L_i\beta})$. As a consequence we obtain:

$$ G_{x} = F_{x} \rvert_{\beta \leftarrow \frac{\hat{F}_{\beta} \rvert_{x \leftarrow 0 }}{\beta'}} = \frac{ \displaystyle \big( \prod_{\overline{b}_{L_i x} >0} L_i^{\overline{b}_{L_i x}} \big)\Big( \prod_{K_i >0} L_i^{K_i} \Big)ey + \big(\prod_{\overline{b}_{L_i x} <0} L_i^{-\overline{b}_{L_i x}} \big)f\alpha\beta'}{\frac{\alpha\beta'}{y}}$$

From here we see $L_i$ will have exponent $|\overline{b}_{L_i x} + K_i|$ in $MG_x$. Moreover, $L_i$ will appear in the left or right monomial of $MG_x$ respective of whether $\overline{b}_{L_i x} + K_i$ is positive or negative. From Lemma \ref{exceptional relations} we saw $\overline{b}'''_{L_i\beta} = \overline{b}_{L_i x} + K_i$, so LP mutation does indeed describe how the exchange polynomials change for this flip. \newline

For the cases when the boundaries of flip regions are not all arcs, analogous calculations show that LP mutation still describes how the exchange polynomials change.
\end{proof}

\subsection{Principal laminations}

Theorem \ref{partial main theorem} asserts that for a laminated quasi-cluster algebra $\mathcal{A}(S,M,\mathbf{L})$, if the exchange polynomials in each seed are irreducible and distinct then flips coincide with LP mutations.
Therefore, to establish an LP structure on a bordered surface $(S,M)$ we must concoct a multi-lamination which guarantees irreducibility and uniqueness of the exchange polynomials in any quasi-triangulation. This multi-lamination will follow the flavour of principal coefficients, but to introduce it we will first need some preliminaries.

\begin{defn}

An arc $\gamma$ of $(S,M)$ is called \textit{\textbf{orientable}} if it has an orientable neighbourhood. Otherwise $\gamma$ is said to be \textit{\textbf{non-orientable}}. Examples of these types of curves are given in Figure \ref{orientablearcs}.

\end{defn}

\begin{rmk}
Note that an arc $\gamma$ will be non-orientable \textit{if and only if} it has a unique endpoint and crosses through an odd number of cross-caps.
\end{rmk}

\begin{figure}[H]
\begin{center}
\includegraphics[width=10cm]{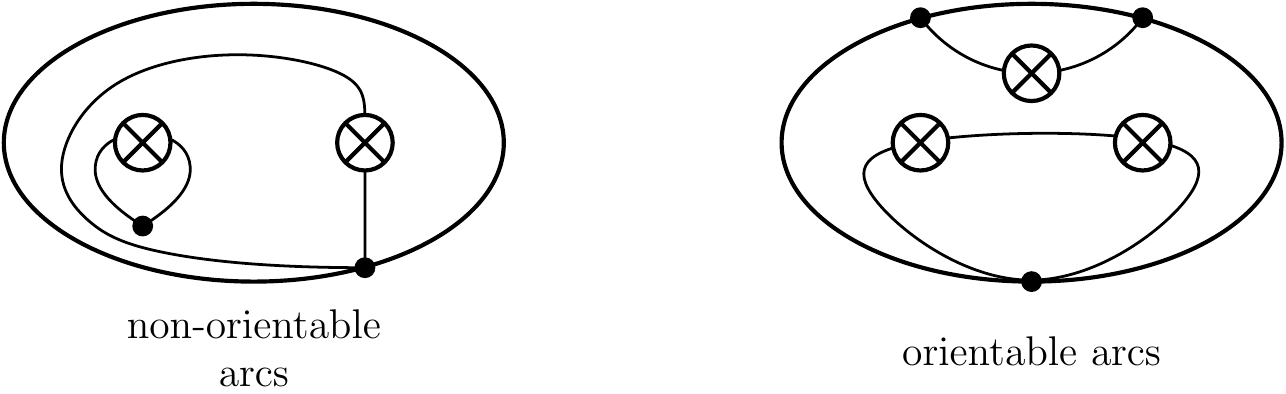}
\caption{Examples of orientable and non-orientable arcs.}
\label{orientablearcs}
\end{center}
\end{figure}

\begin{defn}

Define the parity, $p(\gamma)$, of an arc $\gamma$ of $(S,M)$ to be $+1$ or $-1$ respective of whether $\gamma$ passes through an even or odd number of cross-caps.

\end{defn}

\begin{lem}
\label{parity}
Let $\Delta = (\alpha, \beta, \gamma)$ be a triangle in $(S,M)$. Then $$p(\alpha)p(\beta)p(\gamma) = 1.$$

\end{lem}

\begin{proof}

Consider a slightly smaller triangle $\Delta' = (\alpha', \beta', \gamma')$ lying in the interior of $\Delta$. The parities of the arcs in $\Delta'$ remain the same as those in $\Delta$ since they are only slight perturbations of their original versions. In particular, $$p(\alpha)p(\beta)p(\gamma) = p(\alpha')p(\beta')p(\gamma').$$

Moreover, although two sides of $\Delta$ may be glued together in $(S,M)$, all arcs in $\Delta'$ will be distinct. As a consequence, the neighbourhood of $\Delta'$ in $(S,M)$ is orientable, implying that $p(\alpha')p(\beta')p(\gamma') = 1$.

\end{proof}

\begin{lem}
\label{incident arc}
Let $T$ be a triangulation of $(S,M)$ and $\gamma$ be a non-orientable arc in $T$ with unique endpoint $m \in M$. Then there exists an orientable arc $\beta \in T$ with (at least one) endpoint $m$.

\end{lem}

\begin{proof}

The non-orientable arc $\gamma$ belongs to a triangle $\Delta$ in $T^{\circ}$ (see Figure \ref{idealtriangulation} regarding definition of $T^{\circ}$). Lemma \ref{parity} ensures there exists an orientable arc $\beta \in \Delta$. If $\beta \in T$ then we are done, so consider the other possibility of $\beta$ enclosing two arcs $\beta_1,\beta_2 \in T$ which only differ by a tagging at one puncture. $\beta_1$ and $\beta_2$ have distinct endpoints which ensures they are orientable, and this concludes the proof.

\end{proof}

\begin{defn}[Principal lamination]
\label{principal lamination def}
Let $T$ be a triangulation of $(S,M)$. We define a \textit{\textbf{principal lamination}}, $\mathbf{L}_T := \{L_{\gamma} | \gamma \in T\}$ to be a multi-lamination satisfying the bullet points below. In Figures \ref{perturbingarcs} and \ref{nonperturbingarcs2} we provide examples of the types of laminations which constitute $\mathbf{L}_T$.
\begin{itemize}

\item If $\gamma$ is an orientable plain arc then $L_{\gamma}$ is taken to be the lamination that runs along $\gamma$ in a small neighbourhood thereof, which consistently spirals around the endpoints of $\gamma$ both clockwise (or anti-clockwise). For endpoints of $\gamma$ which are not punctures $L_{\gamma}$ cannot `spiral', instead we mean it turns clockwise (resp. anti-clockwise) at the marked point, and ends when it reaches the boundary.

\item If $\gamma$ is an orientable arc with some notched endpoints, $L_{\gamma}$ is defined as above, except now, at notched endpoints the direction of spinning is reversed.

\item If $\gamma$ is a non-orientable arc with (unique) endpoint $m$ situated on the boundary, then consider two points on the boundary, $m_1$ and $m_2$, that lie either side of $m$ in a small neighbourhood thereof. $L_{\gamma}$ is the lamination with endpoints $m_1$ and $m_2$, which runs along a small neighbourhood of $\gamma$ - note that $L_{\gamma}$ will intersect $\gamma$ once.

\item If $\gamma$ is a non-orientable arc situated at a puncture $p$ then, by Lemma \ref{incident arc}, $\gamma$ has an incident orientable arc $\beta \in T$. In Figure \ref{nonperturbingarcs2} we provide an illustration of what $L_{\gamma}$ looks like in this case, but to be precise it is the lamination which:
\begin{itemize}

\item spirals out of the puncture $p$, then

\item runs parallel to $\gamma$ after intersecting $\gamma$ and then $\beta$, (after $L_{\gamma}$ intersects $\gamma$ it is allowed to intersect both endpoints of $\beta$ before running parallel to $\gamma$, it is just not allowed to intersect one endpoint of $\gamma$ and then run parallel to the other endpoint of $\gamma$, without intersecting $\beta$ inbetween) then,

\item intersects $\gamma$ and continues to run parallel to it, then

\item when it arrives back to a neighbourhood of $p$, it should run against the orientation of spiralling at $p$ until it reaches an endpoint of $\beta$, then

\item runs along a small neighbourhood of $\beta$, and at the endpoint spirals depending on the type of tagging of $\beta$: if the endpoints of $\beta$ receive the same tagging then the direction of spiralling should be consistent with the spiralling of $\gamma$ at $p$; if the endpoints of $\beta$ receive different taggings then the direction of spiralling should oppose the spiralling of $\gamma$ at $p$.

\end{itemize}

\end{itemize}

\end{defn}

\begin{figure}[H]
\begin{center}
\includegraphics[width=8cm]{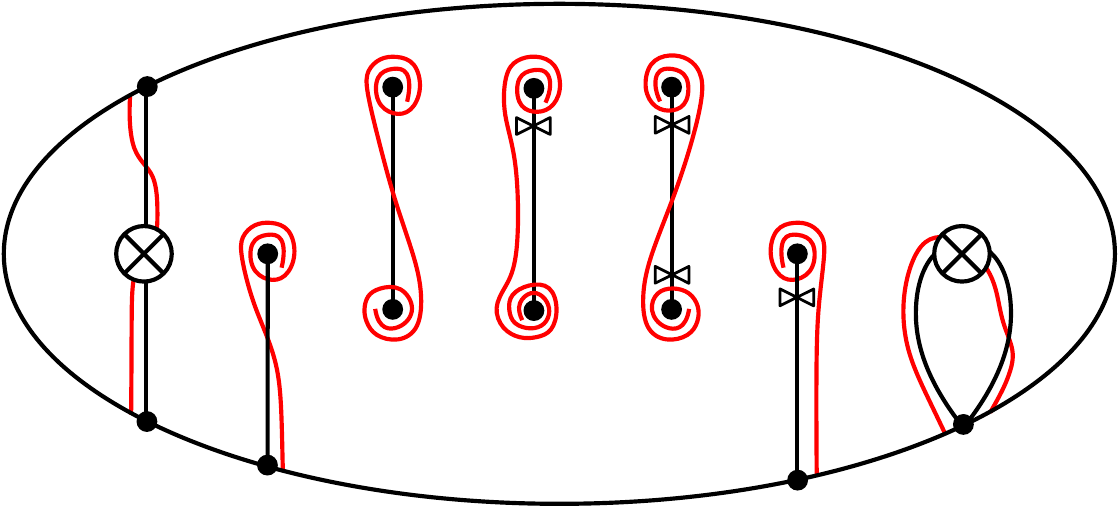}
\caption{Examples of various types of laminations occurring in a principal lamination.}
\label{perturbingarcs}
\end{center}
\end{figure}

\begin{figure}[H]
\begin{center}
\includegraphics[width=12cm]{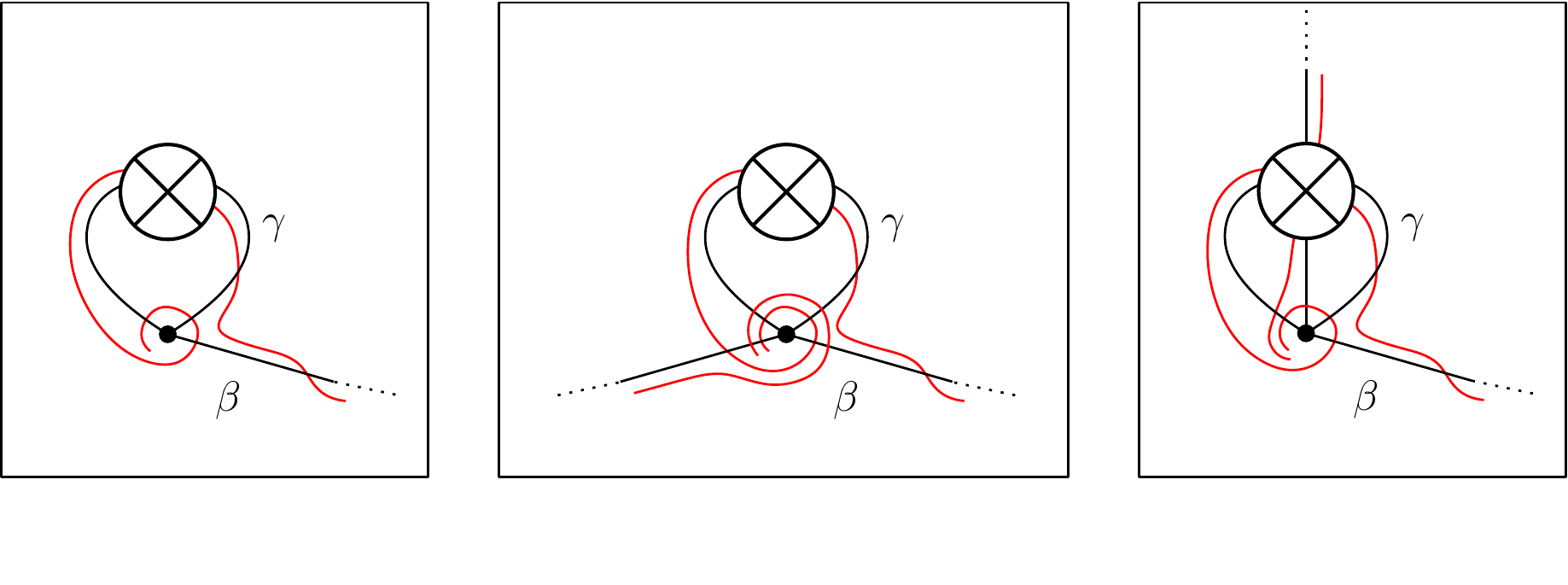}
\caption{The types of laminations, $L_{\gamma}$, when $\gamma$ is a non-orientable arc situated at a puncture, and $\beta$ is the chosen incident orientable arc.}
\label{nonperturbingarcs2}
\end{center}
\end{figure}

\begin{rmk}
In Definition \ref{principal lamination def}, $\beta$ is required to be orientable so that if it has unique endpoint $p$, when $L_{\gamma}$ runs parallel to it (and intersects it) and arrives back at $p$, it is able to spiral back around $p$. If $\beta$ was non-orientable it would not be able to spiral back around $p$ without self-intersections. Likewise, the conditions that $L_{\gamma}$ must:

\begin{enumerate}[label=(\alph*)]

\item intersect $\gamma$ and then $\beta$ before running parallel to $L_{\gamma}$;

\item (when moving against the orientation of spiralling at $p$) run parallel to $\beta$ as soon as it meets an endpoint $\beta$; 

\end{enumerate}

\noindent are required, otherwise self-intersections would occur if $\beta$ also has unique endpoint $p$.
\end{rmk}

\begin{rmk}

In general, for a given arc $\gamma \in T$ the choice of $L_{\gamma}$ is not unique. We are just concerned about choosing some lamination $L_{\gamma}$ that satisfies the rules demanded in Definition \ref{principal lamination def}. The motivation behind the definition is to ensure the shortened exchange matrix associated to $T$ and $L_{\gamma}$ is of full rank.

\end{rmk}

\begin{prop}

Let $(S,M)$ be a bordered surface and $T$ a triangulation. If $\mathbf{L}_T$ is a principal lamination of $T$ then the shortened exchange matrix $\overline{B}$ is of full rank, and the gcd of each column is $1$.

\end{prop}

\begin{proof}

By the bigon criterion [Proposition 1.7, \cite{farb2011primer}], since $L_i$ does not form a bigon with any arc in $T$, then it is in minimal position (regarding intersections). Therefore, $L_i$ will add weight $\pm 1$ to $\gamma_i$, and to $\beta_i$ if $\gamma_i$ is non-orientable and situated at a puncture. Moreover, $L_i$ will not add weight to any other arcs. Consequently, after rearranging columns of $\overline{B}$, the bottom $n \times n$ submatrix will be upper triangular with $\pm 1$ entries on its diagonals. This confirms $\overline{B}$ has full rank, and that the gcd of each column is $1$.

\end{proof}

In Proposition \ref{rank preserving} we will show that the rank of the shortened exchange matrix is preserved under mutation. For this we need the following technical Lemma \ref{mutationequation}.

\begin{lem}
\label{mutationequation}
Let $i \in \{1,\ldots, n\}$ be a vertex in an anti-symmetric quiver $Q$, and suppose there is no path $k \rightarrow i \rightarrow \tilde{k}$ for any vertex $k$ in $Q$.

If $\overline{b}_{ij} \geq 0 \geq \overline{b}_{ji}$ or $\overline{b}_{ij} \leq 0 \leq \overline{b}_{ji}$ for every $j \in \{1,\ldots, n\}$, then, for any $j,k \in \{1,\ldots, n\}\setminus \{i\}$, mutation at $i$ and $\tilde{i}$ in $Q$ gives:

\begin{equation}
\label{shortened equation}
\overline{b}_{jk}' = \overline{b}_{jk} + \max(0, -\overline{b}_{ji})\overline{b}_{ik} + \max(0, \overline{b}_{ik})\overline{b}_{ji}.
\end{equation}

\end{lem}

\begin{proof}

Without loss of generality we may assume $\overline{b}_{ji}:= b_{ji} + b_{\tilde{j}i} \geq 0$; otherwise we could just reverse all the arrows, as this has no effect on the truth of the proposed equation (\ref{shortened equation}). Since there are no paths $k \rightarrow i \rightarrow \tilde{k}$ for any $k$, then $b_{ji}, b_{\tilde{j}i} \geq 0$. Moreover, by using this path condition again, and anti-symmetry, we realise either 

\begin{enumerate}[label=(\alph*)]

\item $b_{ik} \leq 0$ and $b_{\tilde{i}k} \geq 0$, or

\item $b_{ik} \geq 0$ and $b_{\tilde{i}k} \leq 0$.

\end{enumerate}

The respective local configurations of the quiver for cases (a) and (b) are shown in Figure \ref{twoquivers}. In case (a) we see that mutating the quiver (at $i$ and $\tilde{i}$) adds no new arrows between $j,k,\tilde{j},\tilde{k}$, so 

\begin{equation}
\label{mutationa}
\overline{b}_{jk}' = \overline{b}_{jk}
\end{equation}.

In case (b) we see mutation produces 
\begin{equation}
\label{mutationb}
\overline{b}_{jk}' :=  b_{jk}' + b_{\tilde{j}k}' = (b_{jk} + b_{ji}b_{ik} - b_{j\tilde{i}}b_{\tilde{i}k}) + (b_{\tilde{j}k} + b_{\tilde{j}i}b_{ik} - b_{\tilde{j}\tilde{i}}b_{\tilde{i}k}) = \overline{b}_{jk} + \overline{b}_{ji}\overline{b}_{ik}.
\end{equation}

With this knowledge at hand, we will now check agreement of the proposed equation \ref{shortened equation} and quiver mutation. We shall achieve this by splitting the task into two parts, depending on whether

\begin{enumerate}

\item $\sgn(\overline{b}_{ji}) = \sgn(\overline{b}_{ik}) = \pm 1$, or

\item $\sgn(-\overline{b}_{ji}) = \sgn(\overline{b}_{ik})$, or at least one of $\overline{b}_{ji}, \overline{b}_{ik}$ is zero.

\end{enumerate}

For case 1(a), $\overline{b}_{ki}:= b_{ki} + b_{\tilde{k}i} > 0$. Since $b_{ik} = 1 > 0$, this contradicts the conditions of the lemma, meaning case 1(a) is redundant. \newline \indent For cases 1(b) and 2(a) our proposed equation \ref{shortened equation} produces exactly what is written in (\ref{mutationb}) and (\ref{mutationa}), respectively. \newline
\indent For case 2(b) we have $\overline{b}_{ki}:= b_{ki} + b_{\tilde{k}i} \leq 0$. However, since $0 \leq \sgn(-\overline{b}_{ji}) = \sgn(\overline{b}_{ik})$, then $\overline{b}_{ik} \leq 0$, so by the conditions of the lemma we deduce that $\overline{b}_{ik} = 0$. In turn, this implies $b_{ki} = b_{\tilde{i}k}$, and equation (\ref{mutationb}) reduces to $\overline{b}_{jk}' = \overline{b}_{jk}$. This is exactly what our proposed equation \ref{shortened equation} produces.

\end{proof}

\begin{figure}[H]
\begin{center}
\includegraphics[width=12cm]{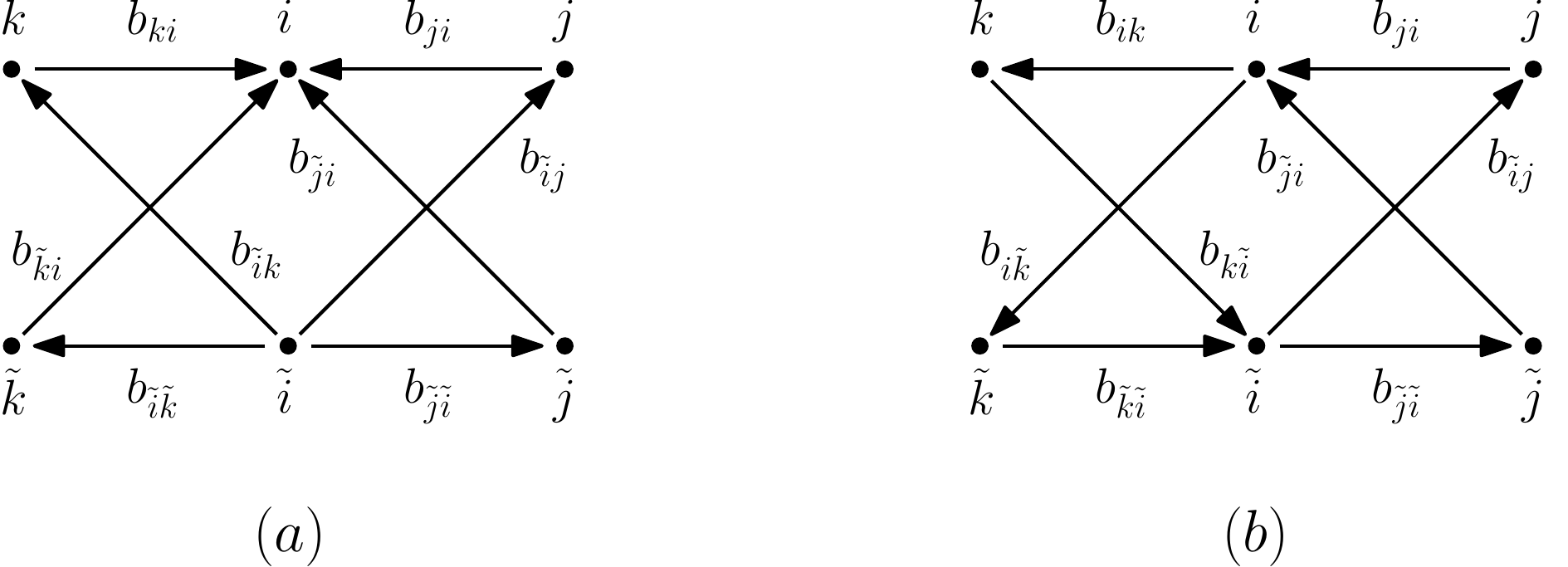}
\caption{The two possible (local) configurations of $Q$ with respect to $i,j,k$. (Here all coefficients present are $\geq 0$.)}
\label{twoquivers}
\end{center}
\end{figure}

\begin{prop}
\label{rank preserving}
Let $i \in \{1,\ldots, n\}$ be a vertex in an anti-symmetric quiver $Q$, and suppose there is no path $k \rightarrow i \rightarrow \tilde{k}$ for any vertex $k$ in $Q$. Then mutation at $i$ and $\tilde{i}$ in $Q$ preserves the rank of the shortened exchange matrix $\overline{B}$.

\end{prop}

\begin{proof}

We would like to apply Lemma \ref{mutationequation} to understand how the coefficients in $\overline{B}$ change under mutation. However, it may be that $\overline{B}$ does not satisfy the conditions demanded in the lemma. Explicitly, there may exist $j \in \{1,\ldots,n\}$ such that $\overline{b}_{ij},\overline{b}_{ji} > 0$ or $\overline{b}_{ij},\overline{b}_{ji} < 0$. However, swapping the labels $j \leftrightarrow \tilde{j}$ in $Q$ gives us a different shortened exchange matrix $\overline{B}^{*}$; for any $k \in \{1,\ldots, n\}$ we get:

$$ \overline{b}^{*}_{jk} = b^{*}_{jk} + b^{*}_{\tilde{j}k} = b_{\tilde{j}k} + b_{jk} = \overline{b}_{jk}$$
$$ \overline{b}^{*}_{kj} = b^{*}_{kj} + b^{*}_{\tilde{k}j} = b_{k\tilde{j}} + b_{\tilde{k}\tilde{j}} = -(b_{\tilde{k}j} + b_{kj}) = -\overline{b}_{kj}$$.

In particular, we obtain $\overline{b}^{*}_{ji} = \overline{b}_{ji} > 0 > -\overline{b}_{ij} = \overline{b}^{*}_{ij}$. This means that we can perform a relabelling, $j \leftrightarrow \tilde{j}$, of the quiver for any $j$ which fails the condition demanded in Lemma \ref{mutationequation}. The new corresponding shortened exchange matrix $\overline{B}^{*}$ will then satisfy the desired conditions. Note that this relabelling process only multiplies the $j^{th}$ column by $-1$, so it preserves the rank of the matrix, and the corresponding exchange polynomials remain unchanged. \newline
Therefore, without loss of generality, we may assume $\overline{B}$ satisfies the conditions of Lemma \ref{mutationequation}. As a consequence of Lemma \ref{mutationequation} the following equations holds.

\[
\begin{array}{c}
\begin{pmatrix}
  \hspace{5mm} \scalebox{0.75}{-1} \hspace{5mm} \vline & \hspace{-2mm} \scalebox{0.75}{0} \cdots \scalebox{0.75}{0} \\ \hline
  \scalebox{0.5}{$ {\max(0,-\overline{b}_{21})}$} \hspace{1mm} \vline & \hspace{-3mm} \raisebox{-15pt}{{\mbox{{$I_{m-1}$}}}} \\[-3.5ex]
  \vdots & \\[-0.5ex]
  \scalebox{0.5}{$ {\max(0,-\overline{b}_{m1})}$} \hspace{0.45mm} \vline &
  \end{pmatrix}
  \overline{B}
 \begin{pmatrix}
  \scalebox{0.75}{-1} \hspace{1mm} \vline & \hspace{-2mm} \scalebox{0.75}{$ {\max(0,\overline{b}_{12})}$} \cdots \scalebox{0.75}{$ {\max(0,\overline{b}_{1n})}$} \\ \hline
  \scalebox{0.75}{0} \hspace{2mm} \vline & \hspace{-3mm} \raisebox{-15pt}{{\mbox{{$I_{n-1}$}}}} \\[-3.5ex]
  \vdots & \\[-0.5ex]
  \scalebox{0.75}{0} \hspace{2mm} \vline &
  \end{pmatrix}
  =
 \begin{pmatrix}
  \hspace{5mm} 0 \hspace{4.1mm} \vline & \hspace{-2mm} -\overline{b}_{12} \hspace{1mm} \cdots \hspace{1mm} -\overline{b}_{1n} \\ \hline
  \hspace{1mm} -\overline{b}_{21} \hspace{2mm} \vline & \hspace{-3mm} \raisebox{-15pt}{{\mbox{{$\big(\overline{b}'_{jk}\big)$}}}}_{j,k\geq 2} \\[-3.5ex]
  \vdots & \\[-0.5ex]
  \hspace{1mm} -\overline{b}_{m1} \hspace{0.85mm} \vline &
  \end{pmatrix}

\end{array}
\]

The matrix on the right is, $\overline{B}'$, the shortened exchange matrix of the mutated quiver $Q' = \mu_1 \circ \mu_{\tilde{1}}(Q)$. Moreover, since the matrices we are multiplying $\overline{B}$ by are invertible, $\overline{B}$ and $\overline{B}'$ have the same rank.

\end{proof}

\begin{lem}
\label{distinctness of polys}
Let $\mathbf{L}$ be a principal lamination of $(S,M)$ and $T$ a quasi-triangulation. Then the exchange polynomials of the quasi-arcs in $T$ are distinct.

\end{lem}

\begin{proof}

Let $T_{\mathbf{L}}$ be the triangulation that $L$ is constructed from. By construction we know that the shortened exchange matrix of $T_{\mathbf{L}}$ will have full rank. As a direct consequence, the exchange polynomials of $T_{\mathbf{L}}$ will be distinct. Moreover, by Lemma \ref{t-mutable path} and Proposition \ref{rank preserving} we know that the shortened exchange matrix of any triangulation will have full rank, in turn implying the desired uniqueness of the exchange polynomials. It remains to show the exchange polynomials of quasi-triangulations containing one-sided closed curves are distinct. Since any quasi-triangulation can be mutated into a triangulation by successive mutations at one-sided closed curves, it suffices to show that mutating to a one-sided closed curve in a quasi-triangulation preserves the uniqueness of the exchange polynomials.\newline
\indent Let $\alpha'$ be an arc in a quasi-triangulation $T$ that flips to a one-sided closed curve $\alpha$. Denote by $\beta$ the unique arc intersecting $\alpha$, and let $x$ and $y$ denote the boundary segments of the flip region. Assuming uniqueness of the exchange polynomials of $T$, we will argue why all exchange polynomials in the quasi-triangulation $T' := \mu_{\alpha'}(T)$ are also distinct. Suppose for now that $x$ and $y$ are not arcs enclosing $M_1$ or a punctured monogon. \newline
\indent Since $F_{\alpha'} = F'_{\alpha}$ and $\alpha' \in F_x, F_y$, then as all other exchange polynomials remain unchanged, we have $F'_{\alpha} \neq F_{\gamma}$ for any quasi-arc $\gamma \in T'\setminus \{\alpha\}$. The only exchange polynomials of $T'$ depending on $\alpha$ are $F'_x$, $F'_y$ and $F'_{\beta}$. Furthermore, when viewed as a polynomial in $\alpha$, $F_{\beta}$ is the only degree $2$ polynomial in $T'$, so our task is reduced to showing that $F'_x \neq F'_y$. \newline
\indent Consider the subquiver $Q$ of $Q_{\overline{T'}^*}$ obtained from looking solely at the flip region in question. We see that there is an arrow between $x$ (or $\tilde{x}$) and $y$ (or $\tilde{y}$) in $Q$, however, for $F'_x$ and $F'_y$ to be equal there cannot be any arrows between them in the global quiver $Q_{\overline{T'}^*}$. It must therefore be the case that the arrow in $Q$ gets cancelled, and our quasi-triangulation must contain the configuration shown in Figure \ref{containeddigon}.

\begin{figure}[H]
\begin{center}
\includegraphics[width=10cm]{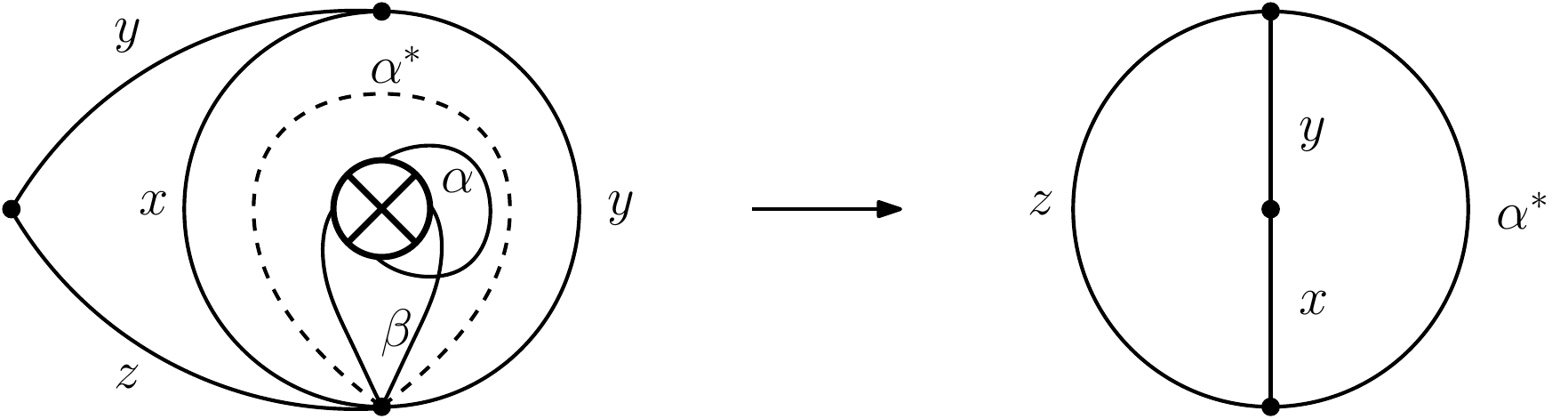}
\caption{On the left we illustrate the (local) configuration of $T'$ required to ensure $\overline{b}_{xy} = \overline{b}_{yx} = 0$ in $Q_{\overline{T}^*}$. The digon embodying this configuration is shown on the right.}
\label{containeddigon}
\end{center}
\end{figure}

However, from here we realise that $x$ and $y$ are the interior arcs of a punctured digon (with boundary segments $z$ and $\alpha^*$). By construction of our principal lamination $\mathbf{L}$ there is a lamination spiralling into every puncture of $(S,M)$, so there will be a lamination spiralling into the puncture of this digon. In turn this implies $F'_x \neq F'_y$. \newline
\indent Finally, we need to turn our attention back to the possibility that $x$ or $y$ is an arc enclosing $M_1$ or a punctured monogon. Without loss of generality, suppose that $x$ is such an arc, and let $x_1$ and $x_2$ be the quasi-arcs it bounds. Analogous to the reasoning employed in our proof thus far, we may deduce that the only possibility for non-uniqueness of the polynomials in $T'$ is if $F_{x_1} = F_{x_2}$. However, if $x$ bounds $M_1$ then $F_{x_1}$ and $F_{x_2}$ have different degrees. If $x$ bounds a punctured monogon then $x_1$ and $x_2$ are the interior arcs of a punctured digon, and since there is a lamination spinning into this puncture we obtain $F_{x_1} \neq F_{x_2}$.

\end{proof}

\subsection{Proof of Main Theorem}

\begin{thm}
\label{maintheorem}
Let $(S,M)$ be an orientable or non-orientable marked surface and $\mathbf{L}$ a principal lamination. Then the LP cluster complex $\Delta_{LP}(S,M,\mathbf{L})$ is isomorphic to the laminated quasi-arc complex $\Delta^{\otimes}(S,M,\mathbf{L})$, and the exchange graph of $\mathcal{A}_{LP}(S,M,\mathbf{L})$ is isomorphic to $E^{\otimes}(S,M,\mathbf{L})$.\newline
More explicitly, if $(S,M)$ is not a once-punctured closed surface, the isomorphisms may be rephrased as follows.
Let $T$ be a quasi-triangulation of $(S,M)$ and $\Sigma_{T}$ its associated LP seed. Then in the LP algebra $\mathcal{A}_{LP}(\Sigma_{T})$ generated by this seed the following correspondence holds:
\begin{align*}
&\hspace{8mm} \mathbf{\mathcal{A}_{LP}(\Sigma_T)} & & &\mathbf{(S,M,\mathbf{L})} \hspace{26mm}&  \\ 
&\textit{Cluster variables} &\longleftrightarrow& &\textit{Laminated lambda lengths of quasi-arcs} & \\
&\hspace{8mm}\textit{Clusters}  &\longleftrightarrow& &\textit{Quasi-triangulations} \hspace{16mm}& \\
&\hspace{4mm} \textit{LP mutation}   &\longleftrightarrow&  &\textit{Flips} \hspace{30.5mm}& \\
\end{align*}
 \end{thm}
\begin{proof}
This is a consequence of Proposition \ref{irreducible}, Theorem \ref{partial main theorem} and Lemma \ref{distinctness of polys}.
\end{proof}

\begin{rmk}

If $(S,M)$ is a closed once-punctured bordered surface then Proposition \ref{flipconnected} tells us that $E^{\otimes}(S,M,\mathbf{L})$ has two connected components. In this case, Theorem \ref{maintheorem} reveals the cluster variables correspond to the laminated lambda lengths of one-sided closed curves and plain arcs (or equivalently notched arcs), and the clusters will therefore correspond to quasi-triangulations consisting of one-sided closed curves and plain arcs (notched arcs).

\end{rmk}

\begin{cor}
\label{maincor}

Let $(S,M)$ be a bordered surface. Then the quasi-cluster algebra $\mathcal{A}(S,M)$ is a specialised LP algebra. 

\end{cor}

\begin{proof}

Let $\mathbf{L}$ be a principal lamination of $(S,M)$. Theorem \ref{maintheorem} yields that $\mathcal{A}(S,M,\mathbf{L})$ is an LP algebra. Specialising the lamination coefficients yields the desired result.

\end{proof}

\nocite{*}

\bibliography{laurent2}
\bibliographystyle{plain}

\Addresses

\end{document}